\documentclass{scrartcl}

\usepackage[
paper=a4paper,
left=2cm,
right=2cm
]{geometry}
\usepackage{import}

\usepackage{amsmath, amssymb, bbold, mathtools}
\usepackage{algorithm,algpseudocode}
\usepackage{multirow}
\usepackage{graphicx}
\usepackage{tcolorbox}
\usepackage[]{hyperref}
\usepackage{bbding}
\usepackage[export]{adjustbox}
\usepackage{subcaption}
\usepackage{enumerate}
\usepackage[textsize=footnotesize]{todonotes}

\usepackage{xcolor}
\usepackage{colortbl}
\usepackage{booktabs}

\usepackage{authblk}
\usepackage{cleveref} 

\usepackage{wasysym}
\usepackage{lmodern} 

\usepackage{natbib}
 \bibpunct[, ]{(}{)}{,}{a}{}{,}

\usepackage{tikz}
\usetikzlibrary{calc}
\usetikzlibrary{shapes}
\usetikzlibrary{spy}

\usepackage{amsthm}
\theoremstyle{plain}


\newcounter{parentnumber}

\newtheorem{theorem}{Theorem}
\newtheorem{lemma}{Lemma}

\newtheorem{proposition}{Proposition}
\newtheorem{assumption}{Assumption}

\usepackage{acronym}
\acrodef{wlog}[WLOG]{without loss of generality}
\acrodef{lsc}[lsc]{lower semi-continuous}

\definecolor{red}{RGB}{163, 31, 52}
\definecolor{gray}{RGB}{194, 192, 191}
\definecolor{blue}{RGB}{59, 89, 152}
\definecolor{green}{RGB}{0, 179, 0}

\newcommand{\norm}[1]{\left\|#1\right\|}

\newcommand{\ceil}[1]{\left\lceil#1\right\rceil}

\newcommand{\abs}[1]{\left|#1\right|}

\newcommand{\E}[2]{{\mathbb E}_{#1} \left[ #2 \right]}

\newcommand{\one}[1]{\mathbb{1}\left\{#1\right\}}
\renewcommand{\tfrac}[2]{{#1}/{#2}}

\newcommand{\set}[2]{\left\{ #1\ : \ #2 \right\}}
\newcommand{\tset}[2]{\{ #1\ : \ #2 \}}
\newcommand{\KL}{{\mathrm{KL}}}

\newcommand{\tpose}{^\top}
\newcommand{\defn}[0]{:=}

\newcommand{\mc}{\mathcal}
\newcommand{\mb}{\mathbb}

\newcommand{\mr}{\mathrm}

\renewcommand{\d}{{\mathrm{d}}}

\renewcommand{\P}{\mb{P}}

\renewcommand{\Re}{\mathrm{R}}

\newcommand{\st}{\mr{s.t.}}

\DeclareMathOperator{\cl}{cl}

\DeclareMathOperator{\typ}{tp}
\DeclareMathOperator{\wc}{wc}

\DeclareMathOperator{\Prob}{Prob}

\DeclareMathOperator{\conv}{conv}

\DeclareMathOperator{\ctr}{ctr}
\DeclareMathOperator{\sign}{sign}

\usepackage{pgfplots}
\usepgfplotslibrary{statistics}

\usepackage{amsthm}
\usepackage[framemethod=TikZ]{mdframed} 
\usepackage{xcolor}

\newmdenv[
  topline=false,
  bottomline=false,
  rightline=false,
  leftline=false,
  backgroundcolor=black!5,
  innerleftmargin=10pt,
  innerrightmargin=10pt,
  innertopmargin=6pt,
  innerbottommargin=6pt
]{exampleframe}

\theoremstyle{definition}
\newtheorem{example}{Example}

\newenvironment{examplebox}
  {\begin{exampleframe}\begin{example}}
  {\end{example}\end{exampleframe}}

\usepackage[appendix=inline]{apxproof}

\usepackage{siunitx}

\usepackage{setspace}
\allowdisplaybreaks

\title{Globalized Adversarial Regret Optimization}
\subtitle{Robust Decisions with Uncalibrated Predictions}
\author[1]{Jannis Kurtz \thanks{j.kurtz@uva.nl}}
\author[2]{Bart P.G.\ van Parys \thanks{bart.van.parys@cwi.nl}}

\affil[1]{University of Amsterdam}
\affil[2]{CWI Amsterdam}

\date{\today}

\begin{document}

\maketitle

\begin{abstract}
Optimization problems routinely depend on uncertain parameters that must be predicted before a decision is made. Classical robust and regret formulations are designed to handle erroneous predictions and can provide statistical error bounds in simple settings. However, when predictions lack rigorous error bounds (as is typical of modern machine learning methods) classical robust models often yield vacuous guarantees, while regret formulations can paradoxically produce decisions that are more optimistic than even a nominal solution.
We introduce Globalized Adversarial Regret Optimization (GARO), a decision framework that controls adversarial regret, defined as the gap between the worst-case cost and the oracle robust cost, uniformly across all possible uncertainty set sizes. By design, GARO delivers absolute or relative performance guarantees against an oracle with full knowledge of the prediction error, without requiring any probabilistic calibration of the uncertainty set.
We show that GARO equipped with a relative rate function generalizes the classical adaptation method of \citet{lepskii1993asymptotically} to downstream decision problems. We derive exact tractable reformulations for problems with affine worst-case cost functions and polyhedral norm uncertainty sets, and provide a discretization and a constraint-generation algorithm with convergence guarantees for general settings. Finally, experiments demonstrate that GARO yields solutions with a more favorable trade-off between worst-case and mean out-of-sample performance, as well as stronger global performance guarantees.
\end{abstract}

\clearpage

\section{Introduction}

Consider an optimization problem
\begin{equation}
  \label{eq:optimization}
  \min_{x\in X} f(x, p^\star)
\end{equation}
which seeks a decision $x\in X$ minimizing a cost function $f:X\times P \to\Re_+\cup\{+\infty\}$ given a problem parameter $p^\star\in P$. 
The field of mathematical optimization, which began in earnest with Dantzig's simplex method in 1947, has since produced powerful tools to find optimal decisions in \Cref{eq:optimization} algorithmically.
However, before any such tools can be applied, practitioners must face the fact that the problem parameters $p^\star$ are typically subject to uncertainty. Before proceeding, a decision maker often first makes a prediction $p_0$ of the problem parameters, and then makes a decision $x_{nom}$ by solving a nominal ``predict-then-optimize'' formulation
\begin{equation}
  \label{eq:nom-optimization}
  \min_{x\in X} f(x, p_0).
\end{equation}

\subsection{From Classical to Wild Predictions}
\label{sec:from-classical-wild}

The use of predictions to facilitate downstream decision-making is an idea at least as old as mathematical optimization itself.
Nevertheless, the nature of predictions has changed remarkably.
We argue that downstream decision tools have largely failed to reflect these changes.

\paragraph{Classical Predictions.}
Predictions have classically been constructed with the help of statistical estimators.
As a case in point, problem \eqref{eq:optimization} may represent a vehicle routing problem \citep{toth2002vehicle} with an unknown cost vector $p^\star=\E{}{p}$ of expected edge costs.
The decision-maker estimates $p^\star$ as the sample average $p_0=\frac{1}{N}\sum_{i=1}^N p_i$ of $N$ historical realizations.
If the decision-maker has access to \textit{curated} data, it is common to assume that each historical observation is an independent sample sharing the same distribution as future random edge costs.
In the context of such curated data, predictions often come with desirable concentration inequalities and associated error margin or confidence bounds. We use the problem of univariate mean estimation as a running example.
After imposing sub-Gaussian tails $\E{}{\exp(\lambda (p_i-p^\star))}\leq \exp(\lambda^2 \sigma^2/2)$ for some given finite $\sigma^2$, it is well known that the classical empirical average predictor discussed earlier satisfies the concentration inequality
\(
  \Prob{}\left[|p_0- p^\star| > \gamma \right]\leq 2 \exp\left( - N \gamma^2 / (2\sigma^2)\right).
\)
Hence, we have
\(
\Prob{}\left[|p_0- p^\star| \leq \gamma_N(\delta) \right]\geq 1-\delta
\)
with $d(p_0, p^\star)=|p_0-p^\star|$ and $\gamma_N(\delta) = \mc O( \sigma \sqrt{\log(2/\delta)/N})$.
We denote a prediction as \textit{classical} when it enjoys 
an error margin guarantee 
\begin{equation}
  \label{def:classical-prediction}
  \Prob{}\left[d(p_0, p^\star) \leq \gamma_N(\delta) \right]\geq 1-\delta
\end{equation}
for any $\delta \in (0, 1)$ and an appropriate\footnote{We will assume throughout that the distance function is lower semicontinuous and satisfies $d(p_0, p_0)=0$ for all $p_0\in P$.} distance function $d:P\times P \to\Re_+\cup\{\infty\}$ for all $N\geq 1$.
Classical robust or regret optimization formulations are designed to exploit such classical predictions and are able to return decisions with rigorous performance guarantees.

\paragraph{Adaptive Predictions.}
It has however long been recognized that practical data is more challenging than what classical predictions can account for. We give two examples.

Many parameters related to natural phenomena or economic indicators are regularly varying, i.e., $\Prob \left[p\geq x\right] = x^{-\nu} L(x)$  and $L(x)$ a slowly varying function.
This means that $\E{}{\exp(\lambda (p_i-p^\star))}=\infty$ for any $\lambda>0$ and $\nu > 1$, and that $\E{}{|p_i-p^\star|^a}=\infty$ for $a>\nu$ while $\E{}{|p_i-p^\star|^a}<\infty$ for $a<\nu$ \citep{nair2022fundamentals}.
In particular, the empirical mean is a fragile estimator in this setting.
Robust mean estimators which satisfy
\begin{equation}
  \label{eq:m-estimate}
  \frac1N \sum_{i=1}^N \psi\left(\lambda_N(p_i-p_0) \right)=0
\end{equation}
for influence function $\psi$ and scaling parameter $\lambda_N$ can be considered instead. The empirical mean corresponds to the case where the influence function is the identity.
Assuming $\left(\E{}{|p_i-p^\star|^a}\right)^{1/a}\leq \sigma_a$ for some finite $\sigma_a<\infty$ and moment order $a>1$, \cite{bhatt2022nearly} proposes a classical estimator $p_0(\sigma_a)$ satisfying \eqref{def:classical-prediction} of minimax optimal radius $\gamma_{N}(\delta, \sigma_a) = \mc O( \sigma_a \left(\log(2/\delta)/N\right)^{(a-1)/a})$.
This is achieved with influence function $\psi(\Delta) = \sign(\Delta) \mc O(\log(|\Delta|^a))$ and scaling parameter $\lambda_N(\sigma_a) = \mc O(\sigma_a^{-1} (\log(2/\delta)/N)^{1/a})$.

Many data sets contain a small fraction $\epsilon \in (0, 1/2)$ of the data points which are corrupted either due to malice or simple error. As was already the case in the context of heavy tail data, the empirical mean is well known to be fragile. This spurred on the development of robust alternatives. \citet{huber1964robust, huber1981robust} advances a prediction $p_0(\epsilon)$ based on Equation (\ref{eq:m-estimate}) with clipped influence function $\psi(x) = \max(\min(x, 1), -1)$ which achieves \eqref{def:classical-prediction} for minimax optimal radius $\gamma_N(\delta, \epsilon) = \mc O(\sigma (\sqrt{\epsilon} + \sqrt{\log(2/\delta)/N} ))$ with $\lambda_N=\sigma^{-1}\sqrt{\epsilon+\log(2/\delta)/N}$ assuming only that $\E{}{|p_i-p^\star|^2}\leq \sigma^2<\infty$.

In the previously discussed examples the estimate $p_0(\kappa)$ satisfies
\begin{equation}
  \label{eq:classical-predictor-family}
  \Prob_{\kappa}\left[d(p_0(\kappa), p^\star) \leq \gamma_N(\delta, \kappa) \right]\geq 1- \delta
\end{equation}
where $\Prob_\kappa$ denotes the probability law of the data when the structural parameter takes value $\kappa$.
This guarantee depends critically on an auxiliary parameter $\kappa \in K \subseteq [\kappa_{\min}, \kappa_{\max}]$ ($\sigma_a$ and $\epsilon$ in the case of heavy-tail or corrupt data) which is unknown in practice. As is the case in the discussed examples, we assume here that larger auxiliary parameters correspond to harder prediction problems, resulting in a nondecreasing margin $\kappa \mapsto \gamma_N(\delta, \kappa)$ for any $\delta$, $N$. A conservative prediction $p_0 = p_0(\kappa_{\max})$ corresponding to the worst-case parameter choice could be considered.
However, adaptive guarantees can be obtained without resorting to the worst-case parameter.
For instance, let $\kappa_j = \kappa_{\min} \beta^j$ for $j=0,\dots, J$ with $J=\ceil{\log_{\beta}(\tfrac{\kappa_{\max}}{\kappa_{\min}})}$ and $\beta > 1$.
Define the adaptive estimate $p_0 = p_0(\kappa^\star)$ where $\kappa^\star$ is the smallest grid point $\kappa_j$ for which $\set{p\in P}{d(p_0(\kappa_i), p)\leq \gamma_N(\delta, \kappa_i) ~~\forall i \geq j}$ is nonempty.
\citet{lepskii1993asymptotically} shows the adaptive guarantee
\begin{equation}
  \label{eq:Lepskis-guarantee}
  \Prob_{\kappa}\left[d(p_0, p^\star) \leq 2 \gamma_N(\delta/(J+1) , \beta\kappa) \right]\geq 1- \delta
\end{equation}
indicating performance on par with the predictor $p_0(\kappa)$ which does have access to the unknown parameter when the distance function satisfies the triangle inequality. More generally, we denote a prediction $p_0$ which does not depend on $\kappa$ as \textit{adaptive} when it enjoys 
an error margin guarantee 
\begin{equation}
  \label{def:adaptive-prediction}
  \Prob_{\kappa}\left[d(p_0, p^\star) \leq \gamma_N(\delta, \kappa) \right]\geq 1-\delta.
\end{equation}
for all $N\geq 1$.
Although the adaptive estimates are not as conservative, their guarantees are still relative to an unknown parameter which prevents the use of downstream robust or regret optimization formulations with absolute performance guarantees.

\paragraph{Wild Predictions.}
Increasingly, modern machine learning methods (including deep neural networks \citep{lecun2015deep}), gradient boosted trees \citep{chen2016xgboost}, and large-scale foundation models \citep{bommasani2021opportunities}) produce predictions with remarkable empirical accuracy yet come with no rigorous performance guarantees \citep{goodfellow2015explaining}.
We denote a prediction as \textit{wild} when it is frequently empirically accurate ($\gamma^\star\defn d(p_0, p^\star) \leq \gamma_{\typ}$) but rigorous claims only establish $\gamma^\star \leq \gamma_{\wc}$ where the ratio $\gamma_{\wc}/\gamma_{\typ}\gg 1$ is very large.
The formal worst-case bound $\gamma_{\wc}$ may stem, for instance, from a Lipschitz argument \citep{bartlett2017spectrally}, a generalization bound \citep{bartlett2002rademacher}, or a smoothness assumption \citep{tsybakov2009nonparametric}, but the resulting uncertainty set $\set{p}{d(p_0,p)\leq \gamma_{\wc}}$ is so large as to render any robust formulation practically useless.
Wild predictors are fundamentally different from adaptive ones: while adaptive predictions preserve a formal relationship between the radius $\gamma_N(\delta,\kappa)$ and a structural parameter $\kappa$, wild predictions offer no such structure.
The downstream decision-making challenge is therefore not to exploit a formal guarantee, but rather to design decisions that perform well when the prediction is accurate $(\gamma^\star\leq \gamma_{\typ})$ while degrading gracefully relative to an appropriate benchmark when it is not $(\gamma^\star\leq \gamma_{\wc})$.

\subsection{Literature Review}

The robust optimization literature lists several strategies to help reduce the conservatism inherent to worst-case formulations.
\citet{bertsimas2004price} show that budget uncertainty sets can to some extent help manage the conservatism.
\citet{fischetti2009light} relax hard robust constraints by allowing small, penalized violations, obtaining solutions that are nearly robust yet substantially less conservative.
A comprehensive treatment of these ideas appears in \citet{ben2009robust}. Another model to reduce the conservatism of the classical robust model is robust regret or min-max regret \citep{savage1951theory,kouvelis2013robust,aissi2009min} where the difference to the optimal value in each of the uncertain scenarios is minimized. All previously mentioned models have in common that the considered performance criteria can only be guaranteed on a (usually bounded) uncertainty set. Globalized robust methods attempt to control the performance of the calculated solution also outside of the uncertainty set \cite{long2023robust,ben2017globalized}.
Our framework combines and complements the latter lines of work by measuring conservatism through a globalized adversarial regret operator and by providing a guarantee that degrades smoothly with the prediction error level.

A parallel development in computer science is the growing literature on \emph{algorithms with predictions} \citep{mitzenmacher2022algorithms}.
These algorithms receive a (possibly erroneous) prediction and seek competitive ratios that interpolate between the optimal ratio when the prediction is exact and a worst-case ratio when it is adversarial.
Foundational contributions include \citet{lykouris2021competitive} for caching and \citet{purohit2018improving} for ski rental and scheduling.
Our framework shares the same high-level aspiration: exploit accurate predictions while remaining robust to inaccurate ones.
The key distinction is that algorithms with predictions operate in an online or combinatorial setting with discrete competitive ratios, whereas our framework targets continuous optimization under uncertainty with a smooth, tunable performance guarantee.

The bicriteria perspective on robustness \citep{chassein2016bicriteria} is also related: it treats nominal and worst-case cost as two competing objectives and traces the Pareto frontier between them.
Our framework imposes a richer structure by controlling adversarial regret uniformly across all uncertainty levels $\gamma \in \Gamma$, rather than balancing only two extreme scenarios.
A detailed comparison with robust, regret, and satisficing formulations is deferred to Section \ref{sec:uncert-optim}.

\subsection{Contributions}

The main contributions of this paper are as follows.

\begin{itemize}
  
\item We show that standard robust formulations yield vacuous guarantees under wild predictions, that classical regret formulations can produce decisions more optimistic than even a nominal formulation, and that satisficing formulations are under typical convexity assumptions merely robust formulations in disguise.
  
\item We introduce \emph{adversarial regret} --- the gap between a decision's worst-case cost and the oracle robust cost --- as a novel performance target, and propose Globalized Adversarial Regret Optimization (\ref{eq:glob-rob-regret}), which controls this adversarial regret uniformly over all prediction error levels.

\item GARO can deliver either an \emph{absolute} \eqref{eq:gror-guarantee-absolute} or a \emph{relative} \eqref{eq:gror-guarantee-relative} performance guarantee against the oracle robust benchmark. With a relative rate function, GARO generalizes the classical adaptation method of \citet{lepskii1993asymptotically} from statistical estimation to downstream decision problems (Theorem~\ref{thm:generalized-lepski}).

\item We derive exact tractable reformulations for affine and linear-polyhedral settings, and provide a discretization scheme and a constraint generation algorithm with convergence guarantees for general problems.

\end{itemize}

\section{Uncertain Optimization}
\label{sec:uncert-optim}

Decision-making in the face of uncertainty has been an active topic of research over the last fifty years, but the resulting tools were, by and large, developed with classical predictions in mind. We discuss here the benefits and downsides of classical uncertain optimization formulations in the face of wild predictions.

\subsection{Robust Optimization}

It has long been understood that blindly trusting the prediction $p_0$ as a substitute for $p^\star$ in the nominal formulation \eqref{eq:nom-optimization} does result in rather gullible decisions.
\citet{ben2000robust} were perhaps the first to point out that even if $d(p_0, p^\star)$ is very small, the actual cost $f(x_{nom}, p^\star)$ can far exceed the anticipated cost $f(x_{nom}, p_0)$; a phenomenon known as the ``curse of optimization''.
For a given uncertainty set $P_{\gamma_0}$ robust optimization 
\begin{equation}
  \label{eq:rob-optimization}
  \min_{x\in X} \max_{p\in P_{\gamma_0}} f(x, p)
\end{equation}
has since been widely adopted to avoid this curse through a worst-case perspective.
A notable property that explains the popularity of the worst-case perspective is that it preserves the convexity of the nominal problem.
That is, if the objective function $f$ is convex-concave and the constraint and parameter sets $X$ and $P_{\gamma_0}$ are convex, then the robust formulation (\ref{eq:rob-optimization}) reduces to a convex-concave saddle point problem which, under rather mild technical assumptions, admits tractable algorithms \citep{ben2002robust}.

The robust set $P_{\gamma_0}$ characterizes all scenarios against which the decision maker wishes to anticipate.
In the context of classical predictions, a natural choice is simply to consider a set
\begin{equation}
  \label{eq:ambiguity-set}
  P_{\gamma_0} = \set{p\in P}{d(p_0, p)\leq \gamma_0}
\end{equation}
where the probability of the event $p^\star\in P_{\gamma_0}$ occurring is controlled explicitly by the classical guarantee in Equation (\ref{def:classical-prediction}) through the selection of an appropriate robustness parameter $\gamma_0$.
Robust optimization relies on the implication
\[
  p^\star\in P_{\gamma_0} \implies f(x, p^\star) \leq v_{wc}(x, \gamma_0) \defn \max_{p\in P_{\gamma_0}} f(x, p)
\]
and the observation that in the face of the curse of optimization it typically holds that
\[
v_{wc}(x_{rob}(\gamma_0), \gamma_0) \defn \min_{x\in X} v_{wc}(x, \gamma_0) \ll v_{wc}(x_{nom}, \gamma_0),
\]
where $X_{rob}(\gamma_0)$ denotes the set of optimal solutions of \eqref{eq:rob-optimization} with uncertainty set $P_{\gamma_0}$ and $x_{rob}(\gamma_0)\in X_{rob}(\gamma_0)$ is a selection. A hurdle to the practical adoption of robust formulations outside settings where extreme caution is indeed warranted is their conservatism \citep{roos2020reducing}.
By anticipating the worst-case scenario, performance may be poor in case the worst-case fails to realize.
Although its associated worst-case performance guarantee
\begin{equation}
  \label{eq:rob_perf_guarantee}
  p^\star\in P_{\gamma_0} \implies f(x_{rob}(\gamma_0), p^\star) \leq v_{wc}^\star(\gamma_0) \defn \min_{x\in X} v_{wc}(x,\gamma_0)
\end{equation}
is best possible by construction, its strength nevertheless hinges on how large the optimal worst-case cost $v_{wc}^\star(\gamma_0)$ precisely is.
When the minimal worst-case cost is very large the performance guarantee \eqref{eq:rob_perf_guarantee}  becomes vacuous.
The following example makes this discussion concrete.

\begin{examplebox}\label{ex:hub_location}
We want to find a hub location $\mu$ to serve customers distributed at locations following $\mb P^\star$ within a support set $\Xi$ where $\norm{\xi-\mu}^2$ corresponds to the frustration experienced by a customer at location $\xi$ when served from a hub at location $\mu$, i.e.,
\begin{equation}
  \label{eq:gen-weber}
  \min_{\mu} ~ \int \norm{\xi-\mu}^2 \d \P^\star(\xi). 
\end{equation}
Here the customers (used perhaps to instant gratification) experience superlinear frustration as service time increases.
For our facility problem to be well posed it is necessary that $\P^\star \in \mc P\defn \tset{\P}{\int \norm{\xi}^2\d \P(\xi)<\infty}.$
We remark that the unique optimal location in (\ref{eq:gen-weber}) is simply the mean $\mu^\star =\mu(\mb P^\star) \defn \int \xi \d \mb P^\star(\xi)$ which achieves cost equal to the variance $v^\star =v(\mb P^\star)\defn \int (\xi-\mu(\mb P^\star))^2 \d \mb P^\star(\xi)$. 

The distribution of future customers $\mb P^\star$ may however not be known precisely to the decision-maker.
Instead, consider a robust facility location formulation
\begin{equation}
  \label{eq:robust-weber-problem}
  \min_{\mu} \sup_{\mb P\in \mc P_{\gamma_0}} \int \norm{\xi-\mu}^2 \d \P(\xi).
\end{equation}
with respect to the Wasserstein ball
\(
  \mc P_{\gamma_0} \defn \set{\P\in\mc P}{W(\P, \P_0)\leq \gamma_0}
\)
for a judiciously chosen robustness parameter $\gamma_0$.
Denote by $\mu_{rob}(\gamma_0)$ the minimizer in the robust facility location problem (\ref{eq:robust-weber-problem}).
Depending on the value of $\gamma_0$, two extreme cases can be identified.
Clearly, if $\gamma_0=0$ we recover the unique nominal location $\mu_{rob}(\gamma_0) = \mu_{nom} = \int \xi \d\P_0(\xi)$ whereas at the other extreme for $\gamma_0\geq 2\norm{\Xi}$, the set $\mc P_{\gamma_0}$ contains all distributions supported on $\Xi$ and hence the robust optimal location moves to the unique center\footnotemark{} of $\Xi$. In the latter case,
\begin{equation}
  \label{eq:rob-to-center}
  \mu_{rob}(\gamma_0) = \arg \min_{\mu} \max_{\xi \in \Xi} \norm{\xi-\mu}^2 = \ctr(\Xi).
\end{equation}
In Figure \ref{fig:weber-instance} the path $\set{\mu_{rob}(\gamma)}{\gamma\geq 0}$ between the discussed extreme cases is depicted as a gray line.

The robust facility location formulation promises if $\mb P^\star\in \mc P_{\gamma_0}$ that 
\begin{equation}
  \label{eq:rob-perf-guarantee}
  \int \norm{\xi-\mu_{rob}(\gamma_0)}^2 \d \P^\star(\xi) \leq v^\star_{wc}(\gamma_0) \defn \min_{\mu} v_{wc}(\mu, \gamma_0) = \sup_{\P\in \mc P_{\gamma_0}} \int\norm{\xi-\mu_{rob}(\gamma_0)}^2 \d \P(\xi).
\end{equation}
The strength of a worst-case guarantee hinges on whether the minimum worst-case cost is reasonably small or at least bounded. Clearly, for large $\gamma_0\geq 2\norm{\Xi}$ we get $v^\star_{wc}(\gamma_0)=\norm{\Xi}^2$ and the worst-case guarantee \eqref{eq:rob-perf-guarantee} is vacuous when $\Xi$ is large. Unfortunately, from Lemma \ref{lemma:weber-lower-bound} in the Appendix it follows that also for small value of $\gamma_0>0$ we have
\(
  v^\star_{wc}(\gamma_0) \geq \min\left(\frac{\gamma_0}{2}, \norm{\Xi}\right) \norm{\Xi}
\)
limiting the robust formulation in \eqref{eq:robust-weber-problem} to bounded sets $\Xi$.
The precise problem instance and Julia code to reproduce Figures~\ref{fig:weber-instance}--\ref{fig:weber-costs} are available at \url{https://gitlab.com/vanparys/garo-example}.
\end{examplebox}
\footnotetext{For any bounded set $S$ we write $\norm{S} \defn \min_{s\in\conv(S)}\max_{s'\in S} \norm{s'-s}$ for its circumradius and $\ctr(S)$ for its corresponding unique circumcenter.}

\subsection{Regret Optimization}

\citet{savage1951theory} introduced the absolute regret formulation
\begin{equation}
  \label{eq:regret-formulation}
  \begin{array}{r@{\,}l}
    \min_{x\in X} & \max_{p\in P_{\gamma_0}} R(x, p)
  \end{array}
\end{equation}
in response to the apparent conservatism of the worst-case perspective advocated by \citet{wald1945statistical} where the regret is defined absolutely as $R(x, p)\defn f(x, p)-\min_{x'\in X} f(x', p)$. This approach was later studied for discrete problem structures in \cite{kouvelis2013robust}. Intuitively, the regret formulation \eqref{eq:regret-formulation} prefers decisions whose performance in hindsight can only be improved by a minimal amount.

By construction a regret optimal decision enjoys the performance bound
\[
  p^\star\in P_{\gamma_0} \implies f(x_{reg}(\gamma_0), p^\star) \leq \min_{x'\in X}f(x', p^\star) + R^\star
\]
where $x_{reg}(\gamma_0)$ and $R^\star$ denote a minimizer and minimum in Equation \eqref{eq:regret-formulation}, respectively.
As performance is measured here relative to an oracle decision with access to the unknown parameter $p^\star$, regret optimal decisions come with performance guarantees beyond merely the worst-case.
In some sense, whereas robust formulations emphasize worst-case performance, regret formulations care about performance in all scenarios equally.
By insisting on a minimal regret uniformly over all scenarios, regret formulations are inherently optimistic as the regret of a decision $x$ balances equally its worst-case performance $\max_{p\in P_{\gamma_0}} f(x, p)$ and its best-case performance $\min_{p\in P_{\gamma_0}} f(x, p)$.

Indeed, the optimal regret decision may offer even less protection against a worst-case scenario than a simple nominal decision, i.e.,
\begin{align*}
  v_{wc}(x_{nom}, \gamma) < & v_{wc}(x_{reg}(\gamma_0), \gamma)
\end{align*}
for any $\gamma\geq 0$.
A regret optimal decision may indeed forgo potential cost reductions in the predicted and worst-case scenarios to chase performance improvements in best-case scenarios, which seems to run counter to prudent decision making.
Furthermore, outside a few toy problems \citep{agarwal2022minimax, perakis2008regret}, notably the newsvendor problem, regret formulations become computationally intractable already when the nominal problem (\ref{eq:nom-optimization}) is a simple linear optimization problem \citep{averbakh2005complexity}.
We illustrate the optimism of regret formulations in our facility location example.

\begin{examplebox}
Denote the regret of hub location $\mu$ for a distribution of customers $\P$ as
\begin{align*}
  R(\mu, \P) \defn & \int \norm{\xi-\mu}^2 \d \P(\xi) - \min_{\mu'} \int \norm{\xi-\mu'}^2 \d \P(\xi)
                           = \norm{\textstyle\int \xi \d \P(\xi)-\mu}^2
\end{align*}
where the final equality follows from the bias variance decomposition.
It follows that the regret optimal facility location $\mu_{reg}(\gamma_0)= \arg\min_{\mu} \max_{\P\in \mc P_{\gamma_0}} R(\mu, \P)$ is simply the (unique) center of the convex set
\(
M(\gamma_0)\defn \set{\textstyle\int \xi \d \P(\xi)}{\P\in \mc P_{\gamma_0}}
\)
collecting the means of all distributions in $\mc P_{\gamma_0}$.
In the two extreme regimes discussed earlier, the robust and regret optimal formulations return the same optimal hub locations. Trivially, if $\gamma_0=0$ we have $\mu_{reg}(\gamma_0) = \mu_{nom} \defn \int \xi \d \P_0(\xi)$ whereas for a large enough radius  $\gamma_0\geq 2\norm{\Xi}$, the set $\mc P_{\gamma_0}$ contains all distributions supported on $\Xi$ and hence
\[
  \mu_{reg}(\gamma_0) = \ctr(\conv(\Xi)) = \ctr(\Xi)
\]
for any bounded set $\Xi$. In Figure \ref{fig:weber-instance} the path $\set{\mu_{reg}(\gamma)}{\gamma\geq 0}$ between the discussed extreme cases is depicted with a blue line.

From Lemma \ref{lemma:bound-M2} in the Appendix we have $M(\gamma_0)\subseteq B[\mu_{nom}, \gamma_0]$ (where $B[c,r]$ denotes the closed ball with center $c$ and radius $r$) and hence the worst-case regret remains well defined even for unbounded $\Xi$ and remains bounded by $\gamma_0^2$ in stark contrast to the worst-case cost which is unbounded. However, the regret optimal Weber point can be severely optimistic. Kantorovich-Rubinstein duality ensures that if $\Xi$ is unconstrained that $M(\gamma_0)=B[\mu_{nom}, \gamma_0]$ and hence the regret optimal location $\mu_{reg}(\gamma_0) = \mu_{nom}$ coincides with the nominal location offering no additional protection.
In case $\Xi$ is bounded the situation can become even more dire.
We probe the sensitivity of a location $\mu$ by considering the difference
\begin{align*}
  A(\mu, \gamma) \defn & v_{wc}(\mu, \gamma) - v^\star_{wc}(\gamma) \\
  = &\textstyle\sup_{W(\mb P, \mb P_0)\leq \gamma} \int \norm{\xi-\mu}^2 \d \P(\xi) -   \min_{\mu'} \sup_{W(\mb P', \mb P_0)\leq \gamma} \int \norm{\xi-\mu'}^2 \d \P'(\xi).
\end{align*}
That is, the amount by which our worst-case cost is larger when the demand distribution can deviate from the predicted distribution by an amount $\gamma$ compared to an oracle robust solution.
We denote $A$ as the adversarial regret as it measures our loss of performance when the distributional parameter $\mb P$ is the action of an adversarial player.

To make the discussion tangible consider the problem setup depicted in Figure \ref{fig:weber-instance}.
We see in Figure \ref{fig:weber-costs} that the nominal facility location $\mu_{nom} = \mu^\star(\mb P_0)$ achieves the oracle cost at $\gamma=0$ but quickly deteriorates as $\gamma$ increases (essentially depicting the curse of optimization). The robust facility location $\mu_{rob}(\gamma_0)$ achieves the oracle cost at $\gamma_0=1$ but even for smaller values $\gamma\geq 0.25$ it achieves the oracle cost as $\mu_{rob}(\gamma_0) = \ctr(\Xi)$ for any $\gamma_0\geq 0.25$.
Observe however that
\begin{equation}
  \label{eq:optimistic}
  A(\mu_{nom}, \gamma)  < A(\mu_{reg}(\gamma_0),\gamma) \quad \forall \gamma\geq 0
\end{equation}
indicating that the regret optimal facility location manages to suffer both larger nominal as well as larger worst-case costs than the nominal facility location.
\end{examplebox}

\begin{figure}
  \centering
\begin{tikzpicture}[spy using outlines={circle, magnification=4.5, size=5cm, connect spies}]
\begin{axis}[
  width=11cm,    
  height=8cm,    
  view={60}{60},
  axis lines=middle,
  zmin=0,
  zmax=1,
  xmin=0,
  xmax=2.4,
  ymin=0,
  ymax=1.8,
  xtick={0, 1, 2},
  ytick={0, 0.5, 1, 1.5},
  zticklabels=\empty,
  y tick label style={
  xshift=0pt,
  yshift=3pt,  
  anchor=base   
  },
    ylabel style={
    at={(axis description cs:0.83,0.72)}, 
    rotate=0,                             
    anchor=center
  }
  clip=true
  ]

  \foreach \x in {0,0.5,...,2.4} {
    \addplot3[forget plot, black!20, thin] coordinates {(\x,0,0) (\x,1.8,0)};
  }
  \foreach \y in {0,0.5,...,1.8} {
    \addplot3[forget plot, black!20, thin] coordinates {(0,\y,0) (2.4,\y,0)};
  }

  \addplot3  [draw=black, fill=black!15, opacity=0.5] table [x=xi_1,y=xi_2,z expr=0, col sep=comma] {ch-Xi.csv} --cycle;

  \addplot3  [draw=orange, fill=orange, only marks, mark size=0.1pt] table [x=xi_1,y=xi_2,z expr=0, col sep=comma] {mean.csv};

  \addplot3  [draw=black, fill=black, only marks, mark size=0.1pt, mark=*] table [x=xi_1,y=xi_2,z expr=0, col sep=comma] {robust_mean.csv};
  
  \addplot3  [draw=blue, fill=blue, only marks, mark size=0.1pt, mark=*] table [x=xi_1,y=xi_2,z expr=0, col sep=comma] {min-ball-M_r.csv};

  \addplot3  [forget plot, draw=black, fill=black, only marks, mark size=0.1pt] table [x=xi_1,y=xi_2,z expr=0, col sep=comma] {instance.csv};

  \addplot3[forget plot, draw=red,thick,
  quiver={
  u=0,
  v=0,
  w=\thisrow{P}
  },
    -stealth
    ]
    table[
    x=xi_1,
    y=xi_2,
    z=zero,
    col sep=comma
    ] {instance.csv};

    \addplot3  [draw=blue, fill=blue!10, opacity=0.5] table [x=xi_1,y=xi_2,z expr=0, col sep=comma] {M_r.csv} --cycle;

    \addplot3  [draw=black!50, line width=0.2pt] table [x=xi_1,y=xi_2,z expr=0, col sep=comma] {robust_hubs.csv};
    \addplot3  [draw=blue!50, line width=0.2pt] table [x=xi_1,y=xi_2,z expr=0, col sep=comma] {regret_hubs.csv};
  
    \addplot3  [draw=yellow, fill=yellow, only marks, mark size=0.1pt] table [x=xi_1,y=xi_2,z expr=0, col sep=comma] {sat_mean.csv};
    \addplot3  [draw=green, fill=green, only marks, mark size=0.5pt, mark=star] table [x=xi_1,y=xi_2,z expr=0, col sep=comma] {glob_robust_mean.csv};

    \addplot3  [forget plot, draw=blue, densely dotted, thick] table [x=xi_1,y=xi_2,z=P, col sep=comma] {circle2.csv} --cycle;
    
\end{axis}
\spy on (5.58,2.7) in node [left] at (-1,2.5);
\end{tikzpicture}
  \caption{Facility location instance: customer locations $\Xi$ (red marks) with predicted distribution $\mb P_0$ (red arrows). The path of robust hub locations $\{\mu_{rob}(\gamma)\}_{\gamma\geq 0}$ (gray) moves from the nominal location $\mu_{nom}$ (orange) to $\ctr(\Xi)$ (black) as $\gamma_0$ grows. The set $M(\gamma_0)$ of attainable means under $\mc P_{\gamma_0}$ (blue shaded polytope) and the regret-optimal location $\mu_{reg}=\ctr(M(\gamma_0))$ (blue dot) are also shown, together with the satisficing location $\mu_{sat}$ (yellow). The GARO solution $\mu_{garo}$ (green star) is discussed in Section~\ref{sec:robust_decisions_wild_predictions}.}
  \label{fig:weber-instance}
\end{figure}
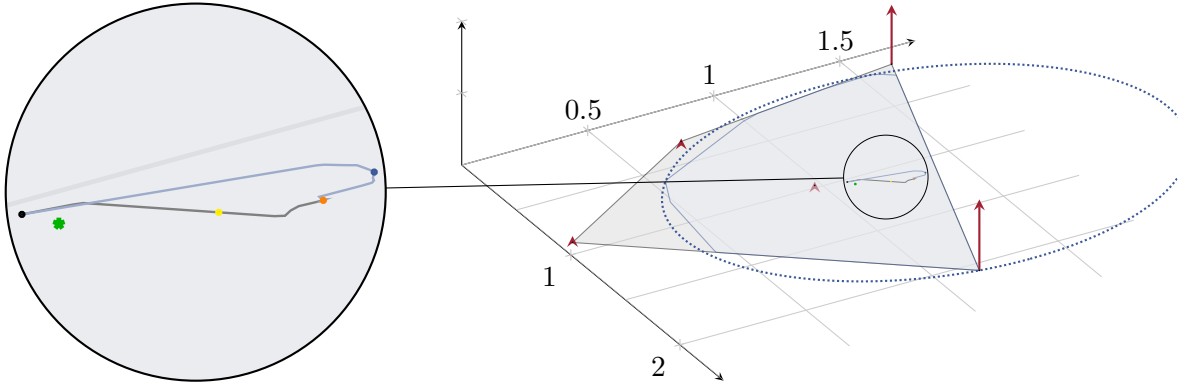

\begin{figure}
    \centering
\begin{tikzpicture}[spy using outlines={circle, magnification=3, size=4cm, connect spies}]
\begin{axis}[
   tick label style={/pgf/number format/fixed},
  width=15cm,    
  height=8cm,    
  axis lines=middle,
  title={Adversarial Regret $A(\cdot, {\gamma})$},
  xmin=0,
  ymin=0,
  ymax=0.07,
  xmax=1.2,
  legend pos=south east,
  legend cell align={left},
  ]

  \addplot [draw=orange] table [x=gamma,y=nom,col sep=comma] {costs.csv};
  \addlegendentry{$\mu_{nom}$}

  \addplot [draw=black] table [x=gamma,y=rob,col sep=comma] {costs.csv};
  \addlegendentry{$\mu_{rob}$}
  
  \addplot [draw=blue] table [x=gamma,y=reg,col sep=comma] {costs.csv};
  \addlegendentry{$\mu_{reg}$}

  \addplot [draw=yellow] table [x=gamma,y=sat,col sep=comma] {costs.csv};
  \addlegendentry{$\mu_{sat}$}
  
  \addplot [draw=green] table [x=gamma,y=grob,col sep=comma] {costs.csv};
  \addlegendentry{$\mu_{garo}$}

  \draw [densely dotted] (axis cs:1, 0) -- (axis cs:1, 0.2);
  \node [above right] at (axis cs:1,0.045) {$\gamma_0$};

    \draw [densely dotted] (axis cs:0.1, 0) -- (axis cs:0.1, 0.2);
  \node [above right] at (axis cs:0.1,0.045) {$\gamma'_0$};

  \addplot [draw=green, densely dotted] table [x=gamma,y=guarantee_garo,col sep=comma] {costs.csv};
  \addlegendentry{$\alpha_{garo}$};
  
\end{axis}
\end{tikzpicture}
    \caption{Adversarial regret $A(\cdot,\gamma)$ as a function of the Wasserstein perturbation level $\gamma$ for the five hub locations shown in Figure~\ref{fig:weber-instance}. The regret-optimal location $\mu_{reg}$ (blue) is so optimistic that its adversarial regret exceeds that of the nominal location $\mu_{nom}$ (orange) for all $\gamma\geq 0$, confirming~\eqref{eq:optimistic}. The robust location $\mu_{rob}$ (black) achieves zero adversarial regret at $\gamma_0$. The GARO location $\mu_{garo}$ (green) maintains uniformly small adversarial regret across all $\gamma$, bounded by the guarantee $\alpha_{garo}$ (dotted green line).}
    \label{fig:weber-costs} 
\end{figure}
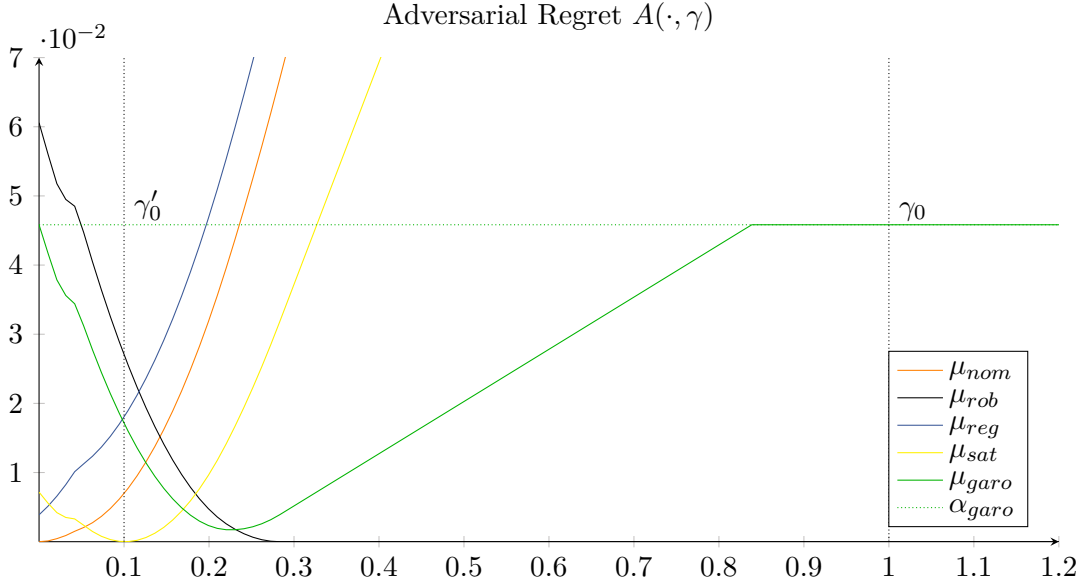

\subsection{Satisficing}
\label{sec:satisficing}

In both the robust and the regret formulation discussed earlier, the set $P_{\gamma_0}$ characterizes those parameters against which the decision maker wishes to anticipate.
When the decision-maker has access to classical predictions satisfying the guarantee (\ref{def:classical-prediction}) the choice suggested in Equation \eqref{eq:ambiguity-set} is both natural and adequate.
With only access to adaptive or wild predictions, the set $P_{\gamma_0}$ must be taken as the set of all possible parameter values $P$, rendering the absolute performance guarantees that robust and regret formulations aspire to simply unattainable.

We now describe a decision formulation introduced by \citet{long2023robust} that enjoys performance guarantees relative to the performance of the prediction itself.
Consider indeed the satisficing formulation
\begin{equation}
  \label{eq:rob-grc}
  \begin{array}{rl}
    \min_{x\in X, \, \alpha\geq 0} & \alpha \\
    \st & f(x, p) \leq f_{0} + \alpha d(p_0, p) \quad \forall p\in P.
  \end{array}
\end{equation}
Here $f_{0}$ denotes a target cost level which the decision-maker deems \textit{satisfactory}.
Let here now $P_\infty=\set{p\in P}{d(p_0, p)<\infty}$.
For the satisficing formulation \eqref{eq:rob-grc} to be feasible and nontrivial we require that
\begin{equation}
  \label{eq:sat-conditions}
  v_{wc}^\star(0) = \min_{x\in X} f(x, p_0) \leq f_0\leq \min_{x\in X} \max_{p\in P_\infty} f(x, p) = v_{wc}^\star(\textstyle\max_{p\in P_\infty} d(p_0, p))
\end{equation}
as we assume here that $p_0\in P$.
We remark that for all $p\not\in P_\infty$ the satisficing constraint in \eqref{eq:rob-grc} is void, since $d(p_0, p)=\infty$ renders the right-hand side infinite.
This distinction arises because $d$ is extended-valued; when $d$ takes only finite values, $P_\infty = P$ and the remark is moot.

Let $\alpha_{sat}$ be the optimal value, $X_{sat}$ the set of optimal solutions, and $x_{sat}\in X_{sat}$ a selection from Equation~(\ref{eq:rob-grc}).
Intuitively, we are looking for decisions whose performance $f(x_{sat}, p^\star)$ meets the target $f_{0}$ if the prediction was indeed correct, i.e., in the event $p_0=p^\star$, and degrades at a minimal rate $\alpha_{sat}$ relative to how wrong the prediction was in hindsight as measured by $d(p_0, p^\star)$.
Satisficing formulations deliver guarantees of the form
\begin{equation}
  \label{eq:satisficing-guarantee}
  f(x_{sat}, p^\star) \leq f_{0} + \alpha_{sat} \gamma^\star
\end{equation}
which allow wild predictors to offer cost guarantees relative to the actual prediction error $d(p_0, p^\star)=\gamma^\star$ rather than a worst-case radius $\gamma_0$ as in robust formulations. 
Additionally, they work well with adaptive predictions which satisfy \eqref{def:adaptive-prediction} and combined with \eqref{eq:satisficing-guarantee} result in adaptive guarantee of the form
\[
  \Prob{}\left[f(x_{sat}, p^\star) \leq f_{0} + \alpha_{sat}  \gamma_N(\delta; \kappa) \right]\geq 1- \delta.
\]
Furthermore, if the objective function $f$ is convex-concave and the constraint and parameter sets $X$ and $P_\infty$ are convex, satisficing formulations are under mild technical conditions amenable to tractable algorithms \citep{long2023robust}.

\begin{examplebox}
In the context of our facility location problem the satisficing formulation in Equation \eqref{eq:rob-grc} reduces to 
\begin{equation}
  \label{eq:weber-satisficing}
  \begin{array}{rl}
    \min_{\mu, \,\alpha\geq 0} & \alpha \\
    \st & \int \norm{\xi-\mu}^2 \d \P(\xi) \leq f_{0} + \alpha W(\mb P_0, \mb P) \quad \forall \mb P \in \mc P
  \end{array}
\end{equation}
and is recognized as the satisficing formulation proposed by \citet{long2023robust}. Here, we choose the satisfaction level $v_{wc}^\star(0) < f_0 \defn 1.04 v_{wc}^\star(0) < v_{wc}^\star(\sup_{\mb P\in \mc P} W(\mb P_0, \mb P))$ to be 4\% suboptimal in case $\mb P^\star=\mb P_0$ occurs.
We mark the satisficing hub location $\mu_{sat}$ with a yellow mark in Figure \ref{fig:weber-instance} and depict its adversarial regret in Figure \ref{fig:weber-costs}. Proposition \ref{prop:sat-eq-rob} implies that $\mu_{sat} = \mu_{rob}(\gamma_0')$ with $\gamma'_0=0.1$.

Unfortunately, a straightforward consequence of the lower bound in Lemma \ref{lemma:weber-lower-bound} in the Appendix is that we need
\(
  \alpha \geq \tfrac{(\norm{\Xi}^2-f_{0})}{(2\norm{\Xi})}
\)
for the satisficing formulation to be feasible.
Hence, for $\Xi$ unbounded, the constraint in the satisficing formulation \eqref{eq:weber-satisficing} is not feasible in stark contrast to a regret formulation.
\end{examplebox}

Hence, satisficing formulations in Equation \eqref{eq:rob-grc} may become infeasible and guarantees in the form \eqref{eq:satisficing-guarantee} can be prohibitively restrictive. Indeed, for the satisficing formulation \eqref{eq:rob-grc} to be feasible it is necessary that the oracle cost is finite and satisfies $v_{wc}^\star(\gamma) = \mc O(\gamma)$. Furthermore, the satisficing guarantee in Equation \eqref{eq:satisficing-guarantee} is relative to an arbitrary satisfaction level $f_0$. In fact, we show in Proposition \ref{prop:sat-eq-rob} in the Appendix that a satisficing solution can sometimes be found as a robust solution for a robustness radius $\gamma_0'$ which is determined by $f_0$.

We remark that the assumption in Proposition \ref{prop:sat-eq-rob} that the cost function $f(x, p)$ is strictly concave in $p$ can be
relaxed, resulting in a slightly weaker conclusion that there always exists a satisficing solution
admitting a robust interpretation; see Appendix~\ref{sec:supporting-sat} for details including
the proof of Proposition~\ref{prop:sat-eq-rob} and supporting examples.

\begin{theorem}
  \label{prop:general-convex}
  Let $P_\infty$ be a compact convex set admitting a continuous strictly convex function $\psi :
P_\infty \to \Re_+$. Let $f(x,p)$ be jointly continuous, convex in $x$, and concave in $p$ for every $x \in X$, with $X$ compact convex. Let $p \mapsto d(p_0,
p)$ be lower semicontinuous and convex. Assume \eqref{eq:rob-grc} is feasible and satisfies condition
\eqref{eq:sat-conditions}. Then there exists an optimal solution $x_{sat}$ of
\eqref{eq:rob-grc} such that
\[
    x_{sat} \in \textstyle\cup_{\gamma \in [0, \gamma_{\max}]} X_{rob}(\gamma),
    \quad \gamma_{\max} := \textstyle\max_{p \in P_\infty} d(p_0, p).
\]
\end{theorem}

The assumption that $P_\infty$ admits a continuous strictly convex function is mild: in finite dimensions it is always satisfied, e.g., by $\psi(p) = \|p\|^2$. In infinite-dimensional settings, such as when $P_\infty$ is a set of probability distributions, the condition requires separate verification.

In summary, each of the three standard formulations exhibits a fundamental weakness when the prediction error level $\gamma$ is unknown. Robust formulations yield vacuous guarantees whenever the uncertainty set must be taken large. Regret formulations are inherently optimistic and can produce decisions offering less protection than the nominal solution. Satisficing formulations, while offering prediction-relative guarantees, are under mild convexity assumptions nothing but robust formulations in disguise.

\section{Robust Decisions with Wild Predictions}\label{sec:robust_decisions_wild_predictions}

When we have access to wild predictions, standard performance guarantees can generally not be achieved.
Indeed, as wild predictions can be arbitrarily bad or even have malicious intent, a robust formulation ignoring the predictions altogether will provide only vacuous worst-case guarantees. In what follows we introduce a decision framework that addresses the shortcomings of the classical robust, regret, and satisficing formulations identified in Section~\ref{sec:uncert-optim}. The main idea is to control adversarial regret uniformly over all prediction error levels.

To make this rigorous, we introduce the univariate family of non-decreasing sets $P_\gamma$, e.g., $P_\gamma = \set{p\in P}{d(p_0, p)\leq \gamma}$, for $\gamma \in  \Gamma$.
We do assume $X$ and $P_\gamma$ to be compact for all $\gamma\in\Gamma$. We note that $p$ may represent either a finite-dimensional parameter vector, recovering standard robust optimization \citep{ben2002robust}, or a probability distribution with a suitable distance $d$ such as the Wasserstein metric, recovering distributionally robust optimization \citep{delage2010distributionally, esfahani2015data}; the running facility location example falls in the latter setting.
Throughout, we write $\min$ and $\max$ when optimizing over $X$ or $P_\gamma$, where attainment is guaranteed by compactness of these sets and continuity of $f$.
We write $\inf$ and $\sup$ when optimizing over infinite-dimensional spaces such as sets of probability distributions, where attainment requires a separate argument. Inspired by the discussion in Section \ref{sec:uncert-optim} we introduce the adversarial regret
\[
  A(x, \gamma) \defn \max_{p\in P_\gamma} f(x, p) - \min_{x' \in X}\max_{p'\in P_\gamma} f(x', p').
\]
That is, the additional cost we may incur if the parameter $p$ is chosen adversarially in $P_\gamma$ compared to an oracle robust solution with access to $\gamma$. We remark that classical regret is larger than the adversarial regret, i.e.,
\(
A(x, \gamma) \leq \max_{p\in P_\gamma} R(x, p)
\)
as the former also takes into account how much performance is left on the table when $p$ would be chosen beneficially.
It was precisely this aspect which rendered regret optimization potentially optimistic as discussed in Section \ref{sec:uncert-optim}.

As mentioned earlier, the defining characteristic of decision-making with wild predictions is that parameter $\gamma$ is fundamentally unknown and hence at a minimum we would like to have that our decision is weakly minimal in
\begin{equation}
  \label{eq:multiobjective}
  \min_{x\in X}~ (A(x, \gamma))_{\gamma\in \Gamma}.
\end{equation}
That is, a decision $x^\star$ for which there is no other decision $y\in X$ so that $A(y,\gamma)< A(x^\star, \gamma)$ for all $\gamma\in \Gamma$.
As its name suggests, though, this is a rather weak requirement.
We observe indeed that the nominal and robust solution as discussed in Section \ref{sec:uncert-optim} are weakly minimal as they minimize $x\mapsto A(x, 0)$ and $x\mapsto A(x, \gamma_0)$, respectively.
We point out that as the facility location example in Figures \ref{fig:weber-instance} and \ref{fig:weber-costs} illustrates, regret optimal solutions nevertheless can fail to be weakly optimal in problem \eqref{eq:multiobjective}.
Finally, under Proposition~\ref{prop:sat-eq-rob}'s strict concavity assumption satisficing solutions are weakly minimal as they minimize $x\mapsto A(x,\gamma_0')$ for some $\gamma_0'\geq 0$; more generally, a satisficing solution is weakly minimal whenever $\Gamma$ is compact and $\gamma\mapsto v_{wc}(x,\gamma)$ is continuous. See Proposition~\ref{prop:sat-weakly-minimal} in Appendix~\ref{sec:supporting-sat} for the formal statement and proof.

Rather than seeking a decision that merely minimizes $A(x,\gamma)$ at a single fixed $\gamma$ --- as robust and satisficing formulations effectively do --- we propose to bound the adversarial regret uniformly over all $\gamma\in\Gamma$, leading to the following class of decision formulations:
\begin{equation}
  \label{eq:glob-rob-regret}\tag{GARO}
  \begin{array}{r@{~}l}
    \displaystyle\min_{x\in X, \, \alpha \ge 0} & \alpha\\[0.5em]
    \st & \displaystyle A(x, \gamma) = \max_{p\in P_\gamma} f(x, p) - \min_{x' \in X}\max_{p'\in P_\gamma} f(x', p') \leq \alpha \phi(\gamma) \quad \forall \gamma\in \Gamma.
  \end{array}
\end{equation}
Here a nondecreasing rate function $\phi:\Gamma\to\Re_+$ controls the amount of adversarial regret our decision suffers as a function of the prediction error $d(p_0, p^\star)$.
We now list the main benefits of our formulation explicitly.

\paragraph{Weak Optimality}

Unlike a regret formulation which can yield decisions which are even more optimistic than a nominal formulation, the proposed formulation is guaranteed by construction to produce weakly minimal solutions in \eqref{eq:multiobjective}.

\paragraph{Performance Guarantees}
Unlike nominal or robust formulations, GARO simultaneously guarantees performance across all error levels $\gamma\in\Gamma$. That is, any minimizer $(x_{garo}, \alpha_{garo})$ of~\ref{eq:glob-rob-regret} satisfies
\begin{equation}
  \label{eq:gror-guarantee}
  v_{wc}(x_{garo}, \gamma) \leq v_{wc}^\star(\gamma) + \alpha_{garo}\phi(\gamma) \quad \forall \gamma\in\Gamma.
\end{equation}
With the constant rate $\phi(\gamma)=1$ this becomes the \emph{absolute} guarantee
\begin{equation}
  \label{eq:gror-guarantee-absolute}
  v_{wc}(x_{garo}, \gamma) - v_{wc}^\star(\gamma) \leq \alpha_{garo} \quad \forall \gamma\in\Gamma,
\end{equation}
bounding the gap to the oracle in absolute terms. With $\phi(\gamma)=v^\star_{wc}(\gamma)$ it becomes the \emph{relative} guarantee
\begin{equation}
  \label{eq:gror-guarantee-relative}
  \frac{v_{wc}(x_{garo}, \gamma) - v_{wc}^\star(\gamma)}{v^\star_{wc}(\gamma)} \leq \alpha_{garo} \quad \forall \gamma\in\Gamma,
\end{equation}
bounding the gap as a fraction of the oracle cost. In both cases $x_{garo}$ tracks the oracle simultaneously at the typical error level $\gamma_{\min}$ and at the worst-case level $\gamma_{\max}$, without committing to any particular level upfront.

\paragraph{Choosing the Rate Function}
The constant rate $\phi(\gamma)=1$ is the natural choice when the oracle cost may be zero or lacks a natural scale; the oracle-cost rate $\phi(\gamma)=v^\star_{wc}(\gamma)$ is preferable when the oracle cost grows with $\gamma$ and an absolute bound would become vacuously large. These are the two principal choices a practitioner faces: in either case a single instance of~\ref{eq:glob-rob-regret} is solved and $\alpha_{garo}$ directly quantifies how closely the decision tracks the oracle benchmark. 

We emphasize that \ref{eq:glob-rob-regret} makes no a priori claim that $\alpha_{garo}$ will be small. Its guarantee is one of ``mere'' optimality; if there is a decision with good absolute or relative performance guarantees then \ref{eq:glob-rob-regret} will find it. A large $\alpha_{garo}$ signals that the problem simply does not admit a decision that uniformly tracks the oracle across $\Gamma$. A concrete example in which $\alpha_{garo}$ is provably small is the adaptive prediction setting of Section~\ref{sec:adaptive-predictions} where the relative rate function yields $\alpha_{garo}\leq 1$, recovering Lepskii's classical factor-of-two guarantee. However, when a standard choice fails to yield a small $\alpha_{garo}$, a problem-dependent rate can salvage the situation as we illustrate in the running example below.

\begin{examplebox}
  \label{example:rate-function}
  With $\phi(\gamma)=1$ the \ref{eq:glob-rob-regret}  solution $\mu_{garo}$ is depicted as a green star in Figure~\ref{fig:weber-instance}. It coincides neither with any robust location on the path $\{\mu_{rob}(\gamma):\gamma\in\Gamma\}$ (gray) nor with any regret-optimal location on the path $\{\mu_{reg}(\gamma):\gamma\in\Gamma\}$ (blue), confirming that \ref{eq:glob-rob-regret}  produces a genuinely distinct decision. Figure~\ref{fig:weber-costs} shows that $\mu_{garo}$ achieves a strictly smaller adversarial regret $A(\mu_{garo},\gamma)$ than every other proposed hub location across the entire range of $\gamma$, bounded uniformly by the guarantee $\alpha_{garo}\approx\num{4.58e-2}$.

  For the bounded instance above in Figure~\ref{fig:weber-instance}, the absolute rate $\phi(\gamma)=1$ yields a reasonably small $\alpha_{garo}$, but as the support $\Xi$ expands the absolute guarantee deteriorates. From Lemma~\ref{lemma:weber-lower-bound} the oracle cost satisfies $v^\star_{wc}(\gamma)\geq \min(\gamma/2,\norm{\Xi})\norm{\Xi}$, so the oracle itself grows with the circumradius $\norm{\Xi}$; maintaining a uniformly small absolute gap $A(\mu,\gamma)\leq\alpha_{garo}$ over all $\gamma\in\Gamma$ therefore forces $\alpha_{garo}$ to grow with $\norm{\Xi}$.

  In the limit $\Xi=\Re_+$ with $\mb P_0=\delta_0$ the absolute rate becomes outright infeasible: direct computation gives $A(\mu,\gamma)=(\mu-\gamma)^2$, which diverges as $\gamma\to\infty$ for any fixed $\mu$. Similarly, as we pointed out in Lemma \ref{lemma:weber-lower-bound}, we have $v^\star_{wc}(\gamma)=\infty$ for any $\gamma>0$ and hence considering a relative rate function is not a sensible option either. Note that the two failure modes are of opposite character: absolute infeasibility ($\alpha_{garo}\to\infty$) signals that no decision uniformly tracks the oracle, while relative triviality ($\alpha_{garo}\to 0$ for every~$\mu$) means all decisions do so vacuously. We can salvage the situation by considering instead $\phi(\gamma)=(1+\gamma)^q$ with $q\geq 2$: the resulting guarantee $A(\mu,\gamma)\leq\alpha_{garo}(1+\gamma)^q$ is meaningful for practically relevant prediction errors (small $\gamma$, where $(1+\gamma)^q\approx 1$) and degrades gracefully as predictions become increasingly wild, a reasonable price to pay when $\Xi$ is unbounded. For $q=2$ the solution is $\alpha_{garo}=1$ and $\mu_{garo}=[0,1]$, while for $q>2$ the unique minimizer satisfies $(\mu_{garo})^2(1+\mu_{garo})^{q-2}=\tfrac{4(q-2)^{q-2}}{q^q}$. 
\end{examplebox}

The choice of rate function $\phi$ encodes a design judgment about how much adversarial regret is acceptable as the prediction error grows.
Together with the distance function $d$ which it shares with the uncertain optimization formulations in Section \ref{sec:uncert-optim}, the rate function shapes the nature of the GARO guarantee \eqref{eq:gror-guarantee} and its appropriate specification should reflect the structure of the application at hand.

\section{Performance Guarantees Across Prediction Regimes}

The three prediction regimes introduced in Section~\ref{sec:from-classical-wild} (classical, adaptive, and wild) call for different analytical tools, but GARO delivers meaningful guarantees in each. The following subsections establish these guarantees in turn, completing the picture begun in Section~\ref{sec:robust_decisions_wild_predictions}.

\subsection{Classical Predictions}

Classical predictions satisfy the guarantee \eqref{def:classical-prediction}. In case the event $d(p_0, p^\star)\leq \gamma_N(\delta)$ occurs it is straightforward to conclude that
\(
  f(x_{rob}(\gamma_N(\delta)), p^\star) \leq \max_{d(p_0, p)\leq \gamma_N(\delta)} f(x_{rob}(\gamma_N(\delta)), p) = v_{wc}^\star(\gamma_N(\delta)).
\)
It follows that 
\begin{equation}
  \label{eq:classical-oos-guarantee}
  \Prob\left[f(x_{rob}(\gamma_N(\delta)), p^\star)\leq v_{wc}^\star(\gamma_N(\delta))\right] \geq 1-\delta.
\end{equation}
Such out-of-sample guarantees have been established throughout the robust optimization literature \citep{kuhn2019wasserstein, vanparys2021data} since the probability of observing a disappointment event in which the actual cost exceeds the anticipated worst-case cost can be controlled explicitly by a judiciously designed robustness parameter $\gamma_0$. More recently, several authors \citep{vanparys2021data, lam2019recovering} have pointed out that the guarantee \eqref{eq:classical-oos-guarantee} does not necessitate the predictor to satisfy Equation \eqref{def:classical-prediction} but rather can be established more directly as well.
In any event, a key limitation of the out-of-sample guarantee \eqref{eq:classical-oos-guarantee} is that it reduces the disappointment risk of its decision to a single probability level. However, reporting multiple probability levels is either encouraged or legally required in many safety‑critical applications. For example, the US Nuclear Regulatory Commission \citep{NRC_PRA} mandates three tiers of probability assessments whereas NASA \citep[Chapter 13]{NASA_SP_2011_3421} promotes a continuous approach through reporting an exceedance curve which in our context means reporting
\begin{equation}
  \label{eq:curve-oos-guarantee}
  \Prob\left[f(x_{rob}(\gamma_N(\delta)), p^\star)\leq v_{wc}(x_{rob}(\gamma_N(\delta)), \gamma)\right]
\end{equation}
for several radii $\gamma\geq 0$ simultaneously. However, beyond the fact that $v_{wc}(x_{rob}(\gamma_N(\delta)), \gamma) \leq v^\star_{wc}(\gamma_N(\delta))$ for $\gamma\leq \gamma_N(\delta)$, the worst-case cost $v_{wc}(x_{rob}(\gamma_N(\delta)), \gamma)$ for $\gamma>\gamma_{N}(\delta)$ is uncontrolled.
In stark contrast, the GARO solution $x_{garo}$ satisfies $A(x_{garo},\gamma)\leq\alpha_{garo}\phi(\gamma)$ for \emph{every} $\gamma\in\Gamma$ by construction, independent of any chosen confidence level $\delta$. This deterministic property immediately controls the full exceedance curve: for every $\gamma\in\Gamma$,
\begin{equation}
  \label{eq:garo-exceedance}
  \Prob\left[f(x_{garo}, p^\star) \leq v_{wc}^\star(\gamma) + \alpha_{garo}\phi(\gamma)\right]
  \;\geq\;
  \Prob\left[d(p_0, p^\star) \leq \gamma\right].
\end{equation}
For any fixed $\delta$, substituting $\gamma=\gamma_N(\delta)$ and applying the classical guarantee~\eqref{def:classical-prediction} yields a single-level bound
\[
  \Prob\left[f(x_{garo}, p^\star) \leq v_{wc}^\star(\gamma_N(\delta)) + \alpha_{garo}\phi(\gamma_N(\delta))\right] \geq 1-\delta,
\]
which at that particular $\delta$ is admittedly weaker than~\eqref{eq:classical-oos-guarantee} by the additional term $\alpha_{garo}\phi(\gamma_N(\delta))$. The fundamental advantage of $x_{garo}$ is therefore not at any single confidence level but across all levels simultaneously: the robust solution $x_{rob}(\gamma_N(\delta_0))$, designed for a specific $\delta_0$, provides no control over $v_{wc}(x_{rob}(\gamma_N(\delta_0)),\gamma)$ for $\gamma>\gamma_N(\delta_0)$ and hence no guarantee at higher confidence $\delta<\delta_0$. The single decision $x_{garo}$, by contrast, satisfies~\eqref{eq:garo-exceedance} simultaneously for every $\delta\in(0,1)$, exactly supplying the exceedance curve that safety-critical reporting standards require.

\subsection{Adaptive Predictions}
\label{sec:adaptive-predictions}

In Section~\ref{sec:from-classical-wild} we pointed out that adaptive predictions can be derived from classical predictions \citep{jain_robust_2022}, and discussed two example settings in which $\kappa$ represented an unknown moment in a heavy-tailed data setting and an unknown corruption level in a corrupt data setting. Consider the setting in which we have access to a family of classical estimators $p_0(\kappa)$ parametrised in $\kappa$, all of which satisfy the pointwise guarantee in Equation~(\ref{eq:classical-predictor-family}). If we knew $\kappa$ then we could use $p_0(\kappa)$ as a classical prediction and guarantee
\begin{equation}
  \label{eq:point-wise-guarantee-lepski}
  \Prob\left[f(x_{rob, \kappa}(\gamma_N(\delta, \kappa)), p^\star)\leq v_{wc, \kappa}^\star(\gamma_N(\delta, \kappa))\right] \geq 1-\delta
\end{equation}
where $x_{rob, \kappa}(\gamma)$ and $v_{wc, \kappa}^\star(\gamma)$ are the minimizers and minimum of $\min_{x\in X} \max_{p\in P_{\gamma, \kappa}} f(x,p)$, with $P_{\gamma, \kappa}=\set{p\in P}{d(p_0(\kappa), p)\leq \gamma}$.

However, $\kappa$ is unknown in practice. Neither robust nor regret formulations can directly exploit adaptive predictions satisfying guarantee~(\ref{def:adaptive-prediction}): any decision calibrated to a single $\kappa$ forces the practitioner to commit upfront, either choosing the pessimistic $\kappa_{\max}$ (yielding an overcautious decision) or guessing $\kappa$ (risking an underprotected one). The GARO formulation sidesteps this commitment entirely, generalizing and refining the adaptation strategy of~\citet{lepskii1993asymptotically}, and in fact ensures
\[
  \Prob_{\kappa}\left[f(x_{garo}, p^\star) \leq v^\star_{wc}(\gamma_N(\delta, \kappa))+  \alpha_{garo} \phi(\gamma_N(\delta, \kappa)) \right]\geq 1-\delta.
\]
Rather than guaranteeing absolute performance, the previous inequality establishes that the performance of our decision is adaptive to the hardness of the prediction problem as captured by the auxiliary parameter $\kappa\in K\subseteq [\kappa_{\min}, \kappa_{\max}]$. Crucially, $x_{garo}$ does not depend on $\kappa$: a single decision simultaneously adapts to every level of problem hardness, just as in the classical prediction setting it supplies the full exceedance curve without committing to any particular confidence level~$\delta$.
 
Let $\beta>1$ be an approximation constant and consider a discretized subset $K_J=\{\kappa_0, \dots, \kappa_J\}$ of $K$. We have for any estimator $p_0(\kappa_j)$  with $\kappa_j\geq \kappa$ the pointwise guarantee
\[
   \Prob_{\kappa}\left[d(p_0(\kappa_j), p^\star)\leq \gamma_N(\delta, \kappa_j) \right]\geq 1-\delta.
\]
Given the above, we introduce the associated parameter set $\Gamma_J = \{\gamma_j \defn \gamma_N(\delta/(J+1), \kappa_j)\}_{j=0}^J$ and define a nondecreasing family of sets
\begin{equation}
  \label{eq:P_gamma_Lepski}
  P_{\gamma} = \set{p\in P}{d(p_0(\kappa_j), p)\leq \gamma_j \quad\forall j \in [0, \dots, J] ~\st~ \gamma_j \geq \gamma}  
\end{equation}
for any $\gamma\geq 0$. Consider the event $E_\kappa = \{ p^\star \in P_{\gamma_N(\delta/(J+1), \kappa^\star)} \}$ where $\kappa^\star = \min\set{\kappa'\in K_J}{\kappa'\geq \kappa}$. Informally, if $E_\kappa$ occurs then all of the applicable estimates $p_0(\kappa_j)$ for $\kappa_j\geq \kappa$ realize within their error margin bound $d(p_0(\kappa_j), p^\star)\leq \gamma_j$. A simple union bound argument from the pointwise guarantee (\ref{eq:classical-predictor-family}) implies now that we have
\begin{equation}
  \label{eq:union-bound-guarantee}
  \Prob_{\kappa}[E_\kappa]\geq 1-{\abs{\set{\kappa'\in K_J}{\kappa'\geq \kappa}}}\tfrac{\delta}{(J+1)}\geq 1-\delta.
\end{equation}

\begin{theorem}
  \label{thm:generalized-lepski}
  The globalized adversarial regret formulation (\ref{eq:glob-rob-regret}) with $\phi(\gamma)=\min_{x\in X}\max_{p\in P_\gamma} f(x, p)$ and $P_\gamma$ defined as in Equation \eqref{eq:P_gamma_Lepski} reduces to
  \begin{equation}
    \label{eq:glob-rob-regret-adaptive}
    \begin{array}{r@{~}l}
      \displaystyle\min_{x\in X, \, \alpha \ge 0} & \alpha\\[0.5em]
      \st & \max_{p\in  P_{\gamma_j}} f(x, p) \leq (1+\alpha) \min_{x' \in X}\max_{p'\in P_{\gamma_j}} f(x', p')  \quad \forall j\in [0,\dots, J].
    \end{array}
  \end{equation}
  Its solution $(x_{garo}, \alpha_{garo})$ does not depend on $\kappa\in K$ and satisfies on $E_\kappa$ the guarantee
    \begin{equation*}
      f(x_{garo}, p^\star) \leq (1+\alpha_{garo}) v^\star_{wc, \kappa^\star}(\gamma_N(\delta/(J+1), \kappa^\star))
    \end{equation*}
    where $\kappa^\star = \min\set{\kappa'\in K_J}{\kappa'\geq \kappa}$.
\end{theorem}
\begin{proof}
  The reduction from \ref{eq:glob-rob-regret} to \eqref{eq:glob-rob-regret-adaptive} follows from the observation that $\phi$ is an increasing function and the definition of $P_\gamma$ in Equation \eqref{eq:P_gamma_Lepski} implies that $P_{\gamma_j} = P_{\gamma}$ for any $\gamma\in [\gamma_j, \gamma_{j+1})$.

  By the choice of rate function $\phi$ the solution $(x_{garo}, \alpha_{garo})$ satisfies here a relative guarantee \eqref{eq:gror-guarantee-relative} and hence
  \[
    \max_{p\in P_\gamma} f(x_{garo}, p) \leq (1+\alpha_{garo}) \min_{x'\in X}\max_{p'\in P_\gamma} f(x', p') \quad \forall \gamma\geq 0.
  \]
  We consider $\gamma^\star = \gamma_N(\delta/(J+1), \kappa^\star)$ and observe that $E_\kappa = \{ p^\star \in \mc P_{\gamma_N(\delta/(J+1), \kappa^\star)} \}$ and so
  \[
    f(x_{garo}, p^\star) \leq (1+\alpha_{garo}) \min_{x'\in X}\max_{p'\in P_{\gamma^\star}} f(x', p').
  \]
  Finally, observe that
  \(
  P_{\gamma^\star}\subseteq P_{\gamma_N(\delta/(J+1), \kappa^\star), \kappa^\star}
  \)
  and hence
  \(
  f(x_{garo}, p^\star) \leq (1+\alpha_{garo}) v^\star_{wc, \kappa^\star}(\gamma_N(\delta/(J+1), \kappa^\star)).
  \)

\end{proof}

From Equation \eqref{eq:union-bound-guarantee} and Theorem \ref{thm:generalized-lepski} it follows that
\[
  \Prob_{\kappa} \left[f(x_{garo}, p^\star) \leq (1+\alpha_{garo}) v^\star_{wc, \kappa^\star}(\gamma_N(\delta/(J+1), \kappa^\star)) \right] \geq 1-\delta
\]
with $\kappa^\star = \min\set{\kappa'\in K_J}{\kappa'\geq \kappa}$, which should be compared to Equation \eqref{eq:point-wise-guarantee-lepski}. Note that \eqref{eq:point-wise-guarantee-lepski} is an oracle bound, achievable only if $\kappa$ is known. Relative to this oracle, GARO's bound is admittedly weaker in two respects: it carries the multiplicative overhead $(1+\alpha_{garo})$ and uses the reduced confidence $\delta/(J+1)$ in place of $\delta$. Since $\gamma_N(\delta,\kappa)$ is typically insensitive to $\delta$, the latter difference is minor and can be further controlled by an appropriate discretization of $K$. The fundamental advantage of $x_{garo}$ is again not at any single $\kappa$ but across all $\kappa\in K$ simultaneously: the single $x_{garo}$ satisfies the above bound for every $\kappa$ without any knowledge of it.

We emphasize that \ref{eq:glob-rob-regret} is not merely feasible but optimal among adaptive decisions: any decision $\hat x$ satisfying $f(\hat x, p^\star) \leq (1+\hat\alpha) v^\star_{wc, \kappa^\star}(\gamma_N(\delta/(J+1), \kappa^\star))$ on $E_\kappa$ must have $\hat\alpha \geq \alpha_{garo}$. That is, $x_{garo}$ achieves the smallest possible multiplicative excess over the oracle cost, for each $\kappa$.

\begin{examplebox}
  \label{example:adaptive-predictions-lepski}
  To recover the adaptation strategy of \citet{lepskii1993asymptotically} we consider a continuous cost function $f(x, p) = d(x,p)$ satisfying a triangular inequality and again consider access to a family of classical predictors $p_0(\kappa)$ for $\kappa\in K$ satisfying the pointwise guarantee
\[
  \Prob_{\kappa}\left[d(p_0(\kappa), p^\star) \leq \gamma_N(\delta; \kappa) \right]\geq 1- \delta.
\]
The cost function reflects here that in classical estimation settings in the absence of any particular downstream decision tasks the norm of the error is the most classical measure of performance. Consider here the increasing family of sets $P_{\gamma, \kappa} = \set{p}{d(p_0(\kappa), p) \leq \gamma }$ with associated oracle cost $v_{wc, \kappa}^\star(\gamma) = \min_{x}\max_{p\in P_\gamma} f(x, p)=\gamma$ and define the increasing family of sets $P_{\gamma}$ as in Equation \eqref{eq:P_gamma_Lepski}. The associated globalized adversarial regret formulation (\ref{eq:glob-rob-regret}) returns $(x_{garo}, \alpha_{garo})$ which following Theorem \ref{thm:generalized-lepski} satisfies on $E_\kappa$ the guarantee
  \begin{equation*}
    d(x_{garo}, p^\star) \leq (1+\alpha_{garo}) \gamma_N(\delta/(J+1), \kappa^\star) .
  \end{equation*}

  \begin{lemma}
    \label{lem:lepskis-upperbound}
    We have
    \(
    E_\kappa \implies \alpha_{garo}\leq 1.
    \)
  \end{lemma}
  \begin{proof}
    Define here $p_{\gamma} \in \arg\min_{x}\max_{p\in P_{\gamma}} d(x, p)$ and $\norm{P_\gamma} = \min_{x}\max_{p\in P_{\gamma}} d(x, p)$.
    Observe that from the triangle inequality it follows that
    \(
    v_{wc}(x, \gamma) = \max_{p\in P_{\gamma}} d(x,p) \leq \max_{p\in P_{\gamma}} d(x,p_\gamma) + d(p_\gamma, p) = d(x,p_\gamma) + \norm{P_{\gamma}}
    \)
    and
    \(
    v^\star_{wc}(\gamma) = \norm{P_{\gamma}}.
    \)
    Hence, we have that
    \[
      \alpha_{garo} \leq \min_{x, \alpha\geq 0} \set{\alpha}{ d(x,p_\gamma) + \norm{P_{\gamma}} \leq (1+\alpha) \norm{P_{\gamma}}\quad \forall \gamma\in \Gamma_J}.
    \]
    However, consider now $\bar x\in P_{\gamma^\star}$ and $\bar \alpha=1$ with $\gamma^\star$ the smallest $\gamma \in \Gamma_J$ so that $\mc P_{\gamma} \neq \emptyset$. We have
    \[
      d(\bar x,p_\gamma) \leq  \norm{P_{\gamma}}
    \]
    as $\bar x\in P_{\gamma^\star}\subseteq P_{\gamma}$ for any $\gamma\geq \gamma^\star$ from which the claim follows immediately.
  \end{proof}

  We have by Lemma \ref{lem:lepskis-upperbound} that in this adaptive prediction setting $\alpha_{garo}$ is bounded above by $1$ on $E_\kappa$.
  From Equation \eqref{eq:union-bound-guarantee} it hence follows that
  \[
    \Prob_{\kappa}\left[ d(x_{garo}, p^\star) \leq 2 \gamma_N(\delta/(J+1), \kappa^\star)\right]\geq 1-\delta.
  \]
  For the particular geometric discretization $K=\{\kappa_j = \kappa_{\min} \beta^j\}_{j=0}^J$ with $J=\ceil{\log_{\beta}(\tfrac{\kappa_{\max}}{\kappa_{\min}})}$ and $\beta > 1$ we have $\kappa^\star \leq \beta\kappa$ and recover Equation \eqref{eq:Lepskis-guarantee}.
  
\end{examplebox}

\subsection{Wild Predictions}

In the wild prediction setting of Section~\ref{sec:from-classical-wild}, the prediction error $\gamma^\star=d(p_0,p^\star)$ satisfies $\gamma^\star\leq\gamma_{\wc}$ in the worst case yet is typically much smaller ($\gamma^\star\leq\gamma_{\typ}$), with no probabilistic relationship between the two levels. Setting $\Gamma=[\gamma_{\typ},\gamma_{\wc}]$, the performance guarantees~\eqref{eq:gror-guarantee-absolute} and~\eqref{eq:gror-guarantee-relative} of Section~\ref{sec:robust_decisions_wild_predictions} apply in their simplest, fully deterministic form: whatever the unknown $\gamma^\star\in[\gamma_{\typ},\gamma_{\wc}]$ turns out to be,
\[
  f(x_{garo},p^\star) \leq v_{wc}^\star(\gamma^\star) + \alpha_{garo}\phi(\gamma^\star),
\]
without invoking any probability model on the prediction. This is the primary regime for which GARO was designed. Unlike the classical and adaptive settings, where statistical concentration inequalities are needed to give the deterministic guarantee probabilistic content, here the bound holds with certainty for any realisation of $\gamma^\star\in\Gamma$. The classical and adaptive guarantees of the preceding subsections should therefore be seen as a bonus: GARO rewards the practitioner with probabilistic exceedance curves or adaptive oracle-tracking when a statistical model is available, but requires neither. Furthermore, by construction $x_{garo}$ is weakly minimal in $(A(x,\gamma))_{\gamma\in\Gamma}$: no other single decision simultaneously reduces adversarial regret across the entire interval $\Gamma=[\gamma_{\typ},\gamma_{\wc}]$, making it the natural choice precisely when predictions cannot be trusted to carry statistical structure.

\section{Solution Methods}
\label{sec:algorithms}

The adversarial regret formulation (\ref{eq:glob-rob-regret}) is computationally more appealing than a classical regret formulation which is already NP hard even for linear optimization problems \citep{averbakh2005complexity}. Nevertheless, the adversarial regret formulation is more challenging to solve than its robust counterpart formulation stated in Equation \eqref{eq:rob-optimization}. In what follows we will take the standard robust problem in Equation \eqref{eq:rob-optimization} as a computational primitive which we assume can be solved efficiently; see \cite{ben2002robust}.
We point out that an exact tractable formulation of \ref{eq:glob-rob-regret} is possible when $v_{wc}$ is affine in $\gamma$: for concave rate functions $\phi$ this reduces to just two constraints (Section~\ref{sec:nomin-robust-form}), while for general $\phi$ it yields a finite number of convex constraints via a conjugate representation (Section~\ref{sec:line-form-with}).
Furthermore, in case $\Gamma$ has finite cardinality (as was the case when working with Lepskii's method for adaptive predictions in Section~\ref{sec:adaptive-predictions}) an exact tractable formulation is guaranteed.
Finally, we give two approximation algorithms to handle more general settings.

\subsection{Nominal-Robust Formulations}
\label{sec:nomin-robust-form}

In this setting we assume that the worst-case cost function is affine in the robustness parameter. 

\begin{assumption}
  \label{ass:affine-in-gamma}
Let $\gamma \mapsto v_{wc}(x, \gamma) = \max_{p\in P_\gamma} f(x, p)$ be affine for all $x\in X$.
\end{assumption}

The following proposition shows that under the simplifying Assumption \ref{ass:affine-in-gamma}, only the interval extremes $\{\min \Gamma, \max \Gamma\}$ of the set $\Gamma$ are of interest when the rate function $\phi$ is concave. We remark that this includes $\phi(\gamma) = v^\star_{wc}(\gamma) = \min_{x\in X} v_{wc}(x, \gamma)$ due to Assumption \ref{ass:affine-in-gamma}.

\begin{proposition}
  \label{prop:nomin-robust-form}
    Let Assumption \ref{ass:affine-in-gamma} hold and $\phi$ be concave. Then, Problem \ref{eq:glob-rob-regret} is equivalent to 
    \begin{equation}
  \label{eq:glob-rob-regret_concave_linear}
  \begin{array}{r@{~}l}
    \min_{x\in X, \,\alpha \ge 0} &  \alpha \\
    \st & \max_{p\in P_{\min\Gamma}} f(x, p)  - v_{wc}^\star(\min \Gamma) \leq \alpha \phi(\min \Gamma), \\
                                  & \max_{p\in P_{\max\Gamma}} f(x, p) - v_{wc}^\star(\max \Gamma) \leq \alpha \phi(\max \Gamma).
  \end{array}
\end{equation}
\end{proposition}
\begin{proof}
We can reformulate Problem \ref{eq:glob-rob-regret} as
\begin{equation}
  \begin{array}{r@{~}l}
    \min_{x\in X, \,\alpha\ge 0} & \alpha\\
    \st & A(x, \gamma) = \max_{p\in P_\gamma} f(x, p) -  v_{wc}^\star(\gamma) \leq \alpha \phi(\gamma) \quad \forall \gamma\in \Gamma.
  \end{array}
\end{equation}
The constraint function $\max_{p\in P_\gamma} f(x, p) -  v_{wc}^\star(\gamma) -\alpha\phi(\gamma)$ is here convex in $\gamma$. It follows that the function attains its maximum on the interval extremities $\min\Gamma$ and $\max \Gamma$.
\end{proof}

Formulation \eqref{eq:glob-rob-regret_concave_linear} is an uncertain optimization formulation with objective balancing nominal and robust performance.
This balance has been extensively studied in the robust optimization literature; see \citet{chassein2016bicriteria}.
Furthermore, the desire to balance performance when the prediction is exact ($p^\star=p_0$ and $\min \Gamma=0$) while retaining control over the performance when the prediction fails ($p^\star\in P_{\max\Gamma}$) has been the topic of the surging literature on algorithms with predictions \citep{mitzenmacher2022algorithms}.

\begin{examplebox}[Uncertain Linear Regression]
Consider an uncertain linear regression problem
\(
  \min_{x\in X} \norm{Ax-b}
\)
where the matrix $A$ is uncertain. We assume that we have a prediction $A_0$ for the matrix $A$ where we expect the operator norm $\norm{A_0-A}=\max_{u\neq 0}\tfrac{\norm{(A-A_0)u}}{\|u\|}$ to remain small.
A well known result from the robust optimization literature is that Assumption \ref{ass:affine-in-gamma} holds as we have
\(
  v_{wc}(x, \gamma) \defn \max_{\|A-A_0\|\leq \gamma}\norm{Ax-b}  = \norm{A_0x-b} + \gamma \norm{x}
\)
revealing the classical connection between robustness and regularization. 
The \ref{eq:glob-rob-regret} formulation of this problem with $\Gamma=[0, \gamma_{\max}]$ for a concave rate function $\phi$ simplifies to
\[
\begin{array}{r@{~}l}
  \min_{x\in X, \,\alpha \ge 0} &  \alpha \\
  \st  & \norm{A_0 x-b}  - v^\star_{wc}(0) \leq \alpha \phi(0), \\
                                &\norm{A_0x-b} + \gamma_{\max} \norm{x} - v^\star_{wc}(\gamma_{\max}) \leq \alpha \phi(\gamma_{\max}).
\end{array}
\]
\end{examplebox}

\subsection{Linear Formulations with Polyhedral Norm Uncertainty}
\label{sec:line-form-with}

Adversarial robust optimization formulation (\ref{eq:glob-rob-regret}) has an exact reformulation also in the context of uncertain linear optimization problems.

\begin{assumption}
  \label{ass:uncertain-linear}
  Let $X$ be a polyhedral set, $P$ unconstrained, the distance function $d(p, p')=\norm{p-p'}$ defined by a polyhedral norm and $f(x, p) = x\tpose p$ a linear function.
\end{assumption}

A standard result from the robust optimization literature is that
\[
  v_{wc}(x, \gamma)\defn \max_{\|p-p_0\|\leq \gamma}x\tpose p = x\tpose p_0 +\gamma \norm{x}_\star
\]
where $\|\cdot\|_*$ is the dual norm. Hence, the worst-case function is affine in $\gamma$ for any decision $x$ and satisfies Assumption \ref{ass:affine-in-gamma}. Consequently, for concave rate functions $\phi$ Proposition \ref{prop:nomin-robust-form} applies and our associated (\ref{eq:glob-rob-regret}) formulation reduces to simply balancing nominal and worst-case costs.

For general rate functions $\phi$, the situation is more complicated. 
As however the dual norm $\|\cdot \|_\star$ is also polyhedral and $X$ a polytope, we have that the robust oracle cost
\begin{align}
  \label{eq:uncertain-linear-oracle}
v_{wc}^\star(\gamma) \defn \min_{x\in X} \ & x \tpose p_0  + \gamma \norm{x}_\star 
\end{align}
is characterized as a linear representable minimization problem. A standard result from parametric linear optimization \citep{bertsimas1997introduction} is that we can partition the interval $\Gamma = \bigcup_{t=1}^{T}\Gamma_t$ such that on each interval $\Gamma_t = [\gamma_t,\gamma_{t+1})$ an extreme point $x_t$ is optimal in \eqref{eq:uncertain-linear-oracle} so that $v_{wc}^\star(\gamma) = {x_t} \tpose p_0  + \gamma \|x_t\|_\star  $ for all $\gamma\in \Gamma_t$.
The semi-infinite constraint in the adversarial robust formulation (\ref{eq:glob-rob-regret}) hence reduces to $T$ semi-infinite constraints
\begin{align}
  \label{eq:J-semi-infinite}
 x\tpose p_0 + \gamma \| x\|_\star -  {x_t}\tpose p_0 - \gamma \| x_t\|_\star & \leq\alpha \phi(\gamma) \quad \forall \gamma\in \Gamma_t \quad \forall t\in [1,\ldots, T].
\end{align}
We now indicate that the latter constraints can be reformulated as $T$ convenient convex constraints.
Associate indeed with every interval $\Gamma_t$ an increasing convex conjugate function given as
\begin{align}
  \phi^\star_t(\Delta) \defn & \max_{\gamma\in \Gamma_t} \Delta\gamma - \phi(\gamma)
     \label{eq:convex-conjugate}
\end{align}
and hence we can reformulate the semi-infinite constraints \eqref{eq:J-semi-infinite} as the convex constraints
\[
  (x-x_t)\tpose p_0  + \alpha \phi_t^\star\left(\tfrac{\left(\|x\|_\star-\|x_t\|_\star\right)}{\alpha}\right) \leq 0  \quad \forall t\in [1,\ldots, T].
\]
Here the $(\Delta, \alpha) \mapsto \alpha \phi_t^\star(\Delta/\alpha)$ term is the perspective of $ \Delta \mapsto \phi_t^\star(\Delta)$, which is well-defined and convex for $\alpha\geq 0$ \citep{combettes2018perspective}.

\subsection{Discretization Formulations}\label{sec:discretization}
As $\gamma\in \Gamma$ is a one-dimensional parameter, Problem \ref{eq:glob-rob-regret} can be approximated by appropriately discretizing the interval $\Gamma$.
For a given discretization $\min\Gamma=\gamma_1\leq \gamma_2\leq \cdots \leq \gamma_{T+1}=\max\Gamma$, the discretized version of \ref{eq:glob-rob-regret} is defined here as the problem
\begin{equation}
  \label{eq:glob-rob-regret_discretized}\tag{GARO$_d$}
  \begin{aligned}
    \min_{x\in X, \,\alpha\ge 0} & \quad \alpha\\
    \st & \quad \max_{p\in P_{\gamma_t}}  f(x, p) - v_{wc}^\star(\gamma_t) \leq \alpha \phi(\gamma_t) \quad \forall t\in [1,\ldots ,T+1] .
  \end{aligned}
\end{equation}
We recall that the oracle costs $v_{wc}^\star(\gamma_t) \defn \min_{x' \in X}\max_{p'\in P_{\gamma_t}} f(x', p')$ require merely the solution of a classical robust optimization formulation. Hence, Problem \eqref{eq:glob-rob-regret_discretized} is an optimization problem with $T+1$ robust constraints. Using standard dual representations, the maximization terms in \eqref{eq:glob-rob-regret_discretized} can in a wide range of circumstances be reformulated (see \cite{ben2015deriving}) resulting in an equivalent optimization problem amenable to off-the-shelf solvers. Alternatively, classical constraint generation methods for robust optimization problems could be applied to each of the $T+1$ constraints to solve the problem.

\begin{assumption}
  \label{ass:lipschitz}
  Assume a selection $x_{rob}(\gamma) \in X_{rob}(\gamma)=\arg\min_{x\in X} \max_{p\in P_\gamma} f(x, p)$ exists for all $\gamma\in\Gamma$. We assume that $\gamma'\mapsto \max_{p\in P_{\gamma'}} f(x_{rob}(\gamma), p)$ is Lipschitz continuous with constant $L$ for all $\gamma\in\Gamma$.
\end{assumption}

In the following we study the effect of discretizing the interval $\Gamma$. The following theorem states that the discretized formulation \eqref{eq:glob-rob-regret_discretized} returns a solution which violates the constraints in the adversarial robust formulation \ref{eq:glob-rob-regret} by a value which is well controlled by the grid size under Assumption \ref{ass:lipschitz}.

\begin{theorem}\label{thm:discretization_error}
Let Assumption \ref{ass:lipschitz} hold and let $\Delta:=\max_{t=1,\ldots, T} \gamma_{t+1} - \gamma_{t}$ be the grid size. Denote with $L'$ the Lipschitz constant of $\phi$ over $\Gamma$. An optimal solution $(x^\Delta_{garo},\alpha^\Delta_{garo})$ in \eqref{eq:glob-rob-regret_discretized} satisfies for all $\gamma \in \Gamma$ the inequality
\[
  \max_{p\in P_\gamma} \left[ f(x^\Delta_{garo}, p) - \min_{x' \in X}\max_{p'\in P_\gamma} f(x', p') \right] - \alpha^\Delta_{garo}\phi(\gamma) \le \Delta(L+\alpha^\Delta_{garo} L').
\]
\end{theorem}
In other words, $(x^\Delta_{garo},\alpha^\Delta_{garo})$ is $\Delta(L+\alpha^\Delta_{garo} L')$-feasible in \ref{eq:glob-rob-regret}.
Since $\alpha^\Delta_{garo}$ is available after solving the discretized problem, the bound $\Delta(L+\alpha^\Delta_{garo} L') \leq \varepsilon$ can be verified a posteriori and the grid refined if necessary.
Theorem~\ref{thm:discretization_error} is a feasibility statement, but it implies an optimality guarantee on $\alpha^\Delta_{garo}$ directly.
Denote $\phi_{\min}\defn\phi(\min\Gamma)$.
If $\phi_{\min}=0$, the GARO constraint at $\gamma=\min\Gamma$ forces $A(x,\min\Gamma)=0$, so $x_{garo}\in X_{rob}(\min\Gamma)$ and the problem essentially reduces to a standard robust formulation.
In a nontrivial case $\phi_{\min}>0$, the discretized problem \ref{eq:glob-rob-regret_discretized} has fewer constraints than \ref{eq:glob-rob-regret}, giving $\alpha_{garo}^\Delta\leq\alpha_{garo}$.
Theorem~\ref{thm:discretization_error} further implies that $(x^\Delta_{garo},\,\alpha_{garo}^\Delta+\Delta(L+\alpha_{garo}^\Delta L')/\phi_{\min})$ is feasible in \ref{eq:glob-rob-regret}, since $\phi(\gamma)\geq\phi_{\min}$ for all $\gamma\in\Gamma$, and hence $\alpha_{garo}\leq\alpha_{garo}^\Delta+\Delta(L+\alpha_{garo}^\Delta L')/\phi_{\min}$.
Combining,
\[
  0 \;\leq\; \alpha_{garo} - \alpha_{garo}^\Delta \leq \tfrac{\Delta(L+\alpha_{garo}^\Delta L')}{\phi_{\min}}\leq \tfrac{\Delta(L+\alpha_{garo} L')}{\phi_{\min}},
\]
so $\alpha_{garo}^\Delta\to\alpha_{garo}$ at rate $\mc O(\Delta)$ as the grid is refined.

\begin{proof}[Proof of Theorem \ref{thm:discretization_error}]
We recall that we defined
\(
v_{wc}(x, \gamma)\defn \max_{p\in P_\gamma} f(x,p)
\)
and
\(
v_{wc}^\star(\gamma) \defn \min_{x\in X} v_{wc}(x, \gamma).
\)
Assume now that $\gamma\in \Gamma$ is arbitrary. There exists a $j\in \{ 1,\ldots ,T\}$ such that $\gamma \in [\gamma_j, \gamma_{j+1})$. Let $x_{j} = x_{rob}(\gamma_j) \in X_{rob}(\gamma_j)$ be an optimal solution of the min-max problem with uncertainty set $P_{\gamma_j}$.
We can bound
\begin{align*}
v_{wc}(x_{garo}^\Delta, \gamma) - v_{wc}^\star(\gamma) & \le v_{wc}(x_{garo}^\Delta, \gamma_{j+1}) - v_{wc}^\star(\gamma_j) \\
& \le v_{wc}^\star(\gamma_{j+1}) - v_{wc}^\star(\gamma_j) + \alpha^\Delta_{garo} \phi(\gamma_{j+1}) \\
& \le v_{wc}(x_{j}, \gamma_{j+1}) - v_{wc}(x_{j}, \gamma_{j}) + \alpha^\Delta_{garo} \phi(\gamma_{j+1})
\end{align*}
where the first inequality follows since $v_{wc}(x, \gamma)$ is non-decreasing in $\gamma$ for every $x$, the second inequality follows since $(x_{garo}^\Delta,\alpha^\Delta_{garo})$ is a feasible solution in Problem \eqref{eq:glob-rob-regret_discretized}, and the third inequality follows since $x_{j} \in X$ is a minimizer of $v_{wc}(x, \gamma_{j})$. We can continue bounding the last term as
\begin{align*}
    & v_{wc}(x_{j}, \gamma_{j+1}) - v_{wc}(x_{j}, \gamma_{j}) + \alpha^\Delta_{garo} \phi(\gamma_{j+1}) \\
    & \le (\gamma_{j+1} - \gamma_{j}) L + \alpha^\Delta_{garo} \phi(\gamma_{j+1}) \\
    & = (\gamma_{j+1} - \gamma_{j}) L + \alpha^\Delta_{garo} \phi(\gamma) + \alpha^\Delta_{garo} \left(\phi(\gamma_{j+1}) - \phi(\gamma)\right) \\
    & \leq (\gamma_{j+1} - \gamma_{j}) L + \alpha^\Delta_{garo} \phi(\gamma) + \alpha^\Delta_{garo} L' (\gamma_{j+1} - \gamma)\\
    & \leq (\gamma_{j+1} - \gamma_{j}) L + \alpha^\Delta_{garo} \phi(\gamma) + \alpha^\Delta_{garo} L' \left(\gamma_{j+1}- \gamma_{j}\right)
\end{align*}
where the first inequality follows from Assumption \ref{ass:lipschitz}, the second inequality follows from Lipschitz continuity of $\phi$ with constant $L'$, and the third inequality follows since $\gamma_j\le \gamma$.
\end{proof}

In a wide array of problems, including those described in Sections \ref{sec:nomin-robust-form} and \ref{sec:line-form-with}, the mapping $\gamma \mapsto v_{wc}(x, \gamma) \defn \max_{p\in P_\gamma} f(x,p)$ is Lipschitz continuous and hence Assumption \ref{ass:lipschitz} indeed holds. For instance, when the worst-case cost function $\gamma \mapsto v_{wc}(x, \gamma)$ is concave in $\gamma$, then $v^\star_{wc}(\gamma) = \min_{x\in X} v_{wc}(x, \gamma)$ is a concave function as well and Assumption \ref{ass:lipschitz} holds with Lipschitz constant $L=\partial_\gamma v^\star_{wc}(0) = \partial_{\gamma} v_{wc}(x_{nom}, 0)$ by the envelope theorem.

\begin{examplebox}
Reconsider the uncertain linear regression problem
\(
  \min_{x\in X} \norm{Ax-b}
\)
where the matrix $A$ is uncertain. 
We have
\(
  v_{wc}(x, \gamma) \defn \max_{\|A-A_0\|\leq \gamma}\norm{Ax-b}  = \norm{A_0x-b} + \gamma \norm{x}
\)
is affine and hence concave. It follows that Assumption \ref{ass:lipschitz} holds for $L=\norm{x_{nom}} = \|A_0^\dagger b\|$.
Similarly, the uncertain linear optimization problem in Section \ref{sec:line-form-with} where
\(
  v_{wc}(x, \gamma) \defn \max_{\|p-p_0\|\leq\gamma} x\tpose p = x\tpose p_0+\gamma \norm{x}_\star
\)
which is affine and hence concave. It follows that Assumption \ref{ass:lipschitz} holds for $L=\norm{x_{nom}}_\star$.
\end{examplebox}

\begin{examplebox}
  Consider an uncertain stochastic optimization problem $\min_{x\in X} \E{\mb P}{\ell(x, \xi)}$ where the distribution $\P$ is uncertain.
  We assume here that $\P$ is supported on $\Xi$ and that the $\ell_{\min} \leq \ell(x, \xi) \leq \ell_{\max}$ for all $x\in X$ and $\xi\in \Xi$. 
  Suppose we have access to a prediction $\P_0$ for which we expect the Kullback-Leibler distance to $P$ to remain small.
  Given that the Kullback-Leibler distance behaves quadratically, i.e., $\KL(\lambda\P_0+(1-\lambda)\P, \P)=\frac {\lambda^2}{2} \chi^2(\P_0, \P)+o(\lambda^2)$ with $\chi^2(\P_0, \P)\defn \int \left(\tfrac{\d\P_0}{\d \P}(\xi)-1\right)^2\d \P(\xi)$, we will  impose  $\KL(\P_0, \P)\leq \gamma^2$.
  A well known result from the optimization literature \citep{vanparys2021data} is that the worst-case cost function satisfies the dual characterization
  \begin{align*}
    v_{wc}(x, \gamma) =& \left\{
    \begin{array}{r@{~~}l}
      \displaystyle\max_{\P} & \E{\P}{\ell(x, \xi)}\\[0.5em]
      \st & \KL(\P_0, \P)\leq \gamma^2
    \end{array}\right.\\
    =& \left\{
    \begin{array}{r@{~~}l}
      \displaystyle\min_{\alpha} & \alpha - \exp(-\gamma^2) \exp\left(\E{\P_0}{\log(\alpha-\ell(x,\xi))}\right)\\[0.5em]
      \st & \max_{\xi\in \Xi} \ell(x, \xi) \leq \alpha
    \end{array}\right.
  \end{align*}
  and can be evaluated efficiently using simple bisection search. Furthermore, as Lemma \ref{lemma:KL-lips} points out the worst-case function is uniformly Lipschitz in the decision $x\in X$ and hence Assumption \ref{ass:lipschitz} is satisfied.
\end{examplebox}

\subsection{Constraint Generation}

The number of discretization points required to achieve a small constraint violation via the discretization approach in the previous section can be large when working with large Lipschitz constants or a large interval $\Gamma$. In this case it could be beneficial to iteratively discretize $\Gamma$ by calling a separation oracle which finds the $\gamma$ value that violates the constraint in \ref{eq:glob-rob-regret}; see Algorithm \ref{alg:general_case}.

\begin{assumption}
  \label{ass:global-lipschitz}
  Let $f(x,p)$ be Lipschitz continuous on $x\in X$ uniformly for all $p\in P_{\max\Gamma}$, $X$, $\Gamma$ and $P_\gamma$ compact sets for every $\gamma\in \Gamma$, and $\phi$ Lipschitz over $\Gamma$.
\end{assumption}

\begin{assumption}
  \label{ass:separation}
  We assume we can solve the univariate separation problem
  \[
    \max_{\gamma\in \Gamma} \ \max_{p\in P_\gamma} f(x, p) - \min_{x' \in X}\max_{p'\in P_\gamma} f(x', p') -\alpha \phi(\gamma)
  \]
  for any $x\in X$ and $\alpha\geq 0$.
\end{assumption}

We show that under Assumptions \ref{ass:global-lipschitz} and \ref{ass:separation} constraint generation Algorithm \ref{alg:general_case} converges to an approximate solution in our adversarial regret optimization formulation.

\begin{algorithm}
\caption{Constraint Generation Algorithm for \ref{eq:glob-rob-regret}}
\begin{algorithmic}[1]
\Require $X\subset \Re^n$ compact, $f(x,p)$ Lipschitz continuous in $x$ for every $p\in P$, interval $\Gamma$, $\varepsilon >0$
\Ensure $\varepsilon$-optimal solution of Problem \ref{eq:glob-rob-regret}

\State Set $C \leftarrow \emptyset$
\State Set $v_{wc}^\star\leftarrow\infty$
\While{$v_{wc}^\star\geq \varepsilon$}
\State Calculate an optimal solution $(x^\star,\alpha^\star)$ in the main problem
\[
  \begin{array}{r@{~}l}
    \displaystyle\min_{x\in X, \,\alpha\ge 0} & \alpha\\[0.5em]
    \st & \displaystyle\max_{p\in P_\gamma} \left[ f(x, p) - \min_{x' \in X}\max_{p'\in P_\gamma} f(x', p') \right] \leq\alpha \phi(\gamma) \quad \forall \gamma\in C.
  \end{array}
\]
\State Calculate an optimal solution $\gamma^\star$ of the separation problem
    \[
    v_{wc}^\star\leftarrow\max_{\gamma\in \Gamma} \ \max_{p\in P_\gamma} \left[ f(x^\star, p) - \min_{x' \in X}\max_{p'\in P_\gamma} f(x', p') \right] -\alpha^\star \phi(\gamma).
    \]
\State Set $C\leftarrow C\cup \{ \gamma^\star\}$.
\EndWhile\\
\Return $(x^\star,\alpha^\star)$
\end{algorithmic}
\label{alg:general_case}
\end{algorithm}

\begin{theorem}
Let Assumptions \ref{ass:global-lipschitz} and \ref{ass:separation} hold and let \ref{eq:glob-rob-regret} be feasible. Then, Algorithm \ref{alg:general_case} converges to an $\varepsilon$-feasible solution of our \ref{eq:glob-rob-regret}.
\end{theorem}
\begin{proof}
  Since \ref{eq:glob-rob-regret} is feasible there exists a solution $(\bar x, A)$ satisfying
  \[
    \max_{p\in P_\gamma} \left[ f(\bar x, p) - \min_{x' \in X}\max_{p'\in P_\gamma} f(x', p') \right] \leq A \phi(\gamma) \quad \forall \gamma\in \Gamma.
  \]
  We can hence restrict $(x, \alpha)$ in \ref{eq:glob-rob-regret} to the compact set $X\times [0, A]$ without loss of optimality.

    Following the convergence analysis in \cite[Section 5.2]{mutapcic2009cutting}, our cutting plane algorithm converges if the constraint function
    \[
      g(x, \alpha; \gamma) = \max_{p\in P_\gamma}  f(x, p) - \min_{x' \in X}\max_{p'\in P_\gamma} f(x', p')  -\alpha \phi(\gamma)
    \]
    are Lipschitz continuous in the decision variables, i.e., $x$ and $\alpha$, uniformly over the quantifiers $\gamma \in \Gamma$.
    Since $f$ is Lipschitz continuous in $x$ uniformly over $p\in P_{\max \Gamma}$ the worst-case cost $(x, \alpha) \mapsto v_{wc}(x, \gamma) \defn \max_{p\in P_\gamma} f(x,p)$ must be Lipschitz continuous for every $\gamma$ since
    \[
    v_{wc}(x_1, \gamma) - v_{wc}(x_2, \gamma) \le f(x_1,p^\star) - f(x_2,p^\star) \le L \| x_1 - x_2\|
    \]
    where $p^\star$ is a maximizer of $\max_{p\in P_\gamma} f(x_1,p)$ and $L$ the uniform Lipschitz constant of $f(x,p)$ in $x\in X$ over $p\in P_{\max \Gamma}$.
    Clearly $(x, \alpha)\mapsto\alpha \phi(\gamma)$ is Lipschitz continuous (and linear) in $\alpha$ with Lipschitz constant $\max_{\gamma\in\Gamma}\phi(\gamma) < \infty$, and the claim follows.
\end{proof}

By the same reasoning as in Section~\ref{sec:discretization}, $\varepsilon$-feasibility also implies an optimality guarantee in the nontrivial case $\phi_{\min}>0$.
At termination, $(x^\star,\alpha^\star)$ satisfies $A(x^\star,\gamma)\leq \alpha^\star\phi(\gamma)+\varepsilon$ for all $\gamma\in\Gamma$, so $(x^\star,\,\alpha^\star+\varepsilon/\phi_{\min})$ is feasible in \ref{eq:glob-rob-regret} and hence $\alpha_{garo}\leq\alpha^\star+\varepsilon/\phi_{\min}$.
Combined with $\alpha^\star\leq\alpha_{garo}$, we obtain $0\leq\alpha_{garo}-\alpha^\star\leq\varepsilon/\phi_{\min}$.

The separation problem over $\gamma$ which has to be solved in Algorithm \ref{alg:general_case} is a univariate maximization problem.
When the function $\gamma \to \max_{p\in P_\gamma} f(x, p)$ is Lipschitz this can be done to any desired accuracy using the Shubert--Piyavskii algorithm.
If $\gamma \to \max_{p\in P_\gamma} f(x, p)$ is concave, this could also be done by exploiting DC-programming algorithms.
However, in certain situations when the objective function is linear, the separation problem can be calculated exactly; see Section~\ref{sec:line-form-with}.

\begin{examplebox}
  We consider a generalized version of the uncertain linear optimization problem discussed in Section \ref{sec:line-form-with} in which $P_\gamma=\set{p\in P}{\|p-p_0\|\leq \gamma}$ and where now $P$ can be an arbitrary convex polytope. Consider support function $h_P(u)=\max_{p\in P} u\tpose p$.
  In the linear case $f(x,p)=x\tpose p$ we have that
  \(
    v_{wc}(x, \gamma) \defn \max_{p\in P_\gamma} f(x, p) = \min_{y}  y\tpose p_0 +\norm{y}_\star\gamma+h_P(x-y) .
  \)

  Observe that as here $\|\cdot \|_\star$ is a polyhedral norm, $X$ a polytope and $h_P$ a linear representable function 
  both
  \begin{align}
    \begin{split}
    v_{wc}(x,\gamma) \defn  & \min_{y}  y\tpose p_0 +\norm{y}_\star\gamma+h_P(x-y) \quad \forall \gamma\in\Gamma\\
      v^\star_{wc}(\gamma)\defn  \min_{x\in X} \max_{p\in P_\gamma} f(x, p)=& \min_{x'\in X, y'}  {y'}\tpose p_0 +\norm{y'}_\star\gamma+h_P(x'-y') \quad \forall \gamma\in\Gamma
    \end{split}
    \label{eq:two-problems}
  \end{align}
  are linear optimization problems.
  Hence, for any fixed decision $x$ we can partition $\Gamma=\cup_{t=1}^{T} \Gamma_t$ such that on each $\Gamma_t=[\gamma_t, \gamma_{t+1})$ we can find extreme points $y_t$ and $(x_t',y_t')$ in \eqref{eq:two-problems} so that
  \begin{align*}
    \begin{split}
      v_{wc}(x,\gamma) \defn  & y_t\tpose p_0 +\|y_t\|_\star\gamma+h_P(x-y_t) \quad \forall \gamma\in \Gamma_t\\
      v^\star_{wc}(\gamma)\defn  \min_{x\in X} \max_{p\in P_\gamma} f(x, p)=& {y_t'}\tpose p_0 +\norm{y'_t}_\star\gamma+h_P(x'_t-y'_t)\quad  \forall \gamma\in \Gamma_t
    \end{split}
  \end{align*}
  Hence, using the same argument as in Section \ref{sec:line-form-with} we can write the separation problem as a maximum of at most $T$ terms, i.e.,
  \begin{align*}
  & \max_{\gamma\in \Gamma} \ \max_{p\in P_\gamma} \left[ f(x, p) - \min_{x' \in X}\max_{p'\in P_\gamma} f(x', p') \right] -\alpha\phi(\gamma)\\
 = & \max_{t=1}^{T}~ (y_t-y'_t)\tpose p_0 +h_P(x-y_t)-h_P(x'_t-y'_t) + \alpha \phi^\star_t\left(\tfrac{\left(\norm{y_t'}_\star-\norm{y_t}_\star\right)}{\alpha}\right)
  \end{align*}
  where the convex conjugate $\phi^\star_t$ associated with each $\Gamma_t$ is stated in Equation \eqref{eq:convex-conjugate}.
\end{examplebox}

\section{A Minimum Knapsack Experiment}\label{sec:Experiments}
We apply the algorithmic framework of Section~\ref{sec:algorithms} to minimum knapsack problems with uncertainty sets constructed from randomly generated data to assess the practical performance of GARO. The code for the experiments is publicly available at \url{https://github.com/JannisKu/GARO.git}.

\subsection{Setup and Baseline Methods}
We consider as deterministic problem the minimum knapsack problem of the form
\[
  \begin{array}{r@{~}l}
    \min_{x} & x\tpose p \\
    \st & a\tpose x \ge b \\
        & x\in [0,100]^n.
  \end{array}
\]
The bounds on the variables prevent unbounded problems when a random scenario $p$ has negative entries. Furthermore, we refrain from applying an upper bound of one to each variable since the effect of the robust models in this case is only marginal. Intuitively, for the chosen variant a nominal solution is sparse, while robustness has the effect of reducing this sparsity.

We assume a training set $\mc D$ and a test set $\mc T$ of historical data points in $80/20$ proportion. We then define the uncertainty set
\[
P_\gamma:=\{ p\in \Re^n: (p-p_0)\tpose \tilde\Sigma^{-1} (p-p_0)\le \gamma\}
\]
where $p_0$ is the empirical mean prediction and $\tilde \Sigma$ the empirical covariance matrix of $\mc D$. If the smallest eigenvalue of $\tilde \Sigma$ is smaller than $10^{-4}$, we add the matrix $10^{-4}\mathbb I_n$ to $\tilde \Sigma$. For a given confidence level $\rho\in [0,1)$, we define $\gamma_\rho$ as the smallest value such that $P_{\gamma_\rho}$ contains a fraction $\rho$ of all training data points. We define the interval for the $\gamma$-parameter as $\Gamma = [0, \gamma_{0.99} ]$. Choosing $\gamma_{0.99}$ instead of $\gamma_{1.0}$ avoids calibrating the interval to the single most extreme sample, yielding a more stable estimate of the effective data range.

We compare our method (GARO) against several baseline methods:
\begin{itemize}
\item (RO) robust optimization with ellipsoidal uncertainty set,
\item (RO$_d$) robust optimization with discrete uncertainty set,
\item (SAT) robust satisficing with ellipsoidal uncertainty set, and
\item (REG) classical robust regret with discrete uncertainty set.
\end{itemize}
We chose the discrete version of robust regret since the same problem with convex uncertainty sets is computationally intractable. We briefly describe below how each method is solved. We denote by $X:=\{ x\in\Re^n : a\tpose x \ge b,\ x\in[0,100]^n\}$ the feasible set.

\paragraph*{Classical Robust Optimization with Convex Uncertainty (RO)}
For a given $\gamma\in\Gamma$, we solve the problem
\(
\min_{x\in X} \max_{p\in P_{\gamma}} x\tpose p
\)
which can be reformulated as
\begin{equation}\label{eq:reformulation_RO_computations}
\min_{x\in X} x\tpose p_0 + \sqrt{\gamma}\sqrt{x\tpose \tilde\Sigma x}.
\end{equation}
\paragraph*{Classical Robust Optimization with Discrete Uncertainty (RO$_d$)}
For a given $\gamma\in\Gamma$, we collect all training scenarios which are contained in $P_{\gamma}$, i.e., we define $\mc D(\gamma):=\mc D\cap P_{\gamma}$ and solve the problem
\(
\min_{x\in X} \max_{p\in\mc D(\gamma)} x\tpose p
\)
which can be reformulated as
\[
  \begin{array}{r@{~}l}
    \min_{x\in X,\alpha} & \alpha \\
    \st & x\tpose p \le \alpha \quad \forall p\in \mc D(\gamma).
  \end{array}
\]

\paragraph*{Robust Satisficing (SAT)}
We discretize the interval $[0,\gamma_{0.99}]$ into $T+1=100$ equidistant values $0=\gamma_1 \le \ldots \le \gamma_{T+1} = \gamma_{0.99}$. We then calculate the oracle nominal cost $v^\star_{wc}(0) = \min_{x\in X} x\tpose p_0$. For a given factor $\beta \in [1,\infty)$ we define the target value as $f_{0}=\beta v^\star_{wc}(0)$ and then solve the robust satisficing problem
\[
  \begin{array}{r@{~}l}
    \min_{x\in X,\alpha\geq 0} & \alpha \\
    \st & \max_{p\in P_{\gamma_t}} x\tpose p \le f_{0} + \alpha \gamma_t \quad \forall t\in [1,\ldots,T+1].
  \end{array}
\]

\paragraph*{Robust Regret (REG)}
The regret formulation for a given $\gamma\in\Gamma$ solves $\min_{x\in X} \max_{p\in P_\gamma} [x\tpose p - \min_{x'\in X} {x'}\tpose p]$, which is computationally intractable due to the nonconvexity induced by the inner minimization.
To allow computation, we consider an approximation in which the outer maximization ranges only over the training scenarios contained in $P_\gamma$.
We define $\mc D(\gamma):= \mc D\cap P_{\gamma}$ and solve the approximate problem
\(
\min_{x\in X} \max_{p\in\mc D(\gamma)} \left[ x\tpose p - \min_{x'\in X} {x'}\tpose p\right]
\)
which can be reformulated as
\[
  \begin{array}{r@{~}l}
    \min_{x\in X,\alpha} & \alpha \\
    \st & x\tpose p \le \alpha + \min_{x'\in X} {x'}\tpose p \quad \forall p\in \mc D(\gamma).
  \end{array}
\]

\paragraph*{Globalized Adversarial Regret Optimization (GARO)}
We discretize the interval $[0,\gamma_{0.99}]$ into $T+1=100$ equidistant values $0=\gamma_1 \le \ldots \le \gamma_{T+1} = \gamma_{0.99}$. For each $\gamma_t$ we calculate the optimal value $v^\star_{wc}(\gamma_t)$ of the classical robust problem, i.e.,
\[
v^\star_{wc}(\gamma_t):=\min_{x\in X} \ \max_{p\in P_{\gamma_t}} \ x\tpose p.
\]
We use the power rate function $\phi(\gamma)=(1+\gamma)^q$ with parameter $q\geq 0$ introduced in Example~\ref{example:rate-function}, which interpolates between the absolute guarantee ($q=0$) and increasingly permissive adversarial regret growth for larger $q$.
We then solve the discretized global robust regret problem
\[
  \begin{array}{r@{~}l}
    \min_{x\in X,\alpha} & \alpha \\
    \st & \max_{p\in P_{\gamma_t}} x\tpose p - v^\star_{wc}(\gamma_t) \le \alpha(1+\gamma_t)^q \quad \forall t\in [1,\ldots,T+1].
  \end{array}
\]
The maximization term can be reformulated as shown for the classical robust optimization baseline.

\paragraph*{Instance Generation}
We generate $5$ random problem instances and $5$ random data sets, each containing $m=5000$ random observations on the profit vector $p$, which we split into training data $\mc D$ and test data $\mc T$ in $80/20$ proportion. The problem instances are generated as follows:
\begin{enumerate}[(i)]
    \item We set $n=50$ and draw five different weight vectors $a$ uniformly from the box $[25,100]^n$.
    \item We define the knapsack capacity as $b=\frac{2}{5} \sum_{i=1}^na_i$.
\end{enumerate}
For the data sets we consider different types of randomly generated data:
\begin{enumerate}[(a)]
    \item \textbf{Gaussian}: We draw $\mu$ uniformly from the box $[0,50]^n$ and the standard deviation $\sigma_i$ uniformly from $[0,\frac{1}{2}\mu_i]$ for each $i\in[n]$. We additionally draw a random orthonormal basis $u^1, \ldots , u^n\in \Re^n$ and define the covariance matrix as $\Sigma = \sum_{i=1}^{n} \sigma_i^2 u^i(u^i)^\top$. We then sample $m$ random points from $\mathcal N(\mu, \Sigma)$.
    \item \textbf{Gaussian with Inverse Mean-Variance Relationship}: The points are generated as in the Gaussian case, except that the standard deviation $\sigma_i$ is chosen uniformly from $[0,50-\mu_i]$ for each $i\in [n]$. This ensures that large mean values are paired with small standard deviations and vice versa. This choice is motivated by the observation that in the robust case the size of the uncertainty has a larger impact on the ordering of entries in the worst-case scenario, which intuitively leads to more diverse solutions for different $\gamma$ values.
    \item \textbf{Heavy-Tail}: We draw $m$ random points from the Pareto distribution with tail index $1.5$, and add the all-ones vector $\mathbf{1}_n$ to each drawn point.
\end{enumerate}

\subsection{Preliminary Analysis of Classical Robust Optimization}
We start with an analysis of RO with Gaussian samples to introduce our evaluation metrics and show the drawbacks of the guarantees provided by the model. A critical question in RO is how to choose $\gamma$ for the uncertainty set $P_\gamma$ to obtain a solution which performs well out-of-sample under a given set of performance metrics. To determine a good value for $\gamma$, we solve RO with radius $\gamma = \theta \gamma_{0.99}$ where $\theta\in \{ 0.7^i: i=0,1,\ldots ,25\}$, which we denote by RO($\theta$). Note that for $\theta=0$ we obtain the nominal solution associated with the prediction $p_0$, while for $\theta=1$ we use an ellipsoid containing $99\%$ of the training data. For each $\gamma$ we solve each combination of the $5$ problem instances and the $5$ training data sets, resulting in $25$ instances.

Each solution is evaluated on all test scenarios $p\in \mc T$, i.e., we calculate the objective value $x\tpose p$ for every solution $x$ and test scenario $p$, and evaluate the average, worst-case and $90\%$-quantile performance. While robust optimization typically aims for solutions with good performance in the upper quantiles, these solutions should also not perform too poorly on average. Hence, decision makers are often interested in the trade-off between average and worst-case performance. In Figure \ref{fig:trade_off_RO} we show the average out-of-sample performance over all $25$ instances, plotting average vs.\ worst-case (left) and average vs.\ $90\%$-quantile (right).
\begin{figure}
    \centering
\includegraphics[width=0.9\textwidth]{ 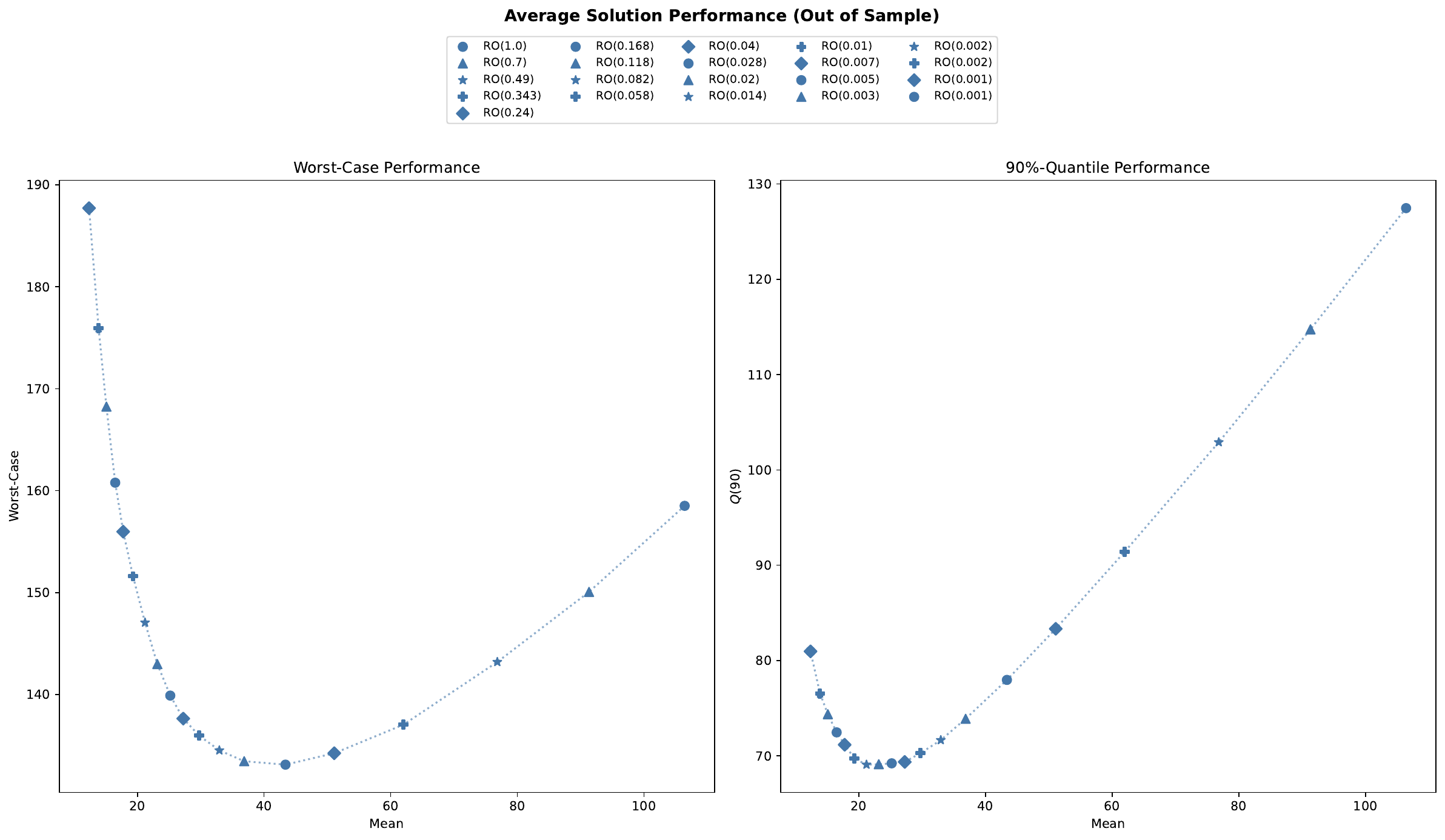}
    \caption{Out-of-sample performance of RO for the minimum knapsack problem with $n=50$ for Gaussian data. Mean vs.\ worst-case (left) and mean vs.\ $90\%$-quantile (right).}
    \label{fig:trade_off_RO}
\end{figure}
The results show that good trade-offs between average and worst-case (or $90\%$-quantile) are achieved for very small $\gamma$ radii. Beyond $\theta = 0.168$, the performance in both average and worst-case (or $90\%$-quantile) already deteriorates as $\theta$ increases. The performance is also highly sensitive to $\gamma$ for small values thereof.

In practice, finding a good value for $\gamma$ as in the previous analysis can be done via an empirical out-of-sample performance test on a validation dataset. However, due to limited or biased data, such empirical tests can be highly imprecise and hence misleading regarding the future performance of the calculated solution. To obtain a more reliable performance guarantee, one can instead consider the guarantee offered by the RO model itself. More precisely, for an optimal solution $x_{RO(\theta)}$ of RO($\theta$) with optimal value opt(RO($\theta$)), we have
\[
x_{RO(\theta)}\tpose p^\star \le \begin{cases}
    \text{opt(RO}(\theta)) & \text{ if } \gamma^\star = d(p_0,p^\star)\leq \theta\gamma_{0.99} \\
    \infty &  \text{ otherwise.}
\end{cases}
\]
These performance guarantees, plotted as a function of $d(p_0, p)$, are shown in Figure \ref{fig:guarantees_RO}. All values are averages over the $25$ instances. For each test sample $p\in \mc T$, the value $d(p_0, p)$ is indicated on the horizontal axis with a cross. The horizontal axis is normalized such that $1$ corresponds to the largest $\gamma_{0.99}$ over all problem instances. Since $\gamma_{0.99}$ is computed from the training data only, test samples may exceed this value significantly.

\begin{figure}
    \centering
\includegraphics[width=0.9\textwidth]{ 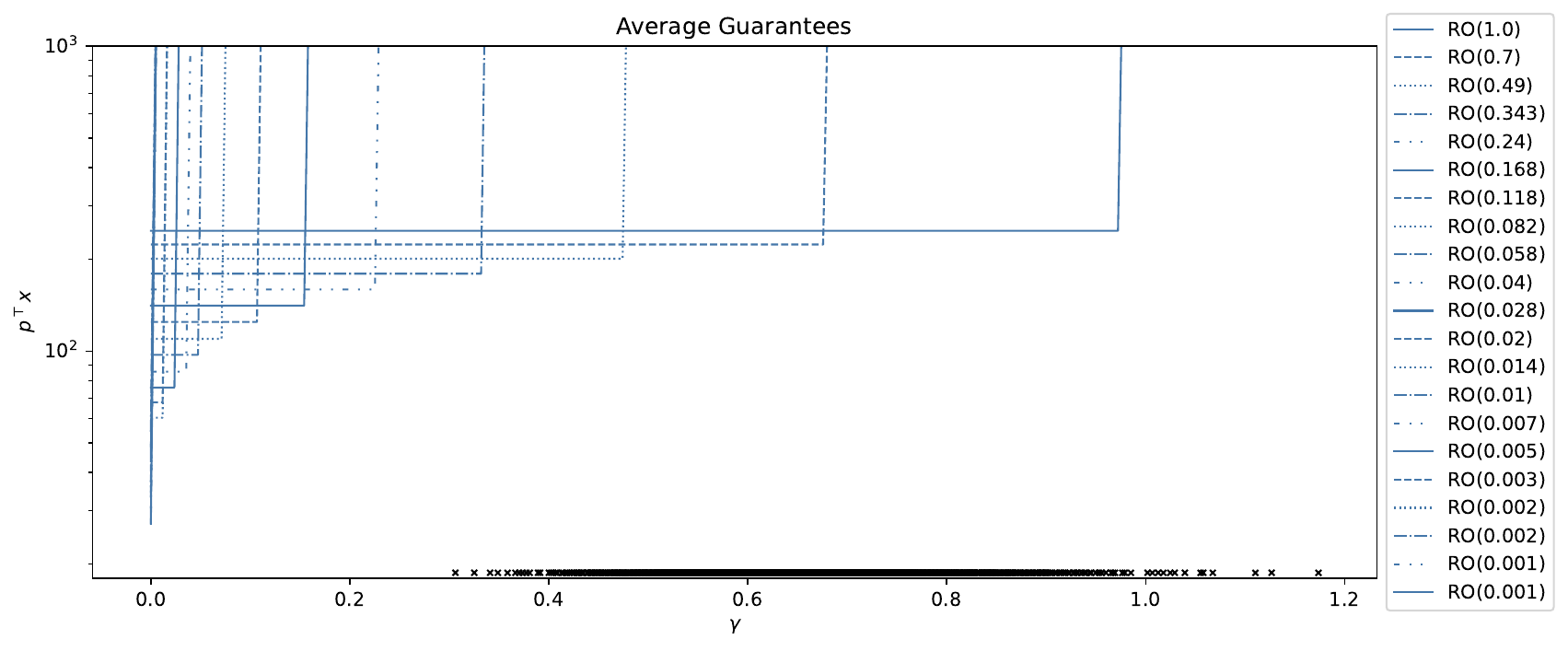}
    \caption{Performance guarantees of RO$(\theta)$ in the minimum knapsack problem with $n=50$ for Gaussian data. Its guarantee is all-or-nothing: it is constant for $d(p_0, p) \leq \theta\gamma_{0.99}$ and is vacuous beyond this threshold.}
    \label{fig:guarantees_RO}
\end{figure}

The results show that the guarantees for the potential values of $\theta$ which lead to good trade-offs ($0\le \theta \le 0.12$) are practically useless since they only hold for scenarios $p$ close to the mean. However, this region does not contain a single test scenario and hence the robust optimization guarantee does not apply to any test scenario. In fact, the uncertainty sets $P_\gamma$ which lead to a good out-of-sample performance do not contain a single training point. This stark observation highlights the impracticality of the guarantees classical robust formulations provide and the need for methods offering broader performance guarantees. Furthermore, the size of the uncertainty set does not provide any intuitive insight into the final performance of the corresponding solution. Indeed, the box-plots in Figure~\ref{fig:boxplots_RO} suggest that the only notable benefit of increasing the uncertainty set is a reduction in the variance of the out-of-sample performance. This is unsurprising: from the reformulation \eqref{eq:reformulation_RO_computations}, larger $\gamma$ values assign greater weight to the standard deviation term in the objective.

\begin{figure}
    \centering
    \includegraphics[width=0.7\textwidth]{ 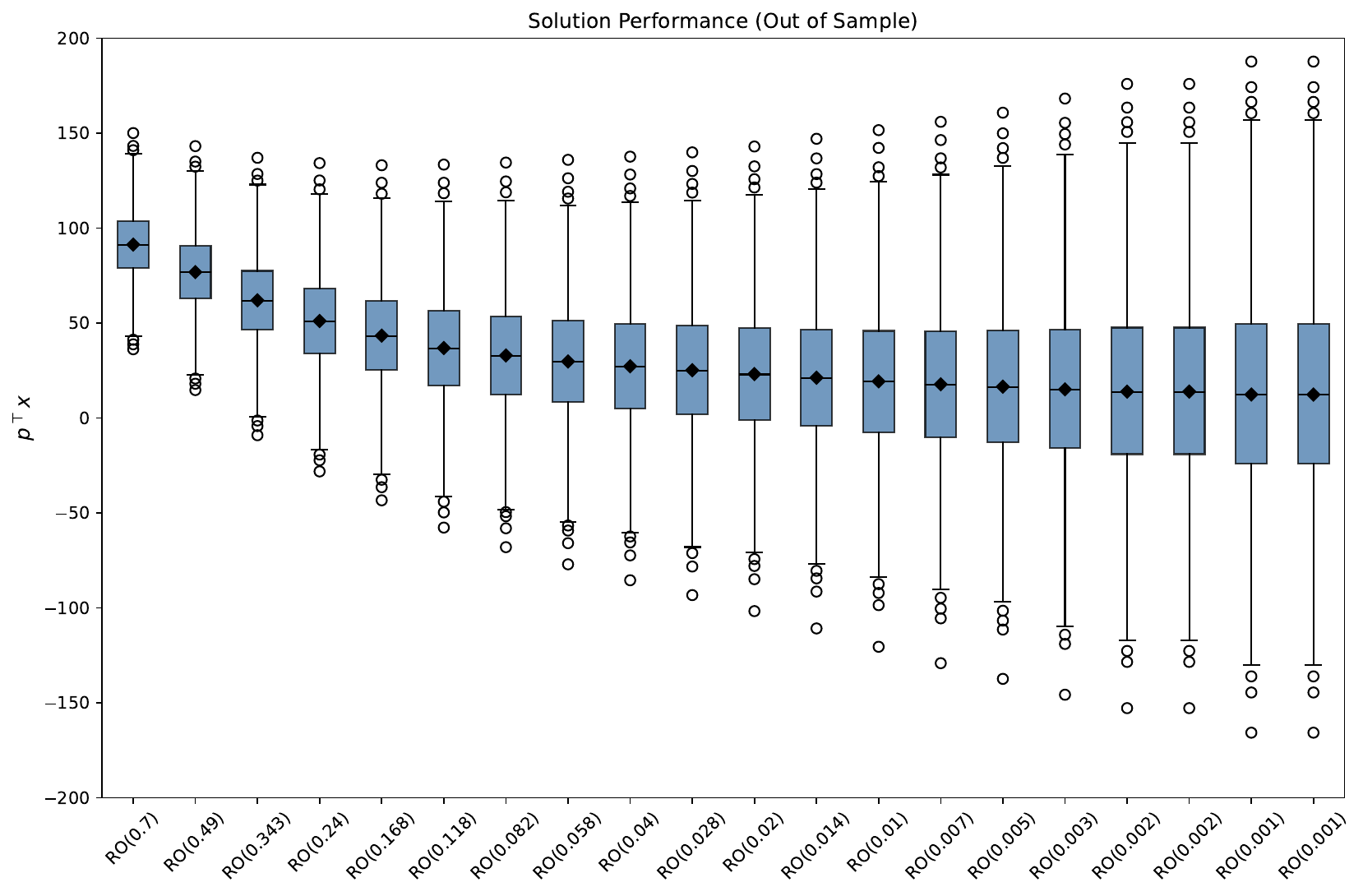}
    \caption{Boxplots of the out-of-sample objective values for the minimum knapsack problem with $n=50$ for Gaussian data. The diamonds denote the mean value.}
    \label{fig:boxplots_RO}
\end{figure}

\subsection{Gaussian Data}
We now turn to the main experiment comparing the different methods mentioned above. For the baseline methods and GARO we consider the following setups, chosen based on the best worst-case performance of each method in preliminary experiments. For RO we consider $\gamma = \theta \gamma_{0.99}$ where $\theta\in \{ 0, 0.02,\ldots ,0.08\}$, and for RO$_d$ and REG we consider values $\theta\in \{ 0.1,\ldots ,0.5\}$. For SAT we consider target values of $\beta\in \{ 1.2, 1.4, \ldots ,2.0\}$, and for GARO we consider exponents $q\in \{0,0.5,\ldots, 2.0 \}$ in the rate function $\phi(\gamma)=(1+\gamma)^q$. We denote the previously described methods as RO($\theta$), RO$_d$($\theta$), SAT($\beta$), REG($\theta$) and GARO($q$).
For each type of data generation process we solve each combination of the $5$ problem instances and the $5$ training data sets, resulting in $25$ instances.

\begin{figure}
    \centering
    \includegraphics[width = 0.9\textwidth]{ 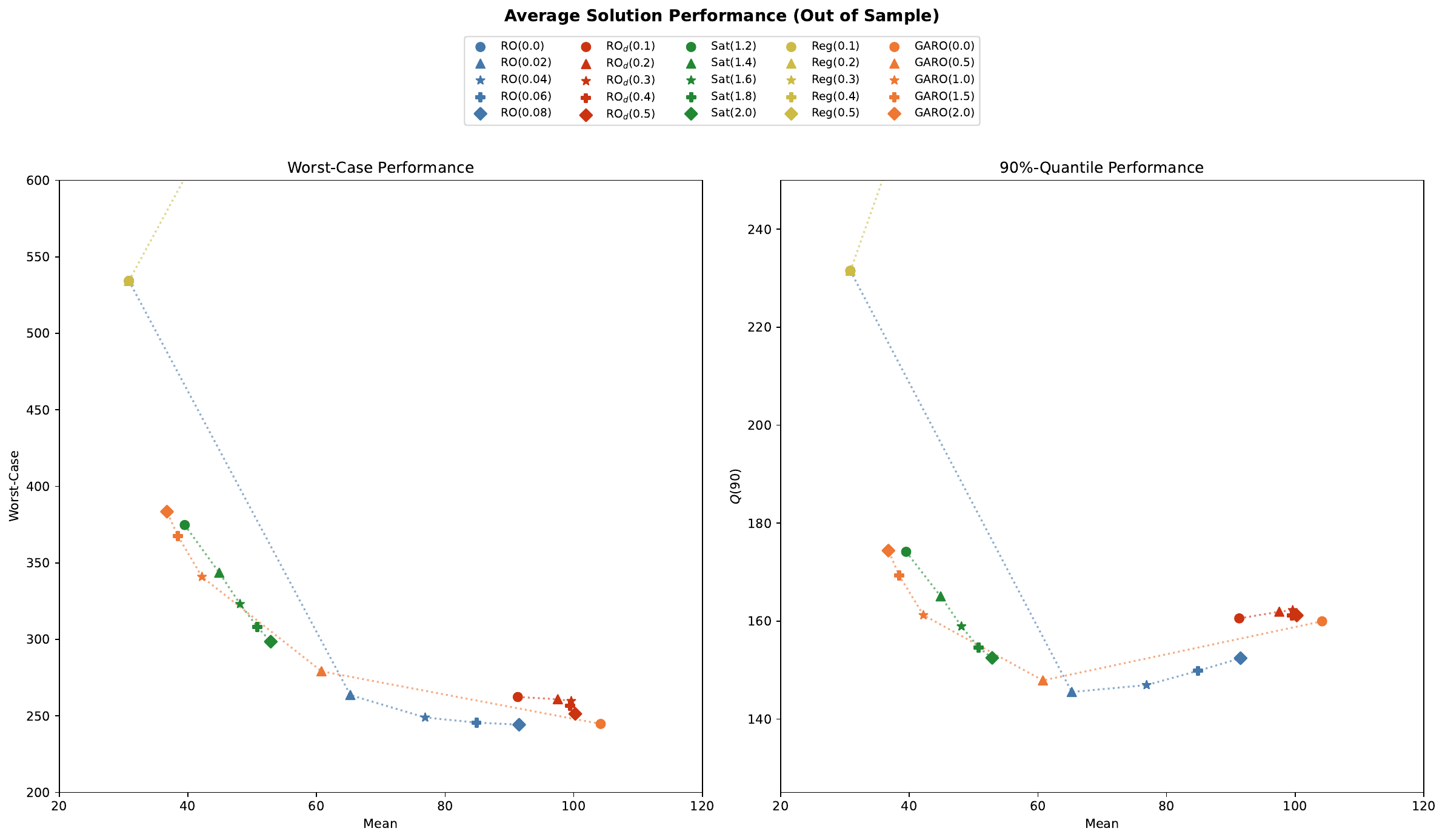}
    \caption{Out-of-sample performance for the minimum knapsack problem with $n=50$ for Gaussian data. Average vs.\ worst-case (left) and average vs.\ $90\%$-quantile (right).}
    \label{fig:gaussian_avg_wc}
\end{figure}

In Figure \ref{fig:gaussian_avg_wc} we show the average out-of-sample cost over all $25$ instances, plotting average vs.\ worst-case (left) and average vs.\ $90\%$-quantile (right) for all methods. Each value is an average over all $25$ instances. The points for REG($0.3$), REG($0.4$) and REG($0.5$) are not shown in the plot since they lead to very large values on average and in worst-case and $90\%$-quantile costs.

The results show that the points of our method for $q\ge 0.5$ provide a good trade-off lying close to the lower left corner of the plot and dominate the solutions of the other methods. The solutions of SAT provide similar values but are sensitive to the choice of the target level. In fact, for different target levels SAT will return one of the RO solutions for a certain $\gamma$-value, and indeed any point of the RO trade-off curve could be generated by setting a corresponding target level (see the discussion in Section \ref{sec:satisficing}).

The solutions of RO and RO$_d$ provide good values for the worst-case costs with large increases in mean cost, where RO$_d$ suffers most. RO$_d$ deteriorates also for the $90\%$-quantile compared to RO. Only the nominal solution provides the best mean performance, at the cost of large worst-case and $90\%$-quantile costs (upper left corner). REG returns the nominal solution for $\theta\in \{ 0, 0.1, 0.2\}$ but leads to very poor performance in mean, worst-case and $90\%$-quantile costs for larger $\theta$ values, which are not visible on the plot.

Some of the studied methods provide performance guarantees on future scenarios after the solution is calculated. We note that RO$_d$($\theta$) only provides guarantees against scenarios in the training set $\mc D(\gamma)$ and hence offers no formal out-of-sample guarantee.
For an optimal solution $x_{GARO(q)}$ of GARO($q$) with corresponding optimal value $\alpha_{GARO(q)}$, the resulting performance guarantee is
\[
x_{GARO(q)}\tpose p^\star \le v^\star_{wc}(\gamma^\star) + \alpha_{GARO(q)}(1+ \gamma^\star)^q
\]
where $\gamma^\star=d(p_0, p^\star)$. Moreover, for an optimal solution $x_{SAT(\beta)}$ of SAT($\beta$) with corresponding optimal value $\alpha_{SAT(\beta)}$, the resulting performance guarantee is
\[
x_{SAT(\beta)}\tpose p^\star \le \beta v^\star_{wc}(0) + \alpha_{SAT(\beta)}\gamma^\star.
\]
Finally, note that for REG we do not naturally obtain a guarantee depending only on $\gamma^\star$.

These performance guarantees for Gaussian data, plotted as a function of $d(p_0, p)$, are shown in Figure \ref{fig:gaussian_guarantees}. For each scenario $p$, the plot shows the guaranteed upper bound on the objective value provided by each method. All values are averages over the different problem and data instances. For each test sample $p\in \mc T$, the value $d(p_0, p)$ is indicated on the horizontal axis with a cross. The horizontal axis is normalized such that $1.0$ corresponds to the largest $\gamma_{0.99}$ over all instances. Since $\gamma_{0.99}$ is computed from the training set, test scenarios can have a normalized distance exceeding $1.0$.
\begin{figure}
    \centering
    \includegraphics[width= 0.9\textwidth]{ 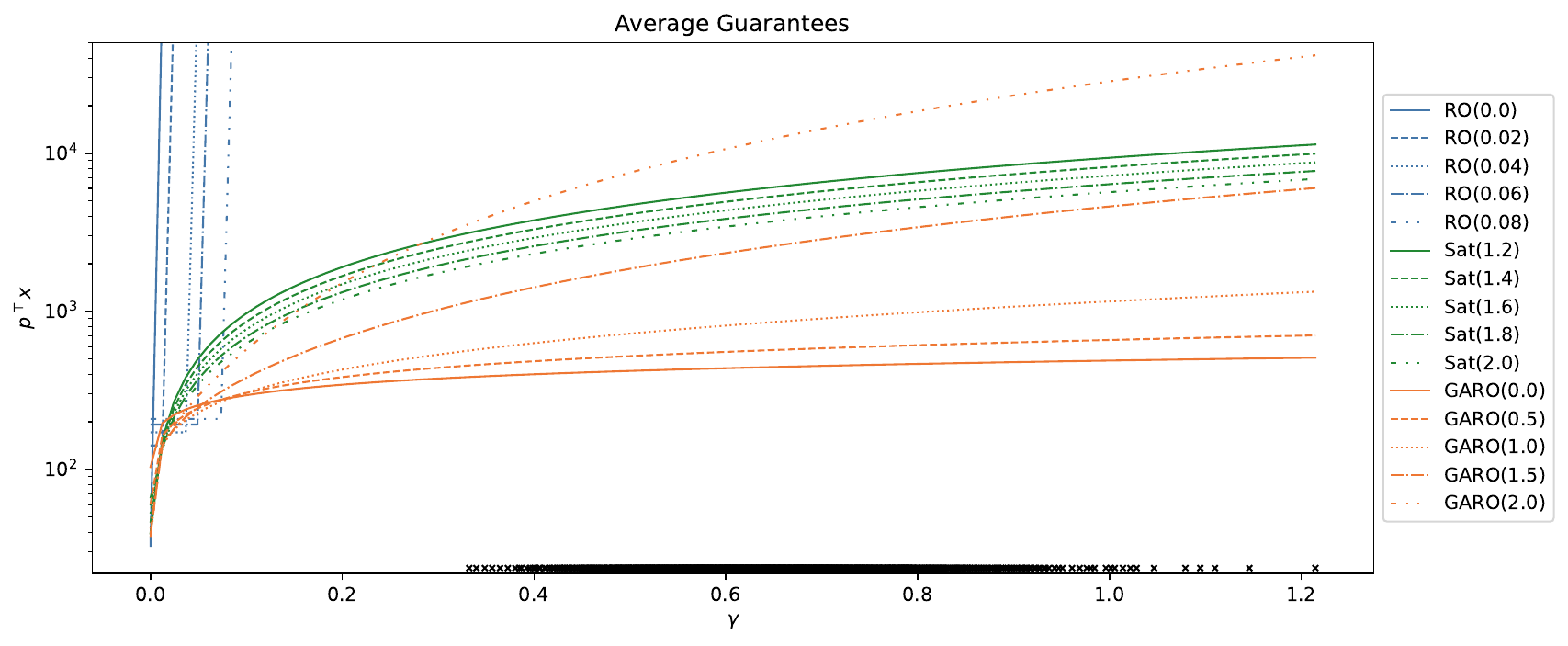}
    \caption{Performance guarantees for the minimum knapsack problem with $n=50$ for Gaussian data.}
    \label{fig:gaussian_guarantees}
\end{figure}

As shown in Figure \ref{fig:gaussian_guarantees}, GARO provides significantly tighter guarantees throughout the test data region for nearly all values of $q$, compared to SAT and RO/RO$_d$. Only for a small region of scenarios very close to the nominal prediction $p_0$ (which did not appear in the test set) are the guarantees of RO/RO$_d$ slightly better than those of GARO. Since the uncertainty sets chosen for RO have a very small radius, the provided guarantees are practically vacuous: under Gaussian sampling the squared Mahalanobis distance of a test point to the training mean is approximately $\chi^2_n$-distributed with mean $n$, so the typical test scenario lies far outside the small uncertainty sets used by RO.

In Figure \ref{fig:gaussian_boxplots} we visualize the distribution of the out-of-sample performance. More precisely, for each method the returned optimal solution $x$ is evaluated on each of the test samples, i.e., we calculate the objective value ${x}\tpose p$ for each $p\in \mc T$ and the distribution of these values is shown in the boxplots. Each value is an average over all $25$ instances.

\begin{figure}
\includegraphics[width=0.9\textwidth]{ 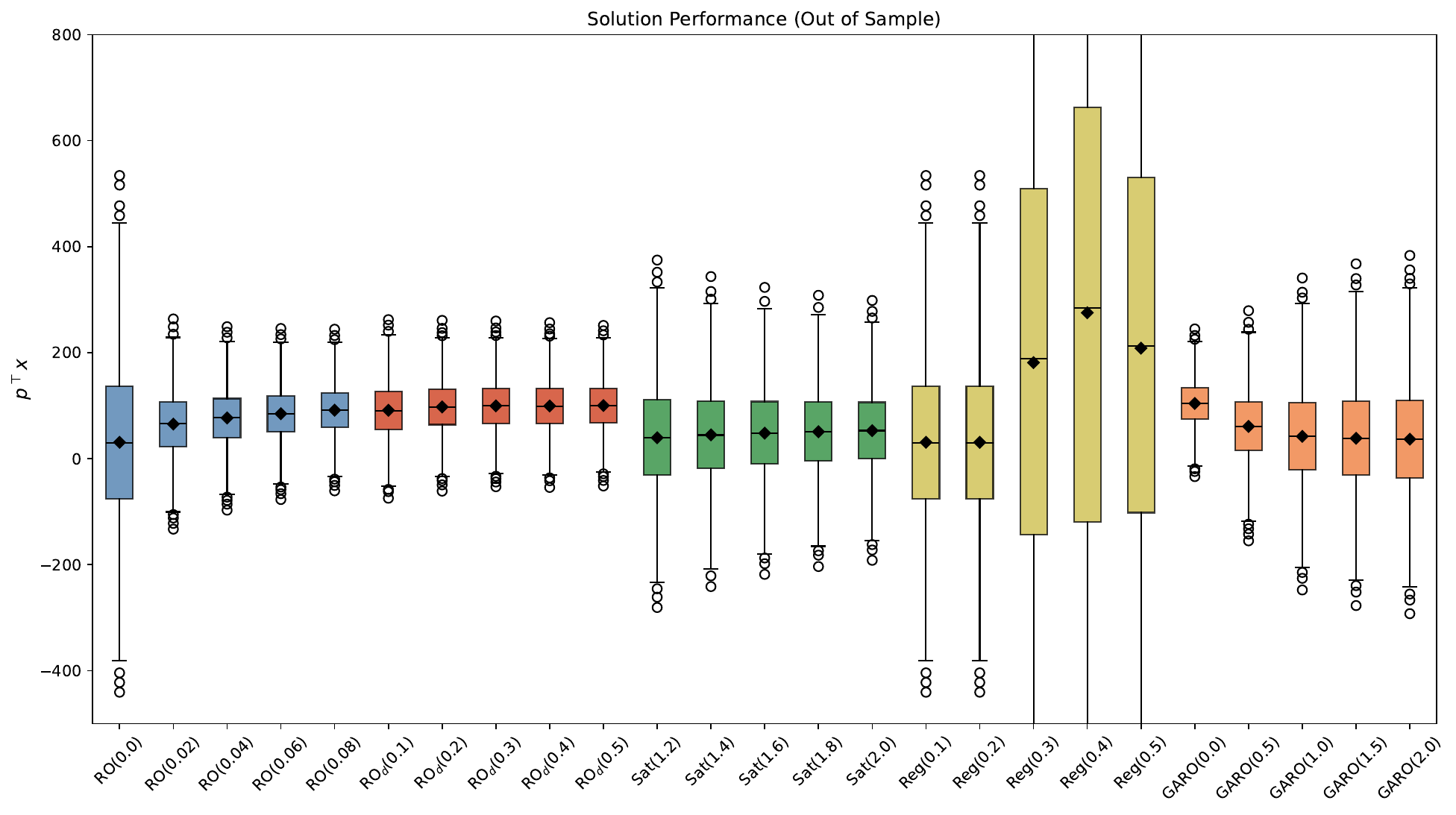}
\caption{Boxplots of the out-of-sample objective values for the minimum knapsack problem with $n=50$ for Gaussian data. The diamonds denote the mean value.}
\label{fig:gaussian_boxplots}
\end{figure}

The results indicate that nearly all methods (except REG and the nominal solution) perform well for the upper quantiles, while for the lower quantiles only REG achieves good values. At the same time, REG exhibits a larger variance compared to the other methods. 

The average runtime is below or around one second for all methods except GARO which has a runtime of around $6$ seconds on average. We note, however, that the REG runtimes reflect the approximate formulation over finitely many training scenarios; solving the exact regret problem would be substantially more expensive due to its inherent nonconvexity.

For Gaussian samples with inverse mean-variance relationship we show the tradeoff plots in Figure \ref{fig:gaussian_inverse_avg_wc}. The results show that the worst-case and $90\%$-quantile costs that GARO achieves are consistently close to the best performance across all methods. At the same time, the mean performance is even better than that of the nominal solution for $q\ge 1$. Meanwhile, SAT also achieves better mean performance than the nominal solution but is worse in terms of worst-case and $90\%$-quantile costs compared to GARO. RO and RO$_d$ perform similarly to GARO in terms of worst-case and $90\%$-quantile costs but underperform in terms of mean performance. The reason GARO performs well here is that items in the knapsack problem with large mean costs have small variance and vice versa. For small uncertainty sets, items with small mean values are therefore beneficial; for larger uncertainty sets, however, these items become less attractive, since their variance eventually exceeds that of items with large means, which are not greatly affected by larger uncertainty sets owing to their small variance.
\begin{figure}
    \centering
    \includegraphics[width=0.9\textwidth]{ 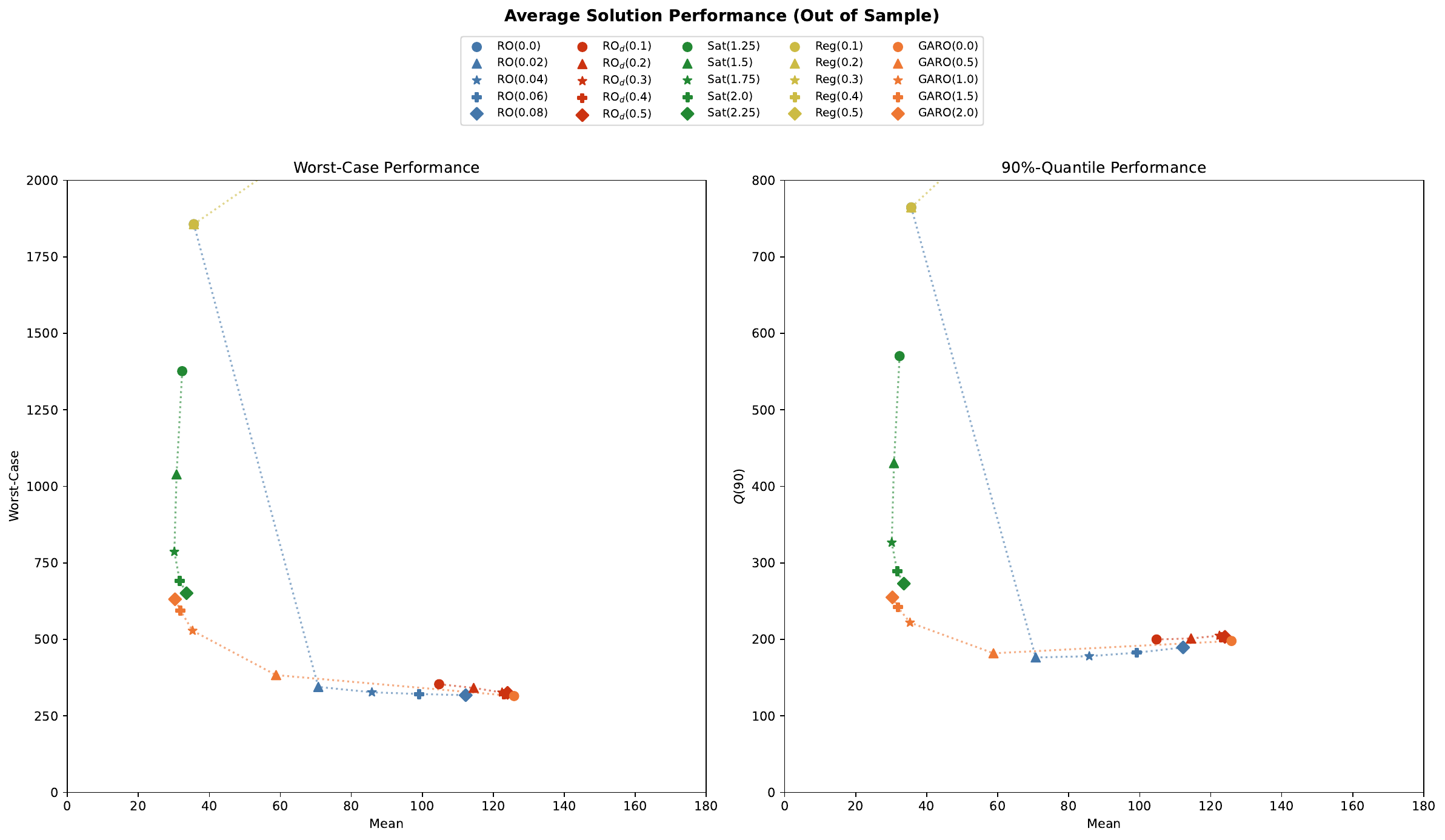} 
    \caption{Out-of-sample performance for the minimum knapsack problem with $n=50$ for Gaussian data with inverse mean-variance relationship. Average vs.\ worst-case (left) and average vs.\ $90\%$-quantile (right).}
    \label{fig:gaussian_inverse_avg_wc}
\end{figure}
We show the performance guarantees in Figure \ref{fig:gaussian_inverse_guarantees}, and the out-of-sample boxplots in Figure \ref{fig:gaussian_inverse_boxplots} in the Appendix which show a similar picture as for classical Gaussian samples.

The runtimes show approximately the same picture as for classical Gaussian samples, but are slightly lower.

\subsection{Heavy-Tail Data}
For the heavy-tail samples we present the trade-off plots in Figure \ref{fig:heavy_tail_avg_wc}. The results indicate that RO, SAT and REG consistently achieve approximately the best worst-case performances, however with significant degradation in mean performance. The GARO method again yields solutions which dominate in terms of worst-case vs.\ mean trade-off. However, its solutions are more sensitive in terms of worst-case behavior, leading to larger worst-case values for $q\ge 1$ compared to most of the benchmark methods, while achieving nearly the best worst-case value for $q=0$. At the same time, GARO with $q\ge 1$ achieves the best mean values together with the nominal solution. REG suffers most in terms of mean performance.

Interestingly, when considering the $90\%$-quantile cost, the trade-off curve changes. The $90\%$-quantile costs are significantly smaller than worst-case costs, owing to the properties of the Pareto distribution, which generates many samples with small values and a few with very large values. GARO for $q\ge 0.5$ performs very well in terms of $90\%$-quantile cost while simultaneously achieving nearly the best mean behavior (only the nominal solution is better). The methods SAT, RO and REG do not achieve good $90\%$-quantile costs, performing even worse than the nominal solution and having significantly larger mean values than GARO. REG achieves the worst $90\%$-quantile costs followed by RO, SAT and RO$_d$.

\begin{figure}
    \centering
    \includegraphics[width=0.9\linewidth]{ 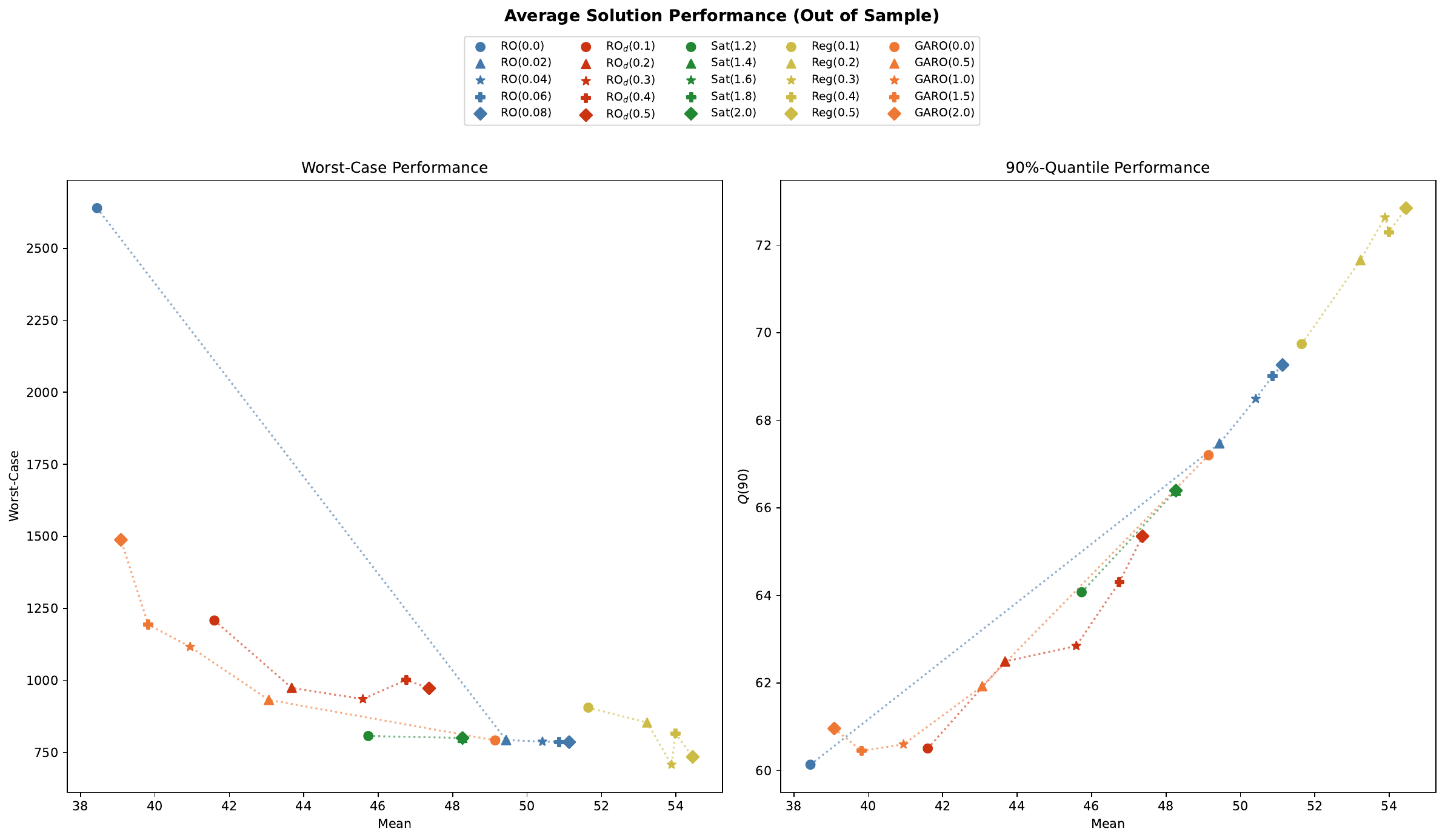}
    \caption{Out-of-sample performance for the minimum knapsack problem with $n=50$ for heavy-tail data. Average vs.\ worst-case (left) and average vs.\ $90\%$-quantile (right).}
    \label{fig:heavy_tail_avg_wc}
\end{figure}

In Figure \ref{fig:heavy_tail_guarantees} we show the guarantees for each method as a function of $\gamma$ for all test scenarios (right) and a zoomed-in view restricted to test scenarios with $\gamma$ close to zero (left). Due to the heavy-tail distribution, most data points lie very close to $p_0$ while a few lie far away. Note that what appears to be piecewise linear functions on the left is an artifact of the plot generation process, since we compute the function value on a grid of $100$ evenly spaced $\gamma$-values connected by straight lines. The results show that GARO($0$) and GARO($0.5$) provide very good guarantees for heavy-tail data, both for scenarios close to $p_0$ and those far from it. The previous trade-off results confirmed that GARO($0.5$) also provided one of the best out-of-sample performances in terms of $90\%$-quantile cost. We show the out-of-sample boxplots in Figure \ref{fig:heavy_tail_boxplots}.

\begin{figure}
    \centering
    \includegraphics[width=0.7\linewidth]{ 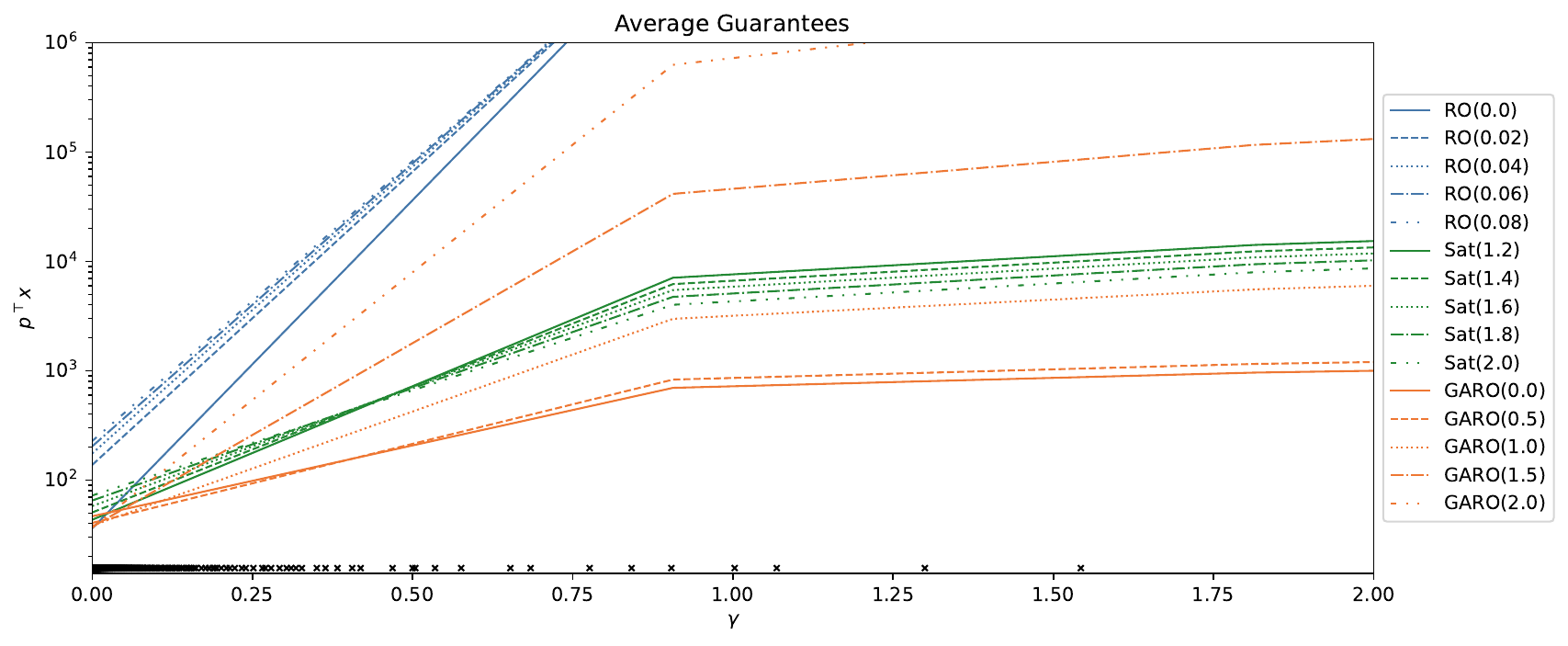}\\[1em]
    \includegraphics[width=0.7\linewidth]{ 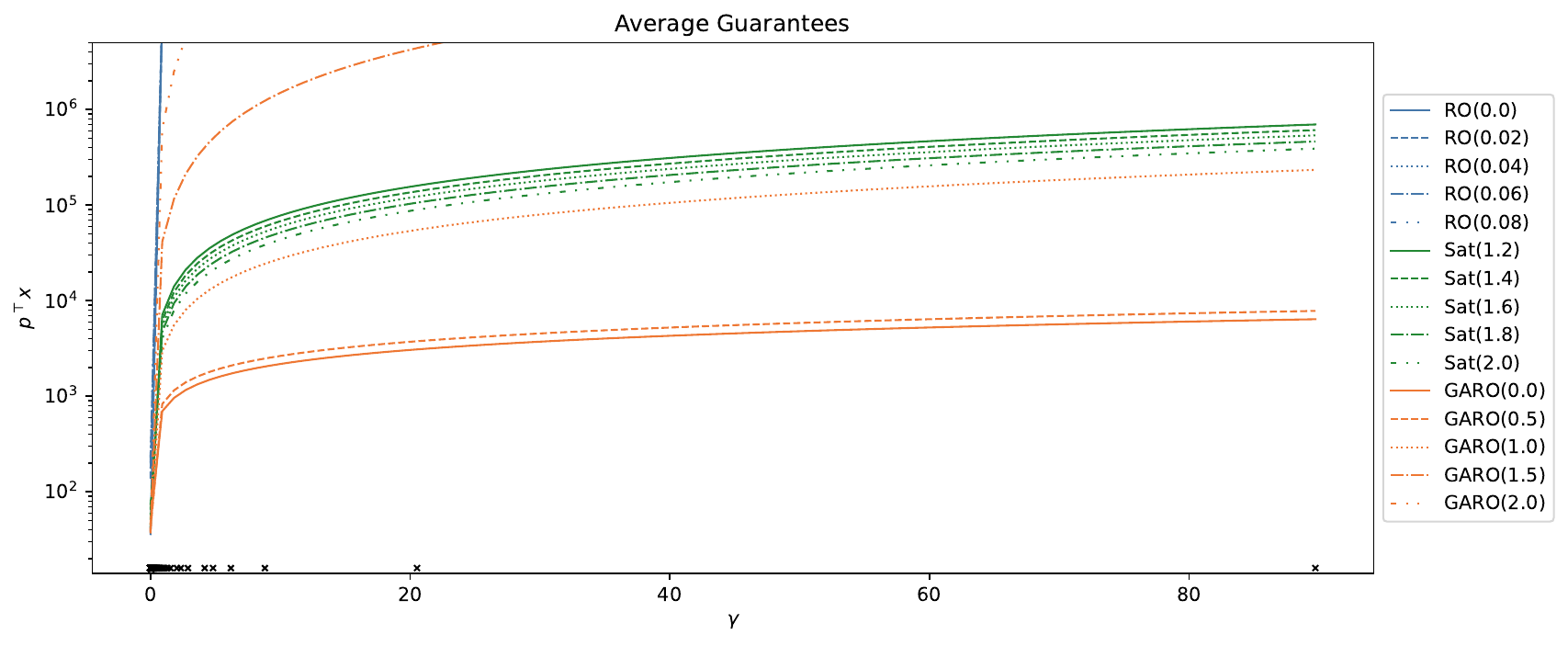}
    \caption{Performance guarantees for the minimum knapsack problem with $n=50$ for heavy-tail data. Zoomed view for small $\gamma$ (top) and full range (bottom).}
    \label{fig:heavy_tail_guarantees}
\end{figure}

\begin{figure}
    \centering
    \includegraphics[width=0.7\linewidth]{ 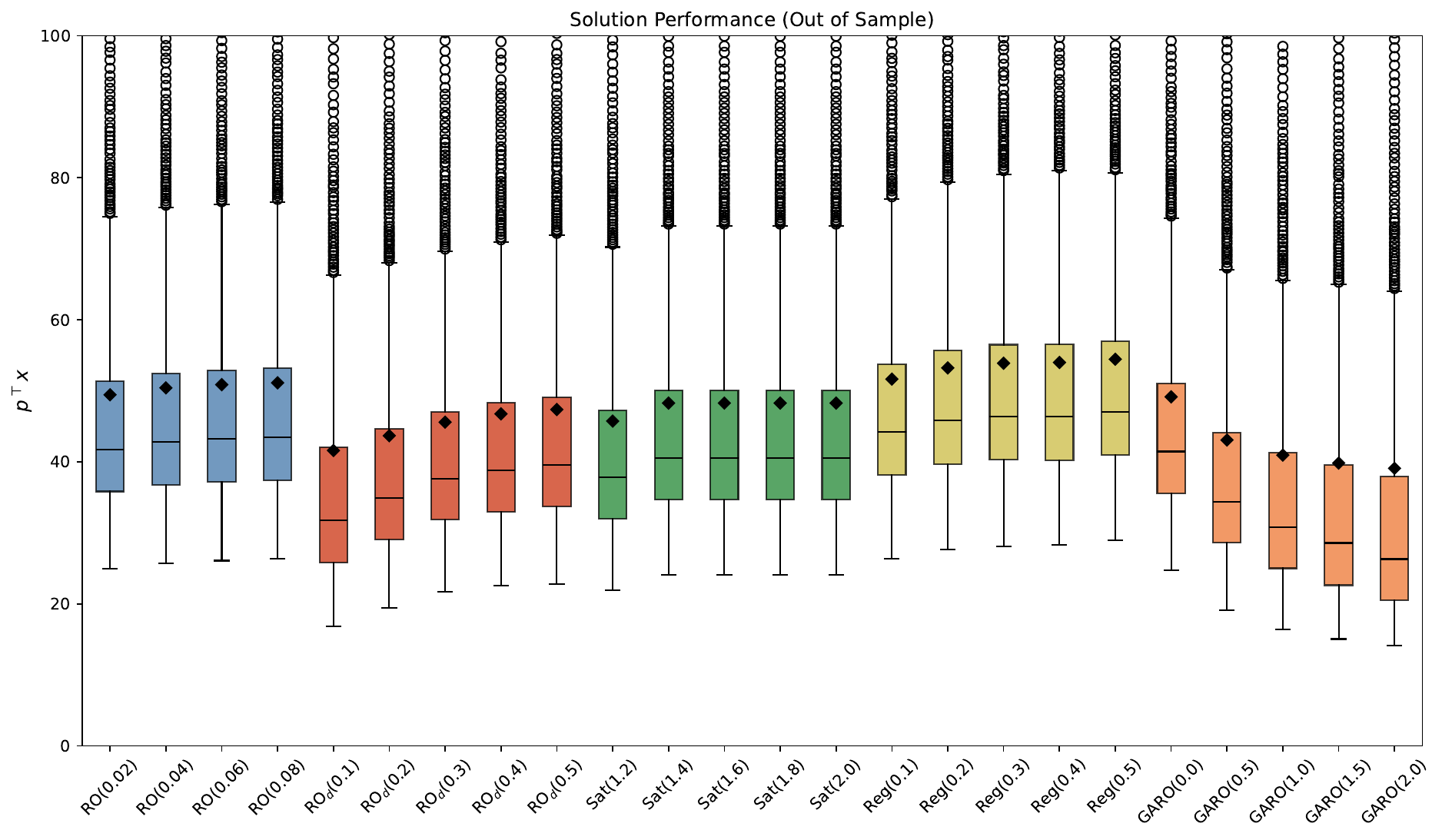}
    \caption{Boxplots for performance over all test scenarios for the knapsack problem with $n=50$ and heavy-tail data.}
    \label{fig:heavy_tail_boxplots}
\end{figure}

The runtimes of RO, RO$_d$, and SAT are again below one second on average, while REG and GARO can have runtimes of up to $10$ seconds where larger uncertainty sets lead to larger runtimes for REG. The runtime for GARO is stable over different $q$.

\section{Conclusion}
We introduced Globalized Adversarial Regret Optimization (GARO), a decision framework that controls adversarial regret uniformly across all prediction error levels.
Unlike classical robust and regret formulations, GARO does not require probabilistic calibration of the uncertainty set and provides meaningful performance guarantees even when predictions lack rigorous error bounds.
We showed that satisficing formulations reduce to robust formulations under typical convexity assumptions, and that GARO with a relative rate function generalizes Lepski's classical adaptation method to downstream decision problems.
On the algorithmic side, we derived exact tractable reformulations for affine and linear-polyhedral settings and proposed a discretization and constraint-generation algorithm with convergence guarantees for the general case.
Experiments on minimum knapsack problems show that GARO yields a favorable trade-off between worst-case and mean out-of-sample performance, and provides stronger global performance guarantees than existing methods.

In this work we focus on single-level problems where the uncertain parameters affect only the objective function.
Extending the globalized adversarial regret perspective to multi-stage problems or to settings where uncertain parameters also affect the constraints is an interesting direction for future research.
Furthermore, not much is known yet about the theoretical complexity of the problem, either in general or for specific problem structures.

\paragraph*{Acknowledgements}
Bart P.G.\ van Parys gratefully acknowledges funding from NWO Vidi grant VI.Vidi.243.021.

\clearpage

\bibliography{main}

\appendix

\section{Supporting Results on the Facility Location Example}

The Wasserstein distance is taken here as
\[
  W(\P, \P_0) \defn \inf \set{\int \norm{\xi-\xi'} \d \mb T(\xi, \xi')}{\mb T\in \mathcal P(\P, \P_0)}
\]
where $\mathcal P(\P, \P_0)$ denotes the set of all probability measures on $\Xi\times \Xi$ with marginals $\P$ and $\P_0$.
The classical Kantorovich-Rubinstein duality theorem states that the Wasserstein distance can be characterized alternatively as
\[
  W(\P, \P_0) = \sup \set{\int \psi(\xi) \d \P(\xi)-\int \psi(\xi) \d \P_0(\xi)}{\psi~{\mathrm{a~1-Lipschitz~function}}}.
\]
The following result states the classical dual characterization of Wasserstein distributional optimization.
\begin{lemma}[{\citet{blanchet2019quantifying}}]
  \label{lemma:weber-dual-representation}
  Let $\int \norm{\xi-\mu}^2 \d\mb P_0(\xi)< \infty$.
  Then, we have
  \[
    \max_{\mb P\in \mc P_\gamma} \int \norm{\xi-\mu}^2 \d \P(\xi) = \inf_{\lambda\geq 0} \lambda \gamma + \int \max_{\xi'\in \Xi} \norm{\xi'-\mu}^2 -\lambda \norm{\xi'-\xi} \d \P_0(\xi).
  \]
\end{lemma}

\begin{lemma}
  \label{lemma:weber-lower-bound}
  We have
  \begin{equation}
    \label{eq:lb_wasserstein}
    v^\star_{wc}(\gamma_0) \geq \max\left(0, 1-\frac{\gamma_0}{2\norm{\Xi}}\right) v(\mb P_0)+ \min\left(\frac{\gamma_0}{2}, \norm{\Xi}\right) \norm{\Xi}.
  \end{equation}
\end{lemma}
\begin{proof}[Proof of Lemma \ref{lemma:weber-lower-bound}.]
  Let $\gamma_0>2\norm{\Xi}$. Then as $W(\mb P, \mb P_0)\leq \gamma_0$ for any $\mb P$ it hence holds that
  \[
    \max_{\mb P\in \mc P_{\gamma_0}} \int \norm{\xi-\mu}^2 \d \P(\xi)=\max_{\mb P} \int \norm{\xi-\mu}^2 \d \P(\xi) = \max_{\xi\in \Xi}\norm{\xi-\mu}^2\geq \norm{\Xi}^2.
  \]
  
  Let $\gamma_0\leq 2 \norm{\Xi}$ and take $\Xi^\star = \arg\max_{\xi\in \Xi} \norm{\xi-\mu}$ which is a nonempty compact set by the extreme value theorem.
  Consider the parameter $\alpha=\gamma_0/ (2\max_{\xi\in \Xi} \norm{\xi-\mu})\leq\gamma_0/(2\norm{\Xi})\leq 1$ and consider the following transport plan
  \[
    \mb T_\alpha(\xi', \xi) = \one{\xi=\xi'} (1-\alpha) + \alpha\,\one{\xi'\in \arg\min_{\xi''\in \Xi^\star} \norm{\xi''-\xi}}.
  \]
  Informally, we transport from each location a fraction $\alpha\in [0,1]$ of the probability mass to a closest point in the set $\Xi^\star$.
  Let us denote with $\mb P_\alpha$ the associated pushforward probability measure.
  From the definition of the Wasserstein distance is trivial to observe that $W(\mb P_{\alpha}, \mb P_0) \leq \max_{\xi\in \Xi} 2\norm{\xi-\mu} \alpha\leq \gamma_0$.
  Furthermore, observe that
  \[
    \int \norm{\xi-\mu}^2 \d \P_\alpha(\xi)= (1-\alpha) \int \norm{\xi-\mu}^2 \d \P_0(\xi)+\alpha \max_{\xi\in \Xi} \norm{\xi-\mu}^2.
  \]
  Hence,
  \begin{align*}
    \max_{\mb P\in \mc P_{\gamma_0}} \int \norm{\xi-\mu}^2 \d \P(\xi) \geq & \int \norm{\xi-\mu}^2 \d \mb P_{\alpha}(\xi)\\
    \geq  & \left(1-\frac{\gamma_0}{\max_{\xi\in \Xi} \norm{\xi-\mu}}\right) \int \norm{\xi-\mu}^2 \d \P_0(\xi) + \frac{\gamma_0}{2} \max_{\xi\in \Xi} \norm{\xi-\mu}
  \end{align*}
  from which the claim follows immediately.
\end{proof}

\begin{lemma}
  \label{lemma:bound-M2}
  Let $M(\gamma_0) \defn \set{\mu_{nom}(\mb P)}{\mb P \in \mc P_{\gamma_0}}$ where $\mu_{nom}(\mb P) \defn \arg\min_{\mu}\int\norm{\xi-\mu}^2 \d\mb P(\xi)$ is the unique minimizer.
  Then $M(\gamma_0) \subseteq B[\mu_{nom}(\mb P_0), \gamma_0]$.
\end{lemma}
\begin{proof}[Proof of Lemma \ref{lemma:bound-M2}]
  Recall that the support function of the set $M(\gamma_0)$ is defined as
  \begin{equation}
    h_{M(\gamma_0)}(c) \defn \, \max_{\mu} \set{c\tpose \mu}{\mu \in M(\gamma_0)} = \norm{c} \sup \set{\int \frac{c\tpose \xi}{\norm{c}} \, \d \P(\xi)}{\P\in \mc P_{\gamma_0}}.
  \end{equation}
  As the function $\tfrac{c\tpose \xi}{\norm{c}}$ is $1$-Lipschitz for any $c$, it follows from the Kantorovich-Rubinstein duality that
  \begin{equation}
    h_{M(\gamma_0)}(c) \leq  \, c\tpose \int \xi \d \P_0(\xi) + \norm{c} \gamma_0 = h_{B}(c)
  \end{equation}
  where $B$ denotes here a closed norm ball with center $\mu_{nom}(\mb P_0) = \int \xi \d \P_0(\xi)$ of radius $\gamma_0$. As the inequality holds for any $c$, it follows that $M(\gamma_0) \subseteq \cl(\conv(M(\gamma_0))) \subseteq B$ from which the claim follows.
\end{proof}

\section{Supporting Results on Satisficing}
\label{sec:supporting-sat}

\begin{proposition}
  \label{prop:sat-eq-rob}
  Let $f(x, p)$ be convex and lower semicontinuous in $x$ and strictly concave and upper semicontinuous in $p$ with $X$ compact convex and $P_\infty$ compact convex. Assume that $p\mapsto d(p_0, p)$ is a convex lower semicontinuous distance function and \eqref{eq:rob-grc} is feasible and satisfies condition \eqref{eq:sat-conditions}.
  Then, we have
  \[
    X_{sat} \subseteq \textstyle\cup_{\gamma \in [0, \gamma_{\max}]} X_{rob}(\gamma), \quad \gamma_{\max} \defn \textstyle\max_{p\in P_\infty} d(p_0, p).
  \]
\end{proposition}
\begin{proof}
  First, it is straightforward to observe that problem \eqref{eq:rob-grc} is here equivalent to
  \[
    \begin{array}{rl}
      \min_{x\in X, \, \alpha\geq 0} & \alpha \\
      \st & \tfrac{(v_{wc}(x, \gamma) - f_0)}{\gamma} \leq \alpha \quad \forall \gamma\geq 0.
    \end{array}
  \]
  As \eqref{eq:rob-grc} admits a feasible point $(x_s, \alpha_s)$, the previous optimization problem satisfies the Slater constraint qualification condition at $(x_s, \alpha_s+1)$. Furthermore, its minimum is attained as $v_{wc}(x, \gamma)$ is lower semicontinuous in $x$ for all $\gamma\geq 0$ and $X$ is a compact set and $\alpha$ can be constrained without loss of generality to the compact set $[0, \alpha_s]$. Consider now an arbitrary minimizer $(x_{sat}, \alpha_{sat})$.

  \medskip\noindent\textbf{Case $\alpha_{sat}=0$.} We must have here that $f_0\geq \min_{x\in X} \max_{p\in P_\infty} f(x, p)$ for \eqref{eq:rob-grc} to be feasible and hence from \eqref{eq:sat-conditions} we get $f_0= \min_{x\in X} \max_{p\in P_\infty} f(x, p)$ and hence $x_{sat}\in X_{rob}(\gamma_{\max})$.

  \medskip\noindent\textbf{Case $\alpha_{sat}>0$.} The constraint $\alpha\geq 0$ is nonbinding and can be relaxed and hence \eqref{eq:rob-grc} is here equivalent to
  \[
    \begin{array}{rl}
      \min_{x\in X, \, \alpha\in\Re} & \alpha \\
      \st & \tfrac{(v_{wc}(x, \gamma) - f_0)}{\gamma} \leq \alpha \quad \forall \gamma\geq 0.
    \end{array}
  \]
  We remark that as $f(x, p)$ is convex in $x$ this optimization problem is convex by construction.
  Introduce the Lagrangian
  \[
    L(x, \alpha, \mu) = \alpha + \int_{\Re_+} \frac{(v_{wc}(x, \gamma) - f_0)}{\gamma}- \alpha \, \d \mu(\gamma)
  \]
  where $\mu\in \mc M_{+}(\Re_+)$ is a positive measure supported on $\Re_+$.
  By \citet[Theorem 8.3.1]{luenberger_optimization_nodate} Slater's constraint qualification implies that strong duality holds and the dual maximum is attained at some nonnegative measure $\mu_{sat}$.
  Combined with primal attainment of $(x_{sat}, \alpha_{sat})$ \citet[Theorem 8.4.1]{luenberger_optimization_nodate} gives that the triple $(x_{sat}, \alpha_{sat}, \mu_{sat})$ is a saddle point of the Lagrangian $L$, i.e.,
  \[
    \max_{\mu\in \mc M_{+}(\Re_+)} L(x_{sat}, \alpha_{sat}, \mu) \leq L(x_{sat}, \alpha_{sat}, \mu_{sat}) = \alpha_{sat} \leq \min_{x\in X, \, \alpha\in\Re} L(x, \alpha, \mu_{sat}).
  \]

  We now indicate that $\mu_{sat}$ must be a degenerate probability distribution.
  From the saddle point condition, we first observe that $\alpha_{sat}$ minimizes $L(x_{sat}, \alpha,  \mu_{sat})$ over $\alpha\in\Re$.
  Given that
  \[
    L(x_{sat}, \alpha, \mu_{sat}) = \alpha \left(1- \int_{\Re_+}  \d \mu_{sat}(\gamma)\right)+ \int_{\Re_+} \frac{(v_{wc}(x_{sat}, \gamma) - f_0)}{\gamma} \d  \mu_{sat}(\gamma)
  \]
  we need that $\int_{\Re_+} \d  \mu_{sat}(\gamma)=1$ and hence it follows that $\mu_{sat}$ is a probability measure on $\Re_+$.
  Second, we have that
  \(
  \alpha_{sat}  = L(x_{sat}, \alpha_{sat}, \mu_{sat})
  \)
  which implies the complementarity condition
  \[
    \int_{\Re_+} \frac{(v_{wc}(x_{sat}, \gamma) - f_0)}{\gamma} -\alpha_{sat}\, \d \mu_{sat}(\gamma) = 0.
  \]
  Observe that as $f(x, p)$ is concave in $p$ and $d(p_0, p)$ is convex in $p$ it follows that $v_{wc}(x, \gamma)$ is an increasing concave function for any $x \in X$.
  From primal feasibility it follows that $(v_{wc}(x_{sat}, \gamma) - f_0)/{\gamma} - \alpha_{sat}\leq 0$ and hence $\mu_{sat}$ must be supported on a convex set $C = \set{\gamma\geq 0}{v_{wc}(x_{sat}, \gamma) = f_0 + \alpha_{sat}\gamma }= \set{\gamma\geq 0}{v_{wc}(x_{sat}, \gamma) \geq f_0 + \alpha_{sat}\gamma }$ where the worst-case function is affine and strictly increasing as we have here $\alpha_{sat}>0$. Hence, we must have that $C\subseteq [0, \max_{p\in P_\infty} d(p_0, p)]$ as for $\gamma\geq \max_{p\in P_\infty} d(p_0, p)$ the worst-case function takes on a constant value $v_{wc}(x_{sat}, \gamma) = \max_{p\in P_\infty} f(x_{sat}, p)$.

  However, the worst-case cost function $\gamma\mapsto v_{wc}(x_{sat}, \gamma)$ is strictly concave following Lemma \ref{lemma:strict-concavity} and hence $C=\{\gamma_0'\}$ must be a singleton. For the sake of contradiction, let the convex set $C$ contain two distinct points $\gamma_a$ and $\gamma_b$. Then for any $t\in (0,1)$ we have
  \(
    v_{wc}(x_{sat}, t\gamma_a + (1-t)\gamma_b) = f_0 + \alpha_{sat}(t\gamma_a + (1-t)\gamma_b) = t (f_0 + \alpha_{sat}\gamma_a)+(1-t)(f_0 +\alpha_{sat}\gamma_b) = tv_{wc}(x_{sat}, \gamma_a) +(1-t)v_{wc}(x_{sat}, \gamma_b)
  \)
  which contradicts the fact that $v_{wc}(x_{sat}, \gamma)$ is strictly concave.

  From the saddle point condition, $x_{sat}$ finally minimizes $L(x, \alpha_{sat}, \mu_{sat}) = \int_{\Re_+} \tfrac{(v_{wc}(x, \gamma) - f_0)}{\gamma} \d \mu_{sat}(\gamma) = \tfrac{(v_{wc}(x, \gamma_0') - f_0)}{\gamma_0'}$ over $x\in X$ and hence $x_{sat}\in X_{rob}(\gamma_0')$.
\end{proof}

\begin{lemma}
  \label{lemma:strict-concavity}
Let $P$ be a compact convex set and $p_0 \in P$, let $f(x, p)$ be strictly concave and upper semicontinuous in $p$, and let $p\mapsto d(p_0, p)$ be a convex lower semicontinuous function. The function
\[
    \gamma\mapsto v_{wc}(x, \gamma) = \max \set{f(x, p)}{p \in P,~ d(p_0, p) \leq \gamma}
\]
is strictly concave on any set $C$ where it is strictly increasing.
\end{lemma}

\begin{proof}
The feasible set $F(\gamma) = \{p \in P : d(p_0, p) \leq \gamma\}$ is a nonempty, compact, and convex set for any $\gamma \in C$. Since $f(x, p)$ is upper semicontinuous in $p$, the maximum defining $v_{wc}(x, \gamma)$ is attained by the Weierstrass extreme value theorem, and since $p\mapsto f(x, p)$ is strictly concave this maximizer $p^\star(\gamma)$ is also unique.

\medskip\noindent\textbf{Concavity.} Consider $\gamma_1$ and $\gamma_2$ and $\lambda \in (0,1)$, and set $\gamma_\lambda = \lambda \gamma_1 + (1-\lambda)\gamma_2$. Let $p_i^\star = p^\star(\gamma_i)$ and $p_\lambda = \lambda p_1^\star + (1-\lambda) p_2^\star$. Since $P$ is convex, $p_\lambda \in P$. Since $p\mapsto d(p_0, p)$ is convex,
\[
    d(p_0, p_\lambda) \leq \lambda\, d(p_0, p_1^\star) + (1-\lambda)\, d(p_0, p_2^\star) \leq \lambda \gamma_1 + (1-\lambda)\gamma_2 = \gamma_\lambda,
\]
so $p_\lambda \in F(\gamma_\lambda)$. By feasibility and concavity of $f$,
\[
    v_{wc}(x, \gamma_\lambda) \geq f(x, p_\lambda) \geq \lambda f(x, p_1^\star) + (1-\lambda) f(x, p_2^\star) = \lambda v_{wc}(x, \gamma_1) + (1-\lambda) v_{wc}(x, \gamma_2),
\]
establishing concavity of $v_{wc}$.

\medskip\noindent\textbf{Strict concavity on $C$.} It suffices to show that for $\gamma_1 < \gamma_2$ in $C$, the optimizers satisfy $p_1^\star \neq p_2^\star$, since then the strict concavity of $f$ gives $f(x, p_\lambda) > \lambda f(x, p_1^\star) + (1-\lambda)f(x, p_2^\star)$, and the argument above yields strict inequality $v_{wc}(x, \gamma_\lambda) > \lambda v_{wc}(x, \gamma_1) + (1-\lambda)v_{wc}(x, \gamma_2)$. Suppose for contradiction that $p_1^\star = p_2^\star = p^\star$. Then $d(p_0, p^\star) \leq \gamma_1$, so $p^\star$ is feasible for every $\gamma' \in [\gamma_1, \gamma_2]$, giving $v_{wc}(x, \gamma') \geq f(x, p^\star) = v_{wc}(x, \gamma_1)$. But $p^\star$ is also optimal at $\gamma_2$, so for any $\gamma' \leq \gamma_2$ we have $v_{wc}(x, \gamma') \leq v_{wc}(x, \gamma_2) = f(x, p^\star)= v_{wc}(x, \gamma_1)$. Hence $\gamma\mapsto v_{wc}(x, \gamma)$ is constant on $[\gamma_1, \gamma_2]$, contradicting the assumption that $\gamma\mapsto v_{wc}(x, \gamma)$ is strictly increasing on $C$.
\end{proof}

We remark that the assumption that the cost function $f(x, p)$ is strictly concave in $p$ can be relaxed, resulting in a slightly weaker result that there is always a satisficing solution which admits a robust interpretation. The following is a proof of Theorem~\ref{prop:general-convex}, which proceeds via a regularization argument.

\begin{proof}[Proof of Theorem~\ref{prop:general-convex}]
For each $\varepsilon > 0$ define
\[
    f^\varepsilon(x, p) := f(x, p) - \varepsilon\, \psi(p).
\]
Since $f(x,\cdot)$ is concave and $-\varepsilon\psi$ is strictly concave, $f^\varepsilon(x,\cdot)$
is strictly concave for every $\varepsilon > 0$. Define
\[
    v_{wc}^\varepsilon(x, \gamma) := \max_{p \in P_\gamma} f^\varepsilon(x,p), \qquad
    v_{wc}(x, \gamma) := \max_{p \in P_\gamma} f(x,p).
\]

\medskip\noindent\textbf{Step 1 (Uniform approximation and compactness).}
Let $M := \max_{p \in P_\infty} \psi(p) < \infty$. Let $p^\varepsilon \in \arg\max_{p \in P_\gamma}
f^\varepsilon(x,p)$ and $p^\star \in \arg\max_{p \in P_\gamma} f(x,p)$. Using optimality of
$p^\varepsilon$ for $f^\varepsilon$ and $p^\star$ for $f$ respectively:
\begin{align*}
v_{wc}^\varepsilon(x,\gamma)
  &= f(x,p^\varepsilon) - \varepsilon\psi(p^\varepsilon)
   \leq f(x,p^\star) - \varepsilon\psi(p^\varepsilon)
   \leq v_{wc}(x,\gamma), \\
v_{wc}^\varepsilon(x,\gamma)
  &= f(x,p^\varepsilon) - \varepsilon\psi(p^\varepsilon)
   \geq f(x,p^\star) - \varepsilon\psi(p^\star)
   \geq v_{wc}(x,\gamma) - \varepsilon M,
\end{align*}
where the first inequality in the first line uses optimality of $p^\star$, the second inequality in the first line uses $\psi \geq 0$ and the first inequality in the
second line uses optimality of $p^\varepsilon$. Hence
\[
    \max_{x \in X,\, \gamma \in [0,\gamma_{\max}]}
    \bigl|v_{wc}^\varepsilon(x,\gamma) - v_{wc}(x,\gamma)\bigr|
    \;\leq\; \varepsilon M \;\to\; 0
    \quad \text{as } \varepsilon \to 0.
\]
Since $v_{wc}^\varepsilon \leq v_{wc}$, any $(x,\alpha)$ feasible for $\mathrm{SAT}$ is also
feasible for $\mathrm{SAT}^\varepsilon$, so $\alpha^\varepsilon \leq \alpha_{sat}$ for all
$\varepsilon > 0$. By compactness of $X \times [0,\alpha_{sat}] \times [0,\gamma_{\max}]$, there
exists a sequence $\varepsilon_k \downarrow 0$ along which
\[
    (x^{\varepsilon_k},\, \alpha^{\varepsilon_k},\, \gamma^{\varepsilon_k})
    \;\longrightarrow\;
    (\hat{x},\, \hat{\alpha},\, \hat{\gamma})
    \;\in\; X \times [0,\alpha_{sat}] \times [0,\gamma_{\max}].
\]

\medskip\noindent\textbf{Step 2 (Continuity of $v_{wc}$).}
We verify the conditions of Berge's maximum theorem \citep[p.\ 116]{berge1963topological} applied to
$v_{wc}(x,\gamma) = \max_{p\in P_\gamma} f(x,p)$ with parameter $(x,\gamma)\in X\times[0,\gamma_{\max}]$.
First, each $P_\gamma = \{p\in P_\infty : d(p_0,p)\leq\gamma\}$ is compact, as a closed sublevel set
of the lower semicontinuous function $d(p_0,\cdot)$ intersected with the compact set $P_\infty$, and nonempty
since $p_0\in P_\infty$ satisfies $d(p_0,p_0)=0\leq\gamma$.
The correspondence $\gamma\mapsto P_\gamma$ is upper hemicontinuous: if $\gamma_n\to\gamma$ and
$p_n\in P_{\gamma_n}$ with $p_n\to p$ for $n\to\infty$, then lower semicontinuity of $d$ gives $d(p_0,p)\leq \lim_{n\to\infty} d(p_0,p_n)
\leq\lim_{n\to\infty}\gamma_n=\gamma$, so $p\in P_\gamma$.
It is lower hemicontinuous: given $p\in P_\gamma$ and $\gamma_n\to\gamma$ for $n\to\infty$, set
$t_n = \min(1,\gamma_n/\gamma)$ (interpreting $t_n=1$ when $\gamma=0$) and
$p_n = t_n p + (1-t_n)p_0\in P_\infty$ by convexity of $P_\infty$.
Convexity of $d(p_0,\cdot)$ gives $d(p_0,p_n)\leq t_n\,d(p_0,p)\leq t_n\gamma\leq\gamma_n$,
so $p_n\in P_{\gamma_n}$, and $p_n\to p$ since $t_n\to 1$.
Since $f$ is jointly continuous, Berge's maximum theorem implies that $v_{wc}(x,\gamma)$ is
jointly continuous in $(x,\gamma)$ on $X\times[0,\gamma_{\max}]$.

\medskip\noindent\textbf{Step 3 (Perturbed satisficing problem).}
Since $f^\varepsilon(x,\cdot)$ is strictly concave and upper semicontinuous, Proposition~\ref{prop:sat-eq-rob} applies to $\mathrm{SAT}^\varepsilon$ and yields an optimal solution
$(x^\varepsilon, \alpha^\varepsilon)$ with $x^\varepsilon \in X_{rob}^\varepsilon(\gamma^\varepsilon)$
for some $\gamma^\varepsilon \in [0, \gamma_{\max}]$, i.e.,
\[
    v_{wc}^{\varepsilon}(x^{\varepsilon}, \gamma^{\varepsilon})
    \;\leq\;
    v_{wc}^{\varepsilon}(x, \gamma^{\varepsilon})
    \quad \forall\, x \in X.
\]

\medskip\noindent\textbf{Step 4 ($\hat{x}$ is a robust solution at $\hat{\gamma}$).}
Fix $x \in X$. For the right-hand side of the inequality in Step 3,
\[
    \bigl|v_{wc}^{\varepsilon_k}(x, \gamma^{\varepsilon_k}) - v_{wc}(x, \hat{\gamma})\bigr|
    \leq
    \underbrace{\bigl|v_{wc}^{\varepsilon_k}(x, \gamma^{\varepsilon_k}) - v_{wc}(x, \gamma^{\varepsilon_k})\bigr|}_{\leq\, \varepsilon_k M \,\to\, 0}
    +
    \underbrace{\bigl|v_{wc}(x, \gamma^{\varepsilon_k}) - v_{wc}(x, \hat{\gamma})\bigr|}_{\to\, 0 \text{ by Step 2}},
\]
and for the left-hand side,
\[
    \bigl|v_{wc}^{\varepsilon_k}(x^{\varepsilon_k}, \gamma^{\varepsilon_k}) - v_{wc}(\hat{x}, \hat{\gamma})\bigr|
    \leq
    \underbrace{\bigl|v_{wc}^{\varepsilon_k}(x^{\varepsilon_k}, \gamma^{\varepsilon_k}) - v_{wc}(x^{\varepsilon_k}, \gamma^{\varepsilon_k})\bigr|}_{\leq\, \varepsilon_k M \,\to\, 0}
    +
    \underbrace{\bigl|v_{wc}(x^{\varepsilon_k}, \gamma^{\varepsilon_k}) - v_{wc}(\hat{x}, \hat{\gamma})\bigr|}_{\to\, 0 \text{ by Step 2}}.
\]
Passing to the limit gives $v_{wc}(\hat{x}, \hat{\gamma}) \leq v_{wc}(x, \hat{\gamma})$ for all $x
\in X$, so $\hat{x} \in X_{rob}(\hat{\gamma})$.

\medskip\noindent\textbf{Step 5 ($\hat{x}$ is feasible and optimal for $\mathrm{SAT}$).}
Since $(x^{\varepsilon_k}, \alpha^{\varepsilon_k})$ is feasible for $\mathrm{SAT}^{\varepsilon_k}$,
\[
    v_{wc}^{\varepsilon_k}(x^{\varepsilon_k}, \gamma)
    \;\leq\; f_0 + \alpha^{\varepsilon_k}\gamma
    \quad \forall\, \gamma \in [0, \gamma_{\max}].
\]
Fix $\gamma \in [0, \gamma_{\max}]$. Adding and subtracting $v_{wc}(x^{\varepsilon_k}, \gamma)$,
\begin{align*}
v_{wc}(\hat{x}, \gamma)
&\leq
\underbrace{\bigl|v_{wc}(\hat{x}, \gamma) - v_{wc}(x^{\varepsilon_k}, \gamma)\bigr|}_{\to\, 0 \text{ by Step 2}}
+
\underbrace{\bigl|v_{wc}(x^{\varepsilon_k}, \gamma) - v_{wc}^{\varepsilon_k}(x^{\varepsilon_k}, \gamma)\bigr|}_{\leq\, \varepsilon_k M\, \to\, 0}
+
v_{wc}^{\varepsilon_k}(x^{\varepsilon_k}, \gamma) \\
&\leq
\bigl|v_{wc}(\hat{x}, \gamma) - v_{wc}(x^{\varepsilon_k}, \gamma)\bigr|
+ \varepsilon_k M
+ f_0 + \alpha^{\varepsilon_k}\gamma.
\end{align*}
Passing to the limit as $k \to \infty$ gives $v_{wc}(\hat{x}, \gamma) \leq f_0 + \hat{\alpha}\gamma$
for all $\gamma \in [0, \gamma_{\max}]$, so $(\hat{x}, \hat{\alpha})$ is feasible for
$\mathrm{SAT}$. Since $\hat{\alpha} \leq \alpha_{sat}$ by Step~1 and $(\hat{x}, \hat{\alpha})$ is
feasible, optimality of $\alpha_{sat}$ gives $\hat{\alpha} = \alpha_{sat}$, so $\hat{x}$ is an
optimal solution of $\mathrm{SAT}$.
\end{proof}

The following examples illustrate that the assumptions in Proposition~\ref{prop:sat-eq-rob} and Theorem~\ref{prop:general-convex} are critical.

\begin{examplebox}
  \label{sec:satisficing-counterexample}
  Let here $X=\{ x^1= (0, 2),  x^2=(1,  0),  x^3=(0.5, 1.7)\}$ (nonconvex), $f(x, p) = p\tpose x$, $d(p_0, p)=\|p-p_0\|_2$, $P=\Re^2$ and $p_0=(1,0)$. The associated worst-case cost is
  \[
    v_{wc}(x, \gamma) = \max_{d(p_0, p)\leq \gamma} p^\top x = p_0^\top x + \gamma \| x\|_2.
  \]
  The worst-case costs are $v(x^1, \gamma)=2\gamma$, $v(x^2, \gamma)=1 + \gamma$ and $v(x^3, \gamma)\approx 0.5 + 1.77\gamma$. For $0\le \gamma \le 1$ solution $x^1$ is robust optimal, while for $1\le \gamma$ solution $x^2$ is robust optimal; see Figure~\ref{fig:counterexample}. In particular, note that $x^3$ is always dominated in terms of worst-case cost by $x^1$ or $x^2$.

  We consider a satisficing formulation
  \[
    \begin{array}{r@{~~}l}
      \min &  \alpha \\
      \st & \max_{d(p_0, p)\leq \gamma} p^\top x \le f_0 + \alpha \gamma \quad \forall \gamma \geq 0 \\
           & x\in X
    \end{array}
  \]
  with satisfaction level $f_0=1/2$.
  Note that solution $x^2$ is not feasible in the satisficing problem since $p_0\tpose x^2 = 1 > f_0$. The optimal $\alpha$ for $x^3$ is $\alpha^3=\norm{x^3}_2\approx 1.77$ while for $x^1$ it must hold $\alpha^1 = \norm{x^1}_2=2$. Hence $(x_{sat}, \alpha_{sat}) = (x^3, \norm{x^3})$ is the unique satisficing solution and yet $x^3$ is not robustly optimal for any $\gamma\geq 0$.
\end{examplebox}

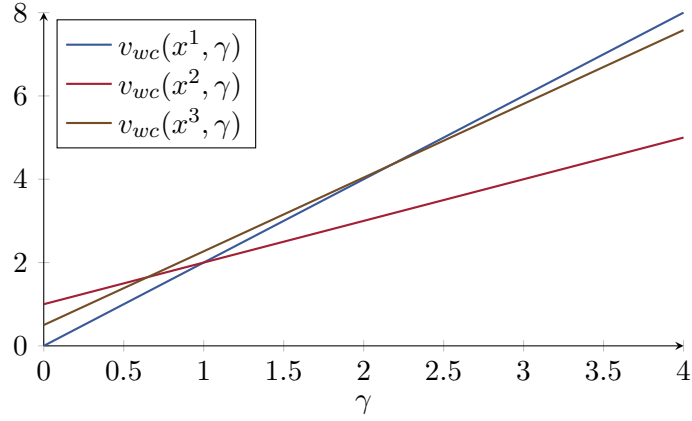
\begin{figure}[t]
    \centering
\begin{tikzpicture}
\begin{axis}[
  domain=0:4,
  samples=100,
  xlabel={$\gamma$},
  legend style={at={(0.02,0.98)},anchor=north west},
  axis lines=left,
  width=10cm,
  height=6cm
  ]

  \addplot+[no marks, thick] {2*x};
  \addlegendentry{$v_{wc}(x^1, \gamma)$}

  \addplot+[no marks, thick] {1 + x};
  \addlegendentry{$v_{wc}(x^2, \gamma)$}

  \addplot+[no marks, thick] {0.5 + 1.77*x};
  \addlegendentry{$v_{wc}(x^3, \gamma)$}

\end{axis}
\end{tikzpicture}
    \caption{Worst-case costs for the three decisions in the nonconvex counterexample.}
    \label{fig:counterexample}
\end{figure}

\begin{examplebox}
  \label{ex:sat-not-rob}
  Let here $X=[0, 2]$, $f(x, p) = (x-1)^2 + p$ (not strictly concave in $p$), $d(p_0, p)=|p-p_0|$, $P=\Re$ and $p_0=0$. The associated worst-case cost is
  \[
    v_{wc}(x, \gamma) = (x-1)^2 +\gamma
  \]
  where clearly we have $x_{rob}(\gamma)=1$ for all $\gamma\geq 0$ achieving oracle worst-case cost $v_{wc}^\star(\gamma) = \gamma$.
  We consider now the satisficing formulation
  \[
    \begin{array}{r@{~~}l}
      \min & \alpha \\
      \st &  v_{wc}(x, \gamma) \le f_0 + \alpha \gamma \quad \forall \gamma \geq 0 \\
           & x\in X
    \end{array}
  \]
  with satisfaction level $f_0=1/4>0$ with solution set $(X_{sat}, \alpha_{sat}) = (\set{x}{(x-1)^2\leq f_0},1)$ and hence there are satisficing solutions which are not robustly optimal.
\end{examplebox}

\begin{proposition}
  \label{prop:sat-weakly-minimal}
  Let $\Gamma$ be compact and $\gamma\mapsto v_{wc}(x,\gamma)$ be continuous for all $x\in X$. Suppose $(x_{sat}, \alpha_{sat})$ is optimal in~\eqref{eq:rob-grc}. Then $x_{sat}$ is weakly minimal in~\eqref{eq:multiobjective}.
\end{proposition}
\begin{proof}
  Suppose for contradiction that $x_{sat}$ is not weakly minimal. Then there exists $y\in X$ such that $A(y,\gamma)<A(x_{sat},\gamma)$ for all $\gamma\in\Gamma$. Adding $v^\star_{wc}(\gamma)$ to both sides gives
  \[
    v_{wc}(y,\gamma) < v_{wc}(x_{sat},\gamma) \leq f_0+\alpha_{sat}\gamma \quad\forall\gamma\in\Gamma,
  \]
  where the second inequality follows from the feasibility of $(x_{sat},\alpha_{sat})$ in~\eqref{eq:rob-grc}. Since $\Gamma$ is compact and $\gamma\mapsto f_0+\alpha_{sat}\gamma - v_{wc}(y,\gamma)$ is continuous and strictly positive on $\Gamma$, its minimum
  \[
    \delta \defn \min_{\gamma\in\Gamma}\bigl(f_0+\alpha_{sat}\gamma - v_{wc}(y,\gamma)\bigr) > 0
  \]
  is attained by Weierstrass' extreme value theorem. Rearranging gives
  \[
    v_{wc}(y,\gamma) \leq f_0+\alpha_{sat}\gamma - \delta \leq f_0+\Bigl(\alpha_{sat}-\tfrac{\delta}{\max\Gamma}\Bigr)\gamma \quad\forall\gamma\in\Gamma,
  \]
  so $(y,\,\alpha_{sat}-\delta/\max\Gamma)$ is feasible in~\eqref{eq:rob-grc}, contradicting the optimality of $(x_{sat},\alpha_{sat})$.

  The assumption that $\Gamma$ is compact is necessary: in Example~\ref{ex:sat-not-rob} with $\Gamma=\Re_+$, which is closed but not bounded (hence not compact), the adversarial regret simplifies to $A(x,\gamma)=(x-1)^2$, so the entire set of satisficing solutions $\set{x}{(x-1)^2\leq f_0}$ fails to be weakly minimal, with the sole exception $x_{sat}=1$.
\end{proof}

\section{Supporting Results on Algorithmic Aspects}
\label{sec:supporting-algorithmic}

\begin{lemma}
  \label{lemma:KL-lips}
  We have that $\gamma\mapsto v_{wc}(x, \gamma)$ is Lipschitz continuous with constant $$L=\sqrt{2}\left(\ell_{\max}-\ell_{\min}\right).$$
\end{lemma}
\begin{proof}
  Define $\psi(r) \defn v_{wc}(x, \sqrt{r})$ for $r \geq 0$ and observe that it is concave and non-decreasing with $\psi(0) = \E{\P_0}{\ell(x,\xi)}$ \citep{vanparys2021data}.
  Pinsker's inequality gives $\KL(\P_0, \P) \leq r \implies \|\P - \P_0\|_{TV} \leq \sqrt{r/2}$, where $\|\cdot\|_{TV}$ denotes the total variation norm, so
  \begin{equation}\label{eq:KL-pinsker}
    \psi(r) - \psi(0) \leq (\ell_{\max} - \ell_{\min})\sqrt{r/2}.
  \end{equation}
  For $0 \leq \gamma_1 < \gamma_2$, set $r_i = \gamma_i^2$.
  By concavity of $\psi$, the secant slope from $r_1$ to $r_2$ is at most the secant slope from $0$ to $r_2$:
  \[
    \psi(r_2) - \psi(r_1)
    \leq \frac{r_2 - r_1}{r_2}\big(\psi(r_2) - \psi(0)\big)
    \leq \frac{(\gamma_2^2 - \gamma_1^2)(\ell_{\max} - \ell_{\min})}{\sqrt{2}\,\gamma_2}.
  \]
  Using $\gamma_2^2 - \gamma_1^2 = (\gamma_2 + \gamma_1)(\gamma_2 - \gamma_1) \leq 2\gamma_2(\gamma_2 - \gamma_1)$ yields
  \begin{equation*}
    v_{wc}(x, \gamma_2) - v_{wc}(x, \gamma_1) \leq \sqrt{2}\,(\ell_{\max} - \ell_{\min}) (\gamma_2 - \gamma_1).
  \end{equation*}
\end{proof}

\section{Supporting Results on the Knapsack Experiment}

\begin{figure}
    \centering
    \includegraphics[width=0.8\linewidth]{ 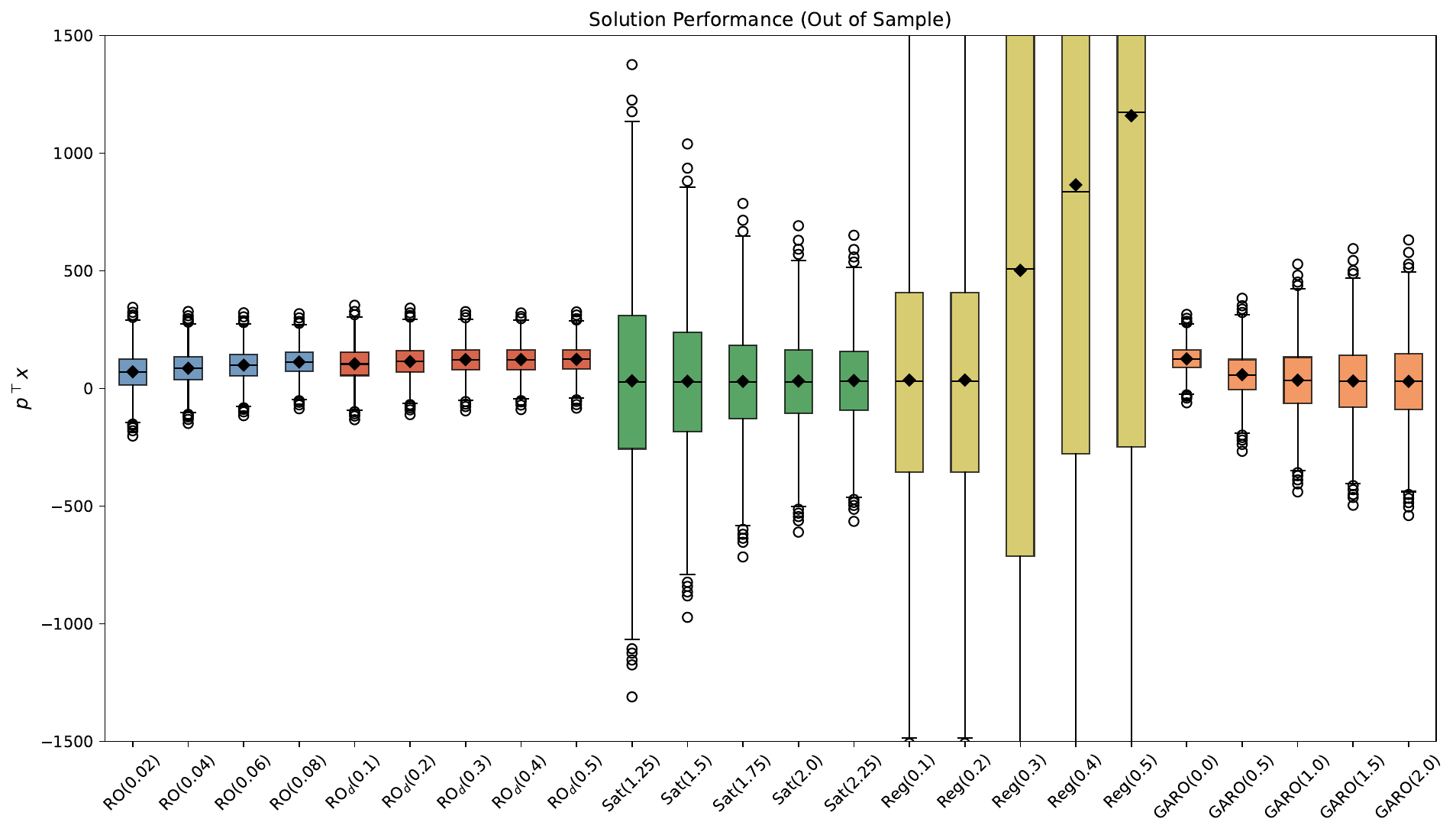}
    \caption{Boxplots for performance over all test scenarios for the knapsack problem with $n=50$ and Gaussian data with inverse mean-variance relationship. The diamonds denote the mean value.}
    \label{fig:gaussian_inverse_boxplots}
\end{figure}

\begin{figure}
    \centering
    \includegraphics[width=0.8\linewidth]{ 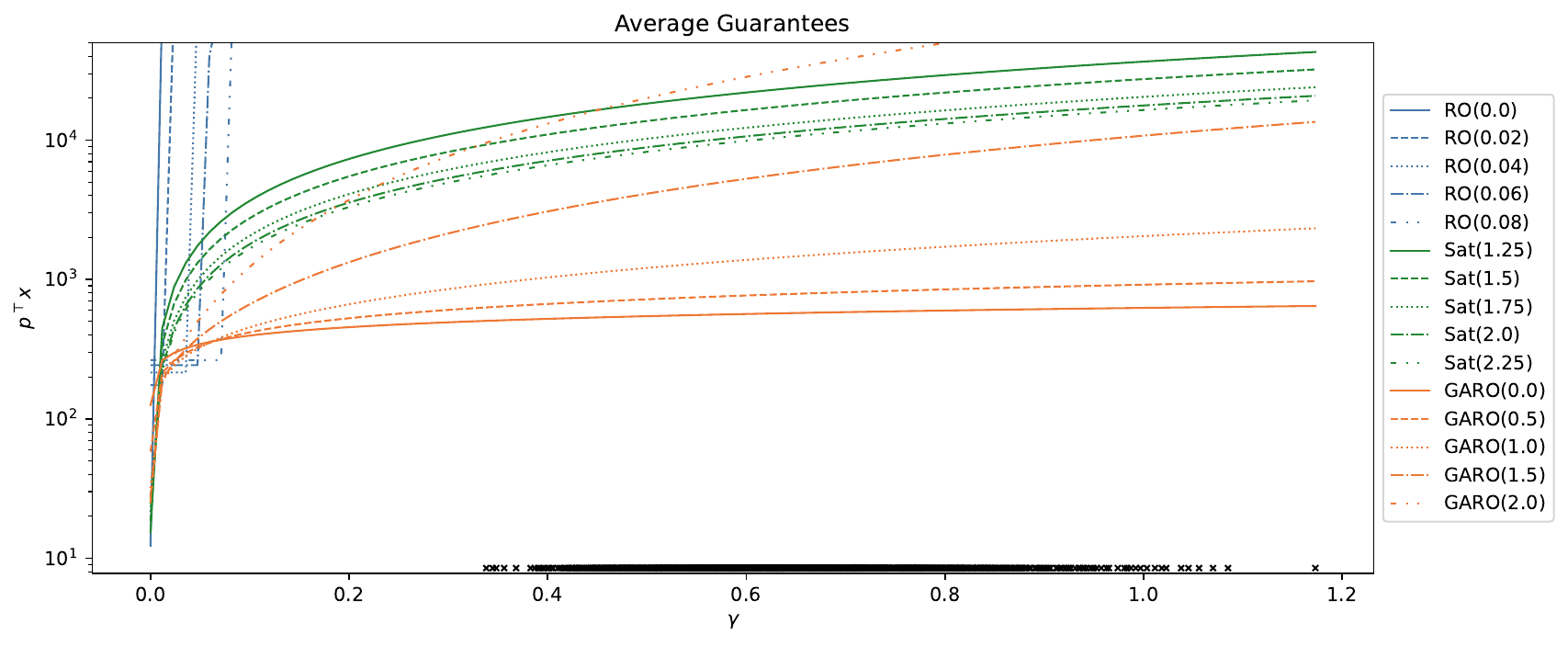}
    \caption{Performance guarantees for the minimum knapsack problem with $n=50$ for Gaussian data with inverse mean-variance relationship.}
    \label{fig:gaussian_inverse_guarantees}
\end{figure}

\end{document}